\documentclass[a4paper,12pt]{preprint}
\usepackage[margin=1.3in]{geometry}  % set the margins to 1in on all sides
\usepackage[full]{textcomp}

\usepackage{mathalfa}
\usepackage{ulem}
\usepackage{microtype}
\usepackage{url}
\usepackage[shortlabels]{enumitem}

\usepackage{booktabs}

\usepackage{enumitem}
\usepackage{amsmath}
\usepackage{amsfonts} 
\usepackage{amssymb}             % for blackboard bold, etc

\usepackage{comment}

\usepackage{mathrsfs}
\usepackage{booktabs}
\usepackage{mathtools}

\usepackage{tikz}
\usepackage{amsthm}                % better theorem environment
\usepackage{mhequ}
\usepackage{hyperref}

\usepackage{bm}

\hypersetup{pdfstartview=}

\addtocontents{toc}{\protect\hypersetup{hidelinks}}

\DeclareSymbolFont{sfoperators}{OT1}{ptm}{m}{n}
\DeclareSymbolFontAlphabet{\mathsf}{sfoperators}

\makeatletter
\def\operator@font{\mathgroup\symsfoperators}
\makeatother

\newcommand{\triplenorm}[1]{%
  \left\vert\!\left\vert\!\left\vert
  #1
  \right\vert\!\right\vert\!\right\vert}

\numberwithin{equation}{section}

\newtheorem{thm}{Theorem}[section]
\newtheorem{defn}[thm]{Definition}
\newtheorem{lem}[thm]{Lemma}
\newtheorem{prop}[thm]{Proposition}

\newtheorem{assumption}[thm]{Assumption}
\newtheorem{postulation}[thm]{Postulation}
\theoremstyle{remark}

\newtheorem{rmk}[thm]{Remark}

\DeclareMathOperator{\id}{id}

\makeatletter
\def\th@newremark{\th@remark\thm@headfont{\bfseries}}

\def\bdiamond{\mathop{\mathpalette\bdi@mond\relax}}
\newcommand\bdi@mond[2]{%
	\vcenter{\hbox{\m@th
			\scalebox{\ifx#1\displaystyle 2.6\else1.8\fi}{$#1\diamond$}%
	}}%
}

\def\bDiamond{\mathop{\mathpalette\bDi@mond\relax}}
\newcommand\bDi@mond[2]{%
	\vcenter{\hbox{\m@th
			\scalebox{\ifx#1\displaystyle 2.6\else1.2\fi}{$#1\Diamond$}%
	}}%
}
\makeatletter

\definecolor{darkgreen}{rgb}{0.1,0.7,0.1}
\definecolor{darkred}{rgb}{0.7,0.1,0.1}
\definecolor{darkblue}{rgb}{0,0,0.7}
\newcommand{\NN}{\mathbb{N}}
\newcommand{\RR}{\mathbb{R}}      % for Real numbers
\renewcommand{\SS}{\mathbb{S}}

\newcommand{\WW}{\mathbb{W}}
\newcommand{\XX}{\mathbb{X}}

\newcommand{\aA}{\mathcal{A}}
\newcommand{\bB}{\mathcal{B}}
\newcommand{\cC}{\mathcal{C}}
\newcommand{\dD}{\mathcal{D}}
\newcommand{\eE}{\mathcal{E}}
\newcommand{\fF}{\mathcal{F}}
\newcommand{\gG}{\mathcal{G}}
\newcommand{\hH}{\mathcal{H}}

\newcommand{\lL}{\mathcal{L}}
\newcommand{\mM}{\mathcal{M}}
\newcommand{\nN}{\mathcal{N}}
\newcommand{\oO}{\mathcal{O}}

\newcommand{\rR}{\mathcal{R}}
\newcommand{\sS}{\mathcal{S}}
\newcommand{\tT}{\mathcal{T}}
\newcommand{\uU}{\mathcal{U}}
\newcommand{\vV}{\mathcal{V}}

\newcommand{\xX}{\mathcal{X}}
\newcommand{\yY}{\mathcal{Y}}

\makeatletter % Guarantees same font as sin, cos, etc. Don't know why they don't show up in Times though...
\newcommand{\pa}{{\operator@font pa}}
\newcommand{\cov}{{\operator@font cov}}
\renewcommand{\div}{{\operator@font div}}
\newcommand{\var}{{\operator@font var}}
\newcommand{\corr}{{\operator@font corr}}
\newcommand{\diam}{{\operator@font diam}}
\newcommand{\Av}{{\operator@font Av}}
\newcommand{\trig}{{\operator@font trig}}
\newcommand{\Sym}{{\operator@font Sym}}
\newcommand{\Enh}{{\operator@font Enh}}
\newcommand{\EEnh}{\overline {\operator@font Enh}}
\makeatother

\newcommand{\W}{\mathbf{W}}
\newcommand{\X}{\mathbf{X}}

\newcommand{\n}{\mathbf{n}}

\newcommand{\emL}{\ensuremath{\mathscr{L}}}

\newcommand{\emC}{\ensuremath{\mathscr{C}}}
\newcommand{\sft}{\mathsf{T}}

\newcommand{\eps}{\varepsilon}
\newcommand{\md}{\mathrm{d}}
\renewcommand{\d}{\partial}

\newcommand{\dist}{\operatorname{dist}}

\newcommand{\bracket}[1]{\langle #1 \rangle}

\usepackage{amsmath}
\allowdisplaybreaks[4]

\title{Generalised stochastic curvature flow in $d \geq 2$, and sharp interface limit for the stochastic Allen-Cahn equation with nonlinear diffusion}

\author{Weijun Xu$^1$ \quad Shuhan Zhou$^2$}
\institute{Institute for Theoretical Sciences, Westlake University, China.\\\email{xuweijun@westlake.edu.cn}
\and School of Mathematical Sciences, Peking University, China.\\\email{zhoushuhan@stu.pku.edu.cn
}
}

\begin{document}
\maketitle

\begin{abstract}
We construct the local-in-time solution of the generalised anisotropic direction-dependent curvature flow in dimension $d \geq 2$ forced by a white-in-time and smooth-in-space Gaussian noise. This seems to be the first construction with a white-in-time noise which also allows spatial dependence, even in the simpler case of isotropic stochastic mean curvature flow. The main difficulty is that the stochastic PDE describing the flow has a multiplicative noise depending nonlinearly on both the solution and its gradient. The key technique is a transform developed in \cite{BKMZ20} based on rough characteristics that removes this rough multiplicative term. We also illustrate the relationship of this transform with previously known special situations. 

As an application of the construction, we show that in a short time interval, the sharp interface limit of the stochastic Allen-Cahn equation with nonlinear diffusion and the same noise (slightly smoothened in time) is given by the above direction-dependent stochastic curvature flow. 
\end{abstract}

\setcounter{tocdepth}{2}
% \microtypesetup{protrusio. n=false}
\tableofcontents
% \microtypesetup{protrusion=true}

\section{Introduction and main result}

Fix $d \geq 2$, and $\Omega \subset \RR^d$ be an open, connected domain with smooth boundary. The aim of the article is to study the sharp interface limit (as $\eps \rightarrow 0$) of the stochastic Allen-Cahn equation with nonlinear diffusion and Neumann boundary condition, given by
\begin{equation} \label{SAC}
\left\{
    \begin{aligned}
 \partial_t u_\eps &=\div \big( D(u_\eps) \nabla u_\eps \big) +\eps^{-2} f(u_\eps) +\eps^{-1} \xi_\eps\;, \quad &&(t,x) \in \RR^+\times \Omega\;,\\
\frac{\partial u_\eps}{\partial \nu} &=0\;, \quad &&(t,x) \in \RR^+\times \partial \Omega\;.\\
\end{aligned}\right.
\end{equation}
Throughout, we make the following assumptions on the diffusion matrix $D$, the reaction term $f$ and noise $\xi_\eps$. 

\begin{assumption} \label{ass:assumption_overall}
    \begin{enumerate}
        \item The diffusion matrix $D: \RR \rightarrow \RR^{d \times d}$ is symmetric and smooth. Furthermore, there exists $\Lambda > 1$ such that 
        \begin{equation*}
            \Lambda^{-1} \leq \eta^\sft \, D(s) \, \eta \leq \Lambda
        \end{equation*}
        for all $s \in \RR$ and $\eta \in \RR^d$ with $|\eta|=1$. 

        \item The reaction $f: \RR \rightarrow \RR$ is smooth. It has exactly three roots: $f(-1) = f(x_0) = f(1) = 0$ for some $x_0 \in (-1,1)$, and satisfy
        \begin{equation} \label{eq:cond_f_bistable}
        f'(\pm  1) < 0\;, \qquad f'(x_0) > 0\;.
        \end{equation}
        \item Equipotential condition: 
        \begin{align} \label{eq:cond_Df_equi}
        \int_{-1}^{1} D_{ij}(s) f(s) \,\md s = 0 \qquad \text{for every} \quad 1 \leq i,j \leq d\;.
        \end{align}
        \item The noise $\xi_\eps = \xi_\eps(t,x)$ has the form 
        \begin{equation} \label{eq:cond_noise_form}
        \xi_\eps(t,x) = \sum_{j=1}^m \sigma_j(x) \cdot \frac{\md W^{j;\eps}_t}{\md t}
        \end{equation}
        where $\sigma_j \in \cC_c^\infty (\Omega)$, and $W^{j;\eps}$ are smooth approximations of independent standard Brownian motions $W^{j}$'s such that for some $\kappa \in (0,\frac{1}{3})$, one has \footnote{This can be achieved by taking any $\kappa'<\kappa$ and letting $W^{j;\eps}$ be mollification of $W^j$ at scale $\eps^{\kappa'}$.}
        \begin{equation} \label{eq:cond_BM_approximation}
        \sup_{t \in [0,T]} \left|\frac{\md}{\md t} W_t^{j; \eps} \right| \lesssim \eps^{-\kappa}\;, \qquad \sup_{t \in [0,T]} \left|\frac{\md^2}{\md t^2} W_t^{j; \eps} \right|\lesssim \eps^{-2 \kappa}
        \end{equation}
        almost surely. The proportionality constants depend on $T$ and realisation of $W$ but are independent of $\eps$. 
    \end{enumerate}
\end{assumption}

Several simplified versions of \eqref{SAC} have been studied before. The case with $D = \id$ and $\xi_\eps \equiv 0$ corresponds to the standard Allen-Cahn equation. It is well known that it has a sharp interface limit given by the (deterministic) mean curvature flow (see for example \cite{ESS92}). 

The version of \eqref{SAC} with general nonlinear diffusion matrix $D$ (but still $\xi_\eps \equiv 0$) arises naturally as the hydrodynamical limit of a class of non-gradient type Glauber-Kawasaki dynamics (\cite{Funaki_particle}). It was shown in \cite{FP24} that it has a sharp interface limit given by an anisotropic direction-dependent curvature flow. 

The stochastic versions of \eqref{SAC} studied before (with $D=\id$) are mainly restricted to approximately white-in-time but space-independent noises (see \cite{Fun99, Web09}). In this case, its sharp interface limit is given by the standard mean curvature flow forced by an additive white-in-time noise. The extension from space-independent to space-dependent noise has been a challenge for a while (see the explanations of the difficulties in Section~\ref{sec:difficulty}). Our main theorem provides the first such extension in the case of a white-in-time and smooth-in-space noise. 

Before we state our main result, we first introduce a few quantities. They mainly arise from the deterministic part of \eqref{SAC}, and have been introduced in \cite[Section~1]{FP24}. For $e \in \RR^d\setminus\{0\}$, define
\begin{equation}\label{e:def_a_W_lambda}
    a_e(s) := e^\sft D(s) e\;, \quad W_e(s) = -2 \int_{-1}^s a_e(\tau) f(\tau) \,\md \tau\;, \quad \lambda(e) := \int_{-1}^{1} \sqrt{W_e(s)}\,\md s\;.
\end{equation}
Note that the equipotential condition \eqref{eq:cond_Df_equi} ensures $W_e(s) \geq 0$ for $s \in [-1,1]$. We then define $b, \mu_{ij}: \RR^{d}\setminus\{0\} \rightarrow \RR$ by
\begin{equation} \label{e:mu_c_lambda}
    \begin{split}
        b(e) &:= \frac{1}{\lambda(e)} \int_{-1}^{1} a_e(s) \,\md s\;,\\
        \mu_{ij}(e) &:= \dfrac{1}{\lambda(e)} \int_{-1}^{1} \left[ D_{ij}(s) \sqrt{W_e(s)} -\dfrac{\partial_{e_i} (W_e(s))}{2} \d_{e_j} \left(
        \dfrac{ a_e(s)}{\sqrt{W_e(s)}} \right) \right]\,\md s\;.
    \end{split}
\end{equation}
We are now ready to state our main theorem. 

\begin{thm} \label{thm:main}
Consider the problem \eqref{SAC} with the noise $\xi_\eps$ as above. Let $\vV_0 \Subset \Omega$ with smooth boundary $\d \vV_0$. Then for the initial data $u_\eps(0,\cdot)$ specified in Section~\ref{sec:convergence_smcf}, almost surely, there exists a random time $\tau>0$ and a randomly evolving closed hypersurface $(\Sigma_t)_{t \in [0,\tau]}$ such that the followings hold:
\begin{itemize}
\item[(i)] $\Sigma_0 = \d \vV_0$, and $(\Sigma_t)_{0 \leq t \leq \tau}$ evolve according to a direction-dependent stochastic curvature flow, given by
\begin{equation} \label{eq:smcf_surface}
    V\,\md t= - \sum_{i,j = 1}^d \mu_{ij}(\n) \, \d_{x_i} n_j \,\md t  - b(\n) \ \bm{\sigma}(x)\,\md \W\;.
\end{equation}
Here $\W = (W, \WW)$ is the Stratonovich enhancement of the standard Brownian motion $W$, $V$ is the outward normal velocity of $\Sigma_t$, and $\n = (n_j)_j$ is the outward normal vector to $\Sigma_t$ (extended constantly along normal directions). 

\item[(ii)] There exists $\theta>0$ such that 
\begin{equation*}
    \sup_{t \in [0,\tau]} \|u_\eps(t,\cdot) - \chi_{\Sigma_t}\|_{\lL^2(\Omega)} \lesssim \eps^\theta\;.
\end{equation*}
Here, $\chi_{\Sigma}$ denotes the function taking the value $-1$ in the interior and $+1$ outside $\Sigma$. 
\end{itemize}
\end{thm}

Part (i) of the theorem concerns local well-posedness of \eqref{eq:smcf_surface} and its precise meaning. These will be covered in Theorem~\ref{thm:well_posedness_smcf} and Proposition~\ref{prop:interpretation}, including a rigorous SPDE characterisation of \eqref{eq:smcf_surface} as well as its interpretation as a ``stochastic curvature flow".  Part (ii) of the theorem will be proved in Section~\ref{sec:convergence_smcf}. 

To the best of our knowledge, this seems to be the first time that a stochastic mean curvature flow in dimension $d \geq 2$ with white-in-time and space dependent noise is constructed. The main difficulty lies in the local well-posedness of \eqref{eq:smcf_surface} when $\W$ arises from standard Brownian motions and $\bm\sigma$ is non-constant. We refer to more discussions in Section~\ref{sec:difficulty}.

\subsection*{Notation}

When a function depends on various different variables (for example, $x$, $u$, $p$, $q$, etc.), we use $D_x$, $D_u$, $D_p$, $D_q$ to denote the Jacobians with respect to corresponding variables. In situations when these variables are themselves functions of the spatial variable $x \in \RR^d$, and that we view the original function as a function of $x \in \RR^d$, then we use $\nabla$ to denote the gradient with respect to the spatial variable (and $\nabla^\sft$ to denote its transpose). 

For a function $F = F(t,x,v, \nabla v, \nabla^2 v)$ with $v$ being real-valued spacetime function, we write $\fF$ as the functional 
\begin{equation*}
    (\fF v)(t,x) := F\big( t, x, v(t,x), \nabla v(t,x), \nabla^2 v(t,x) \big)\;.
\end{equation*}
We use the same convention for other such functions, including $G$ (with $\gG$ being the functional) and $H$ (with $\hH$ being the functional).

We say an open set $\oO_0 \Subset \oO$ if $\overline{\oO_0} \subset \oO$. For spacetime function $g$ on $[0,T] \times \oO$ with $\oO$ being an open bounded subset of $\RR^d$, we say $g \in \lL_t^\infty \cC_{x,loc}^\gamma ([0,T] \times \oO)$ if for every $\oO_0 \Subset \oO$, we have $\|g\|_{\lL_t^\infty \cC_x^\gamma ([0,T] \times \overline{\oO_0})} < +\infty$. Also, we say $g \in \lL_{x,loc}^\infty \cC_t^\gamma$ if for every $\oO_0 \Subset \oO$, we have $\|g\|_{\lL_x^\infty \cC_t^\gamma ([0,T] \times \overline{\oO_0})} < +\infty$. 

For $\gamma \in (0,1)$, we define the parabolic norm $\cC_{\pa}^\gamma$
\begin{equation*}
    \|g\|_{\cC_\pa^\gamma} := \|g\|_{\lL_t^\infty \cC_x^\gamma} + \|g\|_{\cC_t^{\gamma/2} \lL_x^\infty}\;.
\end{equation*}
We also need a parabolic norm with exponent between $2$ and $3$. In this case, we define (for $\gamma \in (0,1)$)
\begin{equation*}
    \|g\|_{\cC_\pa^{2+\gamma}} := \|\d_t g\|_{\cC_\pa^\gamma} + \|\nabla^2 g\|_{\cC_\pa^\gamma} + \|g\|_{\lL_{t,x}^\infty}\;.
\end{equation*}

\subsection*{Organisation of the article}

The rest of the article is organised as follows. In Section~\ref{sec:difficulty}, we give a description of the formal evolution \eqref{eq:smcf_surface} in terms of the signed distance function, and put it in the general context of rough PDEs. We then explain the main difficulties arising from the type of noises and our main strategy. Statements regarding its local well-posedness and stability in the general RPDE context are also given. 

In Section~\ref{sec:derivation_well_posedness_smcf}, we return to the original phase field evolution \eqref{SAC}. Starting with an ansatz for proper sub- and super-solutions to \eqref{SAC}, we derive the equations for the signed distance functions associated with them. Then, based on the statements about general RPDEs in Section~\ref{sec:difficulty}, we prove convergence of these signed distance functions to the limiting one. 

In Section~\ref{sec:sub_sup_convergence}, we rigorously construct the sub- and super-solutions to \eqref{SAC} from the signed distance functions constructed in Section~\ref{sec:derivation_well_posedness_smcf}, and apply the comparison principle to prove Theorem~\ref{thm:main}. 

Finally in Section~\ref{sec:RPDE_stability_proof}, we give detailed proofs of the stability of RPDEs stated in Section~\ref{sec:difficulty}.

\subsection*{Acknowledgement}

We are very grateful to Tadahisa Funaki and Jiajun Tong for inspiring and helpful discussions throughout this project. In particular, Tadahisa Funaki suggested that the techniques developed in an earlier draft version of this article (initially intended to deal with standard Allen-Cahn with $D = \id$) might be adapted to the nonlinear diffusion case \eqref{SAC}. Jiajun Tong pointed out to us the expression of $\pi_t(x)$ in \eqref{eq:foot_point}, which leads to the derivation of the forced mean curvature flow \eqref{eq:mcf_deterministic}. 

W. Xu was supported by Ministry of Science and Technology via the National Key R\&D Program of China (no.2023YFA1010100) and National Science Foundation China via the Key Programme Grant (no.12595280 and 12595281).

\section{Main strategy and statements about general rough PDEs}
\label{sec:difficulty}

\subsection{Formal derivation of the PDE and the main difficulty}

To illustrate the main difficulty, we begin with a simplified version of \eqref{SAC}: setting the diffusion matrix $D = \id$, and also replacing the noise $\eps^{-1} \xi_\eps$ by $\eps^{-1} g$ with a smooth spacetime function $g$. It is well known (see for example \cite{BSS93}) that in this case, the transition layer of $u_\eps$ converges to a hypersurface $\Sigma_t$ evolving according to a mean curvature flow with smooth forcing, given by 
\begin{equation} \label{eq:mcf_surface}
    V = -\div_{\Sigma_t} \n - c_0 \, g\;.
\end{equation}
Here $c_0$ is an explicit constant depending on $f$. The above evolution has a parametrisation by signed distance function as follows. Let
\[
\rho(t,x):=
\begin{cases}
-\mathrm{dist}(x,\Sigma_t)\;, & x \ \text{inside } \Sigma_t\;,\\[1mm]
\ \ \mathrm{dist}(x,\Sigma_t)\;, & x \ \text{outside } \Sigma_t\;.
\end{cases}
\]
In the case without external forcing, it is well known that the evolution of signed distance function associated with \eqref{eq:mcf_surface} (with $g=0$) is given by the nonlinear parabolic equation
\begin{equation*}
    \d_t \rho = \operatorname{tr}\! \big(\nabla^2\rho \, (\id-\rho \nabla^2\rho )^{-1} \big)\;,
\end{equation*}
and that $|\nabla \rho| = 1$, at least in a sufficiently small neighbourhood of $\Sigma_t$. To see what the evolution of $\rho$ should be with a general smooth spacetime function $g$, we first note that for every point $x$ also in a sufficiently small neighbourhood of $\Sigma_t$, there is a unique point $\pi_t(x) \in \Sigma_t$ such that $|\rho(t,x)| = |x-\pi_t(x)|$, and that one has the expression
\begin{equation} \label{eq:foot_point}
    \pi_t(x):=x-\rho(t,x)\nabla\rho(t,x)\in\Sigma_t\;.
\end{equation}
Along the line segment joining $x$ and $\pi_t(x)$ (which is in the normal direction), the signed distance varies exactly with the normal motion of the interface, which gives
\[
\partial_t\rho(t,x)=-\,V\bigl(t,\pi_t(x)\bigr)\;.
\]
Accordingly, the point which the external forcing $g$ acts on the surface $\Sigma_t$ should also be $\pi_t(x)$. This leads us to the equation
\begin{equation} \label{eq:mcf_deterministic}
    \partial_t\rho = \operatorname{tr} \! \big(\nabla^2\rho \, (\id-\rho \nabla^2\rho )^{-1} \big) \;+\; c_0 \, g \bigl(t,\,x-\rho \nabla\rho \bigr)\;.
\end{equation}
In particular, the forcing is evaluated at the point $\pi_t(x) = x - \rho(t,x) \nabla \rho(t,x) \in \Sigma_t$ rather than at $x$ itself.

In our situation where the smooth spacetime function $g$ is replaced by the noise $\xi_\eps$ in \eqref{eq:cond_noise_form} (for the moment still with $D = \id$ for simplicity), one expects the sharp interface limit to be given by mean curvature flow with stochastic forcing. The corresponding parametrisation with signed distance function is given by
\begin{equation} \label{eq:smcf_model}
    \d_t \rho = \operatorname{tr}\! \big(\nabla^2\rho \, (\id-\rho \nabla^2\rho )^{-1} \big) + c_0 \sum_{j=1}^{m} \sigma_j (x - \rho \nabla \rho) \cdot \frac{\md W^j_t}{\md t}\;.
\end{equation}
The main issue comes from the $\nabla \rho$ dependence inside the function $\sigma$, which makes the above equation formally critical. Indeed, if one counts parabolic spacetime regularity -- the white noise in time is formally $\cC^{-1-}$ (in terms of parabolic metric), the nonlinear heat operator increases the regularity of the solution by $2$ (hence one expects $\rho \in \cC^{1-}$), and then $\nabla \rho \in \cC^{0-}$. 

This places the equation beyond the ``subcritical'' framework of \cite{HN24}, which deals with quasi-linear rough evolution equations that are just ``better" than \eqref{eq:smcf_model} (in terms of either stronger smoothing, or more regular noise). The candidate for our limiting evolution \eqref{eq:smcf_model} just falls out of that scope. Furthermore, it is a fully nonlinear equation. 

Note that in the special situation of space-independent noise, $\sigma_j$ are constant functions, and hence the above mentioned issues do not appear.

\subsection{Main strategy}
\label{sec:RPDE_thm}

The equation \eqref{eq:smcf_model} can be put in a more general context of rough PDEs. We first introduce the concept of rough paths. 

\begin{defn} \label{defn:rp}
    For $\alpha \in (\frac{1}{3}, \frac{1}{2}]$, an $\alpha$-H\"older rough path over $\RR^m$ is a pair $\X = (X, \XX)$ with $X: [0,T] \rightarrow \RR^m$ and $\XX: \Delta_{[0,T]} \rightarrow \RR^{m \times m}$ such that
    \begin{itemize}
        \item Chen's identity: for every $0 \leq s \leq r \leq t \leq T$, one has
        \begin{equation*}
            \XX_{s,t} = \XX_{s,r} + \XX_{r,t} + (X_r - X_s) \otimes (X_t - X_r)
        \end{equation*}
        \item We have
        \begin{equation*}
            \|X\|_\alpha := \sup_{0 \leq s < t \leq T} \frac{|X_t - X_s|}{|t-s|^\alpha} < +\infty\;, \quad \|\XX\|_{2 \alpha} := \sup_{0 \leq s < t \leq T} \frac{|\XX_{s,t}|}{|t-s|^{2\alpha}} < +\infty\;.
        \end{equation*}
    \end{itemize}
\end{defn}

For an $\alpha$-H\"older rough path $\X$, we define its $\alpha$-H\"older rough path norm by
\begin{equation*}
    \|\X\|_{\emC^\alpha} := \|X\|_\alpha + \|\XX\|_{2\alpha}^{1/2}\;.
\end{equation*}
An $\alpha$-H\"older rough path is called geometric if it belongs to the closure of the canonical lifts of smooth paths in the $\alpha$-H\"older rough path topology. We denote the resulting space of geometric $\alpha$-H\"older rough paths by $\emC_g^\alpha([0,T],\RR^m)$. We compare $\X$ and $\widetilde{\X}$ in $\emC_g^\alpha$ by
\begin{equation*}
    d_\alpha (\X, \widetilde{\X}) := \|X - \widetilde{X}\|_\alpha + \|\XX - \widetilde{\XX}\|_{2\alpha}^{1/2}\;.
\end{equation*}
In our situation, we consider the rough path arising from Stratonovich enhancement of the Brownian motion. More precisely, let
\begin{equation*}
    \WW^\eps_{s,t} := \int_{s}^{t} (W^\eps_r - W^\eps_s) \otimes \md W^\eps_r\;, \quad \WW_{s,t} := \int_s^t (W_r - W_s) \circ_\otimes \md W_r\;,
\end{equation*}
where the latter is in Stratonovich sense. Let
\begin{equation*}
    \W^\eps := (W^\eps, \WW^\eps)\;, \qquad \W := (W, \WW)\;.
\end{equation*}
The following proposition is well known.

\begin{prop} \label{prop:BM_rough_paths}
    For every $\alpha \in (\frac{1}{3},\frac{1}{2})$, $\W^\eps$ and $\W$ belong to $\emC_g^\alpha ([0,T]; \RR^m)$. Furthermore, there exists $\theta>0$ such that
    \begin{equation*}
        d_\alpha (\W^\eps, \W) \lesssim \eps^\theta\;.
    \end{equation*}
\end{prop} 

We next introduce the concept of controlled rough paths. 

\begin{defn} \label{defn:controlled_rp}
    Let $E$ be a finite-dimensional normed space. A pair $(Y,Y')$ belongs to $\dD_W^{2\alpha}([0,T];E)$ if
\[
    Y\in\cC^\alpha([0,T];E)\;,\quad
    Y'\in\cC^\alpha([0,T];\emL(\RR^m,E))\;,
\]
and the remainder
\[
    R^Y_{s,t}:= Y_t - Y_s - Y'_s (W_t - W_s)
\]
satisfies $\|R^Y\|_{2\alpha}<\infty$. We equip this space with the norm
\[
    \|(Y,Y')\|_{\dD_W^{2\alpha}}
    :=|Y_0|+|Y'_0|+\|Y'\|_\alpha+\|R^Y\|_{2\alpha}\;.
\]
Finally, for an open set $\oO\subset\RR^d$, the notation
\[
    (Y,Y')\in\lL_{x,loc}^\infty
    \dD_W^{2\alpha}([0,T];E)
\]
means that $(Y(\cdot,x),Y'(\cdot,x))\in \dD_W^{2\alpha}([0,T];E)$ for every $x\in \oO$ and, for every compact set $K\subset \oO$,
\[
    \sup_{x\in K}
    \big\|(Y(\cdot,x),Y'(\cdot,x))\big\|_{\dD_W^{2\alpha}}
    <\infty\;.
\]
\end{defn}

Consider the general RPDE of the form
\begin{equation} \label{e:RPDE}
\left\{
\begin{aligned}
    &\md \rho = F(t,x,\rho, \nabla \rho, \nabla^2 \rho) \, \md t +  H(x, \rho, \nabla \rho) \, \md \W\;,\quad &&t\in[0,T]\;,x\in\overline{\oO_t}\;,\\
    &G (x, \rho, \nabla \rho) = 0\;, &&t\in[0,T]\;, x \in \d\oO_t\;,\\
    &\rho(0,\cdot) = \rho^0\;,&& x\in\overline{\oO}_0\;,
\end{aligned}\right.
\end{equation}
where $\oO_t$ is a subset of $\RR^d$ depending on $t$, and $\W=(W,\WW)\in\emC_g^\alpha([0,T],\RR^m)$ is an $\alpha$-H\"older rough path with $\alpha\in(\frac13,\frac12)$. Throughout this subsection we regard all vectors, including $x,p,\nabla\rho,H$ and $W$, as column vectors. In the notation above, $\hH(\rho)\,\md\W$ denotes the scalar rough integral obtained by pairing $\hH(\rho)$ with $\md\W$; we keep this conventional notation and do not write an additional transpose in \eqref{e:RPDE}. We first specify what we mean by a solution to \eqref{e:RPDE}.

\begin{defn} \label{def:RPDE}
    A function $\rho$ is called a solution to the RPDE \eqref{e:RPDE} if the following conditions hold:
    \begin{enumerate}
        \item $\nabla^k \rho \in \cC_\pa^{2\alpha}$ for $0 \leq k \leq 2$. And $\rho$ satisfies the initial and boundary condition such that $\rho(0,x)=\rho^0(x)$ holds for all $x\in\overline{\oO}_0$, and $G[\rho]=0$ holds for all $t\in[0,T]$, $x\in\partial\oO_t$.

        \item $\rho$ and $\nabla \rho$ are both controlled by $W$; more precisely, there exists $\rho'$ and $(\nabla \rho)'$ such that 
        \begin{equation*}
            (\rho, \rho') \in \lL_{x,loc}^\infty \dD_W^{2\alpha}([0,T]; \RR)\;, \qquad \big(\nabla \rho, (\nabla \rho)' \big) \in \lL_{x,loc}^\infty \dD_W^{2\alpha}([0,T]; \RR^d)\;.
        \end{equation*}
        Furthermore, for all $0\leq s\leq t\leq T$ and $x\in\cap_{s\leq r\leq t}\oO_r$, we have 
        \begin{equation*}
            \rho(t,x) = \rho(s,x) + \int_s^t (\fF\rho)(r,x)\,\md r + \int_s^t (\hH\rho)(r,x)\,\md\W_r\;,
        \end{equation*}
        where the second integral is the rough integral. 
    \end{enumerate}
\end{defn}

To solve \eqref{e:RPDE}, inspired by \cite[Chapter~12.2.3]{FH20}, we aim to construct a family of transformations $(\sS_t)_{t\geq0}$ such that the $\md\W$-term is eliminated in the PDE satisfied by 
\[
    v(t):=\sS_t\big(\rho(t)\big)\;.
\]
Differentiating in $t$ on both sides above, we formally get
\[
\md v =\md\sS_t(\rho)+\bracket{ \dD \sS_t(\rho),\md\rho} =\md\sS_t(\rho)+\bracket{ \dD \sS_t(\rho),\fF(\rho)}\,\md t +\bracket{\dD \sS_t(\rho),\hH(\rho)}\,\md\W\;,
\]
where $\dD$ denotes the Fr\'echet derivative. Thus, we would like to find $\sS_t$ such that 
\begin{equation}\label{e:transformation_cancellation}
    \md\sS_t(\rho)+\bracket{\dD \sS_t(\rho),\hH(\rho)}\,\md\W=0\;,
\end{equation}
so that the PDE for $v$ reduces to
\[
    \d_t v = \bracket{\dD\sS_t(\rho),\fF(\rho)} = \bracket{(\dD\sS_t)(\sS_t^{-1}(v)),\fF(\sS_t^{-1}(v))}\;,
\]
which can then be treated using classical theory.

Equation \eqref{e:transformation_cancellation} is an infinite-dimensional first order RPDE. Solving it is naturally related to the characteristic equation
\[
    \md\rho=\hH(\rho)\,\md\W\;.
\]
If $\tT_t$ denotes the solution flow of this characteristic equation, then formally $\md(\sS_t(\tT_t))=0$, and consequently $\sS_t=\tT_t^{-1}$. For this first order equation, the associated finite-dimensional characteristic system is
\begin{equation}\label{e:formal_characteristic}
        \begin{aligned}
        \md X &= -(D_pH)^\sft(X,U,P)\,\md \W\;,\\
        \md U &= \left(H(X,U,P)-D_pH(X,U,P)P\right)^\sft\,\md \W\;,\\
        \md P &= \left((D_xH)^\sft(X,U,P)+P (D_uH)^\sft(X,U,P)\right)\,\md \W\;.
    \end{aligned}
\end{equation}
Here $X,P\in\RR^d$ and $U\in\RR$. The second equation is scalar because $H-(D_pH) P\in\RR^m$, while the third equation is $\RR^d$-valued because $(D_xH)^\sft+P (D_uH)^\sft\in\RR^{d\times m}$.

However, rigorously pursuing this route requires computing the Fr\'echet derivative of the nonlinear flow $\tT_t$, which turns out to be rather cumbersome. To overcome this difficulty, \cite{BKMZ20} proposes a more efficient approach to derive the PDE for $v$. We adopt their method in what follows.

\subsection{Stability of the rough PDEs}

Since we will derive stochastic mean curvature flow from \eqref{SAC}, we start with a smooth-in-time noise and recover white-in-time noise in the $\eps \rightarrow 0$ limit. Hence, we consider the $\eps$-dependent system
\begin{equation*}
\left\{
\begin{aligned}
    &\d_t \rho_\eps = F_\eps \big(t,x,\rho_\eps, \nabla \rho_\eps, \nabla^2 \rho_\eps \big) + H (x, \rho_\eps, \nabla \rho_\eps) \dot{\W}^\eps\;,\quad &&t\in[0,T]\;,x\in\overline{\oO_t^\eps}\;,\\
    &G_\eps (x, \rho_\eps, \nabla \rho_\eps) = 0\;, &&t\in[0,T]\;, x \in \d\oO_t^\eps\;,\\
    &\rho_\eps(0,\cdot) = \rho^0\;,&& x\in\overline{\oO}_0\;.
\end{aligned}\right.
\end{equation*}
Throughout, we let
\begin{equation*}
    \Theta^\eps_t (\bm \theta) := \big( X^\eps_t (\bm\theta), \, U^\eps_t (\bm\theta), \, P^\eps_t (\bm\theta) \big)\;, \quad \bm\theta = (x,u,p) \in \RR^d \times \RR \times \RR^d
\end{equation*}
denote the solution to the RDE
\begin{equation} \label{e:characteristic}
        \begin{aligned}
        \md X^\eps &= -(D_p H)^\sft(X^\eps,U^\eps,P^\eps)\,\md \W^\eps\;, && X^\eps_0 = x\;,\\
        \md U^\eps &= \left(H(X^\eps,U^\eps,P^\eps)-D_p H(X^\eps,U^\eps,P^\eps) P^\eps \right)^\sft \,\md \W^\eps\;, && U^\eps_0 = u\;,\\
        \md P^\eps &= \left((D_x H)^\sft(X^\eps,U^\eps,P^\eps) + P^\eps (D_u H)^\sft(X^\eps, U^\eps, P^\eps) \right)\, \md \W^\eps\;, && P^\eps_0 = p\;,
    \end{aligned}
\end{equation}
where the $\eps=0$ case should be understood as the limiting RDE \eqref{e:formal_characteristic} and RPDE \eqref{e:RPDE}. Define $Q^\eps_t: \RR^d \times \RR \times \RR^d \times \RR^{d \times d} \rightarrow \RR^{d \times d}$ by
\begin{equation} \label{eq:Q_matrix_def}
    Q^\eps_t (\bm\theta, q) := \big( D_x P^\eps_t + D_u P^\eps_t \cdot p^\sft + D_p P^\eps_t \cdot q \big) \cdot \big( D_x X^\eps_t + D_u X^\eps_t \cdot p^\sft + D_p X^\eps_t \cdot q \big)^{-1}\;,
\end{equation}
where the Jacobian matrices of $P_t^\eps$ and $X_t^\eps$ are evaluated at $\bm\theta$. For $F_\eps: \RR^+ \times \RR^d \times \RR \times \RR^d \times \RR^{d \times d} \rightarrow \RR$ and $G_\eps: \RR^d \times \RR \times \RR^d \rightarrow \RR$, define $\widetilde{F}_\eps$ and $\widetilde{G}_\eps$ by (with $\bm\theta = (x,u,p)$)
\begin{equation} \label{eq:transforms_FG_tilde}
    \begin{split}
    \widetilde{F}_\eps (t,\bm\theta,q) &:= F_\eps \big( t, \,\Theta_t^\eps(\bm\theta), \,  Q^\eps_t (\bm\theta, q) \big) \cdot \exp \Big( - \int_0^t (D_u H)^\sft \big( \Theta_s^\eps(\bm\theta) \big) \,\md \W^\eps_s \Big)\;,\\
    \widetilde{G}_\eps (t, \bm\theta) &:= G_\eps \big( \Theta_t^\eps (\bm\theta) \big)\;.
    \end{split}
\end{equation}
We have the following crucial lemma. 

\begin{lem}\label{lem:transform}
Let $\widetilde{\oO}\subset\RR^d$ be a bounded, connected Lipschitz domain. Suppose $\widetilde{F}_\eps$ and $\widetilde{G}_\eps$ are transforms of $F_\eps$ and $G_\eps$ as in \eqref{eq:transforms_FG_tilde}. Suppose that the equation
    \begin{equation}\label{e:RPDE_transformed}
    \left\{
    \begin{aligned}
    &\partial_t v_\eps = \widetilde{F}_\eps(t,x,v_\eps,\nabla v_\eps,\nabla^2 v_\eps)\;, \quad &&t\in[0,T]\;, \; x \in \overline{ \widetilde{\oO}}\;,\\
    &\widetilde{G}_\eps(t,x,v_\eps,\nabla v_\eps) = 0\;,\quad &&t\in[0,T]\;, \, x\in \partial \widetilde{\oO}\;,\\
    &v_\eps(0,\cdot) = \rho^0\;,\quad &&x\in\overline{\widetilde{\oO}}\;.
    \end{aligned}\right.
    \end{equation}
    has a classical solution $v_\eps$ on $[0,T] \times \widetilde{\oO}$. Let $\Theta^\eps_t(\bm\theta)$ be the solution to \eqref{e:characteristic}. Write
    \begin{equation*}
        \Phi^\eps_t (x) := \big( x, \, v_\eps(t,x), \, \nabla v_\eps (t,x) \big)\;, \qquad \widehat{\Theta}^\eps_t (x) := \Theta^\eps_t \big( \Phi^\eps_t (x) \big)\;,
    \end{equation*}
    and the same for components $X^\eps, U^\eps, P^\eps$ of $\Theta^\eps$. Then the followings assertions hold: 
    \begin{enumerate}
    \item For every $\eps \geq 0$, there exists $\tau_\eps>0$ depending on $\|v_\eps\|_{\lL_t^\infty \cC_x^2([0,T] \times \overline{\widetilde{\oO}})}$ such that for every $t \in [0, \tau_\eps]$, the map $\widehat{X}_t^\eps: \widetilde{\oO} \rightarrow \RR^d$ is injective, and hence has an inverse $(\widehat{X}_t^\eps)^{-1}$. 

    \item If in addition $v_\eps \in \lL_t^\infty \cC_{x,loc}^4$ and $\nabla^k v_\eps \in \lL_{x,loc}^\infty \cC_t^\alpha$ for $0 \leq k \leq 3$, then the function
    \begin{equation*}
        \rho_\eps (t,x) := \big( \widehat{U}^\eps_t \circ (\widehat{X}^\eps_t)^{-1} \big)(x)\;, \quad t \in [0, \tau_\eps]\;, \; x \in \oO_t^\eps := \big\{ \widehat{X}_t^\eps(y): \; y \in \widetilde{\oO} \big\}
    \end{equation*}
    is a solution to \eqref{e:RPDE} in the sense of Definition~\ref{def:RPDE} (and classical solution for $\eps>0$). Moreover, we have 
    \begin{equation*}
        \big(\nabla^k\rho_\eps, \nabla^k(\hH(\rho_\eps))\big)\in \lL_{x,loc}^\infty \dD_{W^\eps}^{2\alpha}\;
    \end{equation*}
    for $0 \leq k \leq 2$, and $\nabla^3\rho_\eps\in\lL_{x,loc}^\infty\cC_t^\alpha$. Furthermore, for all open sets $\oO',\oO$ satisfying $\oO'\Subset\oO\Subset\widetilde{\oO}$, we have
    \begin{equation*}
        \|\nabla^4\rho_\eps\|_{\lL^\infty([0,\tau_\eps]\times \overline{\oO'})}\lesssim_{\|\W^\eps\|_{\emC^\alpha},\,\|v_\eps\|_{\lL_t^\infty\cC_x^4([0,T]\times \overline{\oO})}} 1\;.
    \end{equation*}
    \item Suppose further that $\|v_\eps\|_{\lL_t^\infty \cC_x^2([0,T] \times \overline{\widetilde{\oO}})}$ are uniformly bounded in $\eps \in [0,1]$. Then one can choose $\tau>0$ independent of $\eps$ such that all of the above are true with $\tau_\eps$ replaced by $\tau$. Furthermore, for every open set $\oO \Subset \widetilde{\oO}$, by decreasing $\tau$ if necessary (but still positive), one can ensure that
    \begin{equation*}
        \oO \Subset \big\{ \widehat{X}_t^\eps(y): \; y \in \widetilde{\oO} \big\}\;,
    \end{equation*}
    and one has the bound
    \begin{equation} \label{eq:PDE_stability}
        \|\rho_\eps - \rho_0\|_{\lL_t^\infty \cC_x^2 ([0,\tau]\times\overline{\oO} )} \lesssim d_\alpha (\W^\eps, \W) + \|v_\eps - v_0\|_{\lL_t^\infty \cC_x^2([0,T] \times \overline{\widetilde{\oO}})}\;.
    \end{equation}
    \end{enumerate}
\end{lem}

\begin{rmk}
    The converse also holds: if $\rho$ solves the RPDE \eqref{e:RPDE}, then $v$ is the classical solution to the transformed PDE \eqref{e:RPDE_transformed}; see \cite[Theorem~4.5]{BKMZ20} for details.
\end{rmk}

In what follows, when a fixed reference domain $\widetilde{\oO}$ is prescribed, we say that the RPDE \eqref{e:RPDE} is induced by $\widetilde{\oO}$ if the transformed problem is posed on the fixed domain $\widetilde{\oO}$, and the physical domain in \eqref{e:RPDE} is given by
\[
    \oO_t=\{\widehat{X}_t(x):x\in\widetilde\oO\}\;.
\]

\begin{thm} \label{thm:stability}
Fix $\alpha\in\left(\frac{1}{3},\frac{1}{2}\right)$. Let $\widehat{\oO} \subset \RR^d$ be an open, connected, bounded domain with smooth boundary, and fix a domain $\widetilde{\oO} \Subset \widehat{\oO}$ with $\partial\widetilde{\oO}\in\cC^{2+2\alpha}$. Let $\dD_2 \subset \RR \times \RR^d$ and $\dD_3 \subset \RR \times \RR^d \times \RR^{d \times d}$ be open sets such that
\begin{equation*}
    \begin{split}
    \dD_2 &\supset \big\{ (\rho^0(x), \nabla \rho^0(x))\;, \; x \in \overline{\widetilde{\oO}} \big\}\;,\\
    \dD_3 &= \dD_2 \times \text{a neighbourhood of} \; \big\{ \nabla^2 \rho^0(x)\;, \; x \in \overline{\widetilde{\oO}} \big\}\;.
    \end{split}
\end{equation*}
Consider the RPDE \eqref{e:RPDE} induced by $\widetilde{\oO}$, where $F_\eps$, $G_\eps$, $H$, $\W^\eps$ and $\rho^0$ satisfy the following conditions.
\begin{enumerate}
\item  $F_\eps$ is uniformly elliptic on $[0,T] \times \widehat{\oO} \times \dD_3$ in the sense that
\begin{equation}\label{e:condition_F_elliptic_RPDE}
\sum_{i,j=1}^d \partial_{q_{ij}}F_\eps(t,x,u,p,q)\,\eta_i\eta_j \geq \lambda |\eta|^2
\end{equation}
for some $\lambda > 0$ independent of $\eps$, $(t,x,u,p,q)$ in the above region and $\eta \in \RR^d$. 

\item $F_\eps$ is four times differentiable in $(x,u,p,q)$ with norm 
\begin{equation}\label{e:condition_F_RPDE}
    \triplenorm{F_\eps}_{\mathcal F,\mathrm{RPDE}}:=
    \sup_{\substack{x\in\overline{\widehat{\oO}}, (u,p,q)\in \dD_3\\ |\beta|\leq 4}}
\bigl\|D_{(x,u,p,q)}^\beta F_\eps(\cdot\,,x,u,p,q)\bigr\|_{\cC^{\alpha}([0,T])}<+\infty\;.
\end{equation}

\item $G_\eps$ is three times differentiable in $(u,p)$. Furthermore, each derivative $D_{(u,p)}^\beta G_\eps$ with $|\beta|\leq3$ is $\cC^{1+2\alpha}$ in $x$ with norm
\begin{equation}\label{e:condition_G_RPDE}
\triplenorm{G_\eps}_{\mathcal G}:=
\sup_{\substack{(u,p)\in \dD_2\\ |\beta|\leq 3}}
\bigl\|D_{(u,p)}^\beta G_\eps(\cdot\,,u,p)\bigr\|_{\cC^{1+2\alpha}(\overline{\widehat{\oO}})} <+\infty\;.
\end{equation}
\item 
    \(H\in\cC^7_b(\overline{\widehat{\oO}}\times\RR\times\RR^d)\). Moreover, there exists $L_\eps\in\cC_b^3 \bigl(\overline{\widehat{\oO}}\times\dD_2;\RR^{1\times m}\bigr)$ such that 
    \begin{equation} \label{e:condition_GH_factorization}
        -(D_x G_\eps)(D_p H)^\sft+(D_u G_\eps)\big(H-(D_p H)p\big)^\sft
        +(D_p G_\eps)\big((D_x H)^\sft+p(D_u H)^\sft\big)
        =G_\eps L_\eps
    \end{equation}
    holds for all
    \((x,u,p)\in\overline{\widehat{\oO}}\times\dD_2\).
\item $\rho^0\in \cC^{4+2\alpha}(\overline{\widetilde\oO})$ satisfies the compatibility condition
\begin{equation*}
G_\eps\bigl(x,\rho^0(x),\nabla \rho^0(x)\bigr)=0\;,\quad x\in\partial\widetilde\oO\;,
\end{equation*}
and the uniform non–tangentiality condition: there exists $c>0$ such that
\begin{equation} \label{e:condition_G_non_tangent_RPDE}
\inf_{x \in \d \widetilde{\oO},\eps\in[0,1]} \Big| \sum_{i=1}^d \d_{p_i} G_\eps \big( x,\rho^0(x),\nabla \rho^0(x) \big) \nu_i(x) \Big| \geq c\;,
\end{equation}
where $\nu(x)$ denotes the unit outward normal on $\partial\widetilde\oO$. 
\end{enumerate}
Then for every $\eps$, there exists $\tau_\eps > 0$ such that there is a unique solution to \eqref{e:RPDE} induced by $\widetilde{\oO}$. 

If furthermore $|\!|\!| F_\eps |\!|\!|$ and $|\!|\!| G_\eps |\!|\!|$  are all uniformly bounded in $\eps$, then one can choose the above existence time $\tau_\eps$ independent of $\eps$. Moreover, for every open set $\oO \Subset \widetilde{\oO}$, there exists $\tau>0$ independent of $\eps$ such that 
\begin{equation*}
    \|\rho_\eps-\rho_0\|_{\cC^{0,2} ([0,\tau]\times\overline{\oO})} \lesssim \triplenorm{F_\eps-F_0}_{\mathcal F,\mathrm{RPDE}} +\triplenorm{G_\eps-G_0}_{\mathcal G} +d_\alpha(\W^\eps,\W)\;.
\end{equation*}
\end{thm}

\section{Derivation of the stochastic curvature flow}
\label{sec:derivation_well_posedness_smcf}

In this section, starting from an ansatz for the solution $u_\eps$ in terms of the signed-distance function $\rho_\eps$, we derive an equation for $\rho_\eps$. This ansatz is based on the modified profile $\uU$ specified in Lemma~\ref{lem:CHL} below. It is inspired by but different from the ones in \cite{FP24, Web09}. 

\subsection{Linearization and perturbation of the profile equation}

Let \(e\in\RR^d\setminus\{0\}\). Recall the quantity $a_e$ defined in \eqref{e:def_a_W_lambda}. Let $\overline{\uU}(\cdot,e)$ be the unique increasing solution to
\begin{equation} \label{e:def_U0}
\big(a_e(\overline{\uU}) \,\overline{\uU}_z\big)_z + f(\overline{\uU})=0\;, \quad \overline{\uU}(0)=0\;, \quad \overline{\uU}(\pm\infty)=\pm1\;.
\end{equation}
We refer to \cite[equations~(20) and~(21)]{EFHPS22} and \cite[equation~(2.4)]{FP24} for existence, uniqueness, and properties of the solution to the above equation. Set
\begin{equation} \label{eq:defn_Y}
Y(z,e):= a_e \big(\overline{\uU}(z,e) \big) \, \overline{\uU}_z(z,e)\;.
\end{equation}
For $\varphi=\varphi(z)$, define the (one-dimensional) linear operator
\begin{equation*}
\aA_e\varphi:= - \Big( \d_{z}^2 \big( a_e(\overline{\uU}) \, \varphi \big) + f'(\overline{\uU})\varphi \Big)\;.
\end{equation*}
Its $\lL^2$-adjoint $\aA_e^*$ is given by
\begin{equation*}
\aA_e^*\psi = - \Big( a_e(\overline{\uU}) \psi'' + f'(\overline{\uU}) \psi \Big)\;.
\end{equation*}
Then we have
\begin{equation*}
\aA_e \,\overline{\uU}_z = 0\;, \quad \aA_e^*Y = 0\;.
\end{equation*}

\begin{lem}\label{lem:profile_refined_asymptotics}
Let \(e\in\RR^d\setminus\{0\}\). Recall \(\overline{\uU}(\cdot,e)\) denotes the unique increasing solution to \eqref{e:def_U0}. Set
\begin{equation} \label{eq:defn_gamma}
    \gamma_\pm(e):=\sqrt{\frac{-f'(\pm1)}{a_e(\pm1)}}\;.
\end{equation}
Then there exist positive constants \(\kappa_\pm(e)\) such that, for all \(z\geq0\), we have
\begin{equation} \label{eq:Ubar_asymptotics}
    \begin{split}
    1\mp \,\overline{\uU}(\pm z,e) &= \kappa_\pm(e) e^{-\gamma_\pm(e)z} + \oO \big( e^{-2\gamma_\pm(e)z} \big)\;,\\
    \overline{\uU}_z(\pm z,e) &= \gamma_\pm(e)\kappa_\pm(e) e^{-\gamma_\pm(e)z} + \oO\big( e^{-2\gamma_\pm(e)z} \big)\;.
    \end{split}
\end{equation}
\end{lem}
\begin{proof}
We prove the claims for $+ z$; the argument for $-z$ is identical after replacing \(1\) by \(-1\). Since $e \in \RR^{d} \setminus \{0\}$ is fixed, we omit the notation $e$ in various quantities (such as $\overline{\uU}$, $\gamma_+$, $\kappa_+$, etc.). Recall the definition of $W_e$ in \eqref{e:def_a_W_lambda}. Differentiating the quantity \((a_e(\overline{\uU})\,\overline{\uU}_z)^2 - W_e(\overline{\uU})\) in \(z\) and using the profile equation \eqref{e:def_U0}, we see this quantity is identically constant. Furthermore, it goes to $0$ as $z \rightarrow -\infty$. Hence, this constant is zero. This gives the first-order identity
\begin{equation}\label{e:first_order_identity}
    Y(z) = a_e(\overline{\uU}(z)) \, \overline{\uU}_z(z)=\sqrt{W_e \big( \overline{\uU}(z) \big)}\;.
\end{equation}
Using that $\overline{\uU}$ is increasing and that $\overline\uU(0) = 0$, we get
\begin{equation} \label{eq:z_inverse_expression}
    z=\int_0^{\overline{\uU}(z)}\frac{a_e(s)}{\sqrt{W_e(s)}}\,\md s\;.
\end{equation}
Since \(W_e'(1)=-2a_e(1)f(1) = 0\), \(W_e''(1) = -2a_e(1)f'(1)>0\), and $W_e(1) = 0$ by the equipotential condition \eqref{eq:cond_Df_equi}, Taylor expanding $W_e$ near $1$ gives
\begin{equation}\label{e:taylor_We}
    W_e(s)= - a_e(1) f'(1) (1-s)^2 + \oO \big( (1-s)^3 \big)\;,
\end{equation}
Therefore, we have
\begin{equation*}
    \frac{a_e(s)}{\sqrt{W_e(s)}} = \frac{1}{\gamma_+(e)(1-s)}+\oO(1)\;,
\end{equation*}
where $\gamma_+ (e)$ is given as in \eqref{eq:defn_gamma}. Integrating the above quantity from $0$ to $\overline{\uU}(z)$ and plugging back to \eqref{eq:z_inverse_expression}, we deduce there exists a constant \(b_+\) such that
\[
    z=-\frac{1}{\gamma_+}
    \log\bigl(1-\overline{\uU}(z)\bigr)+b_+ +\oO\bigl(1-\overline{\uU}(z)\bigr)\;.
\]
Multiplying $-\gamma_+$ on both sides above and then taking the exponential, we get
\begin{equation*}
    1-\overline{\uU}(z)
    = \kappa_+e^{-\gamma_+z}
    \exp\!\left(\oO\bigl(1-\overline{\uU}(z)\bigr)\right) = \Big( 1 + \oO \big(1 - \overline{\uU}(z) \big) \Big) \kappa_+e^{-\gamma_+z}
\end{equation*}
with $\kappa_+ := e^{\gamma_+ b_+}$. This proves the asymptotics for $1 - \overline{\uU}(z)$ in \eqref{eq:Ubar_asymptotics}. 

The asymptotics for $\overline{\uU}_z$ in \eqref{eq:Ubar_asymptotics} follow directly from the identity \eqref{e:first_order_identity} and the Taylor expansion \eqref{e:taylor_We}. The proof for $1 + \overline{\uU}(-z)$ and $\overline{\uU}_z (-z)$ for $z>0$ is the same. So we omit the details. 
\end{proof}

\begin{lem}\label{lem:linearized}
Let \(e\in\RR^d\setminus\{0\}\) and \(G(\cdot,e)\in\lL^\infty(\RR)\). Consider
\begin{align*}
-\aA_e\varphi = G(\cdot,e)\quad\text{in }\RR\;,
\quad
\varphi\in\lL^\infty(\RR)\;,
\quad
\varphi(0)=0\;.
\end{align*}
Then there exists a unique solution $\varphi$ if and only if $\bracket{G,Y}=0$. 

Let \(\gamma_\pm(e)\) be as in Lemma~\ref{lem:profile_refined_asymptotics}. Assume in addition that there exist constants \(G(\pm \infty, e)\) and \(\lambda\in 
\left(0, \min\left\{\gamma_+(e),\gamma_-(e)\right\}\right) \) such that
\begin{align}\label{e:decay_G}
|G(\pm z,e)-G(\pm \infty, e)|\lesssim e^{-\lambda z}\;,\quad z\ge 0\;.
\end{align}
Then \(\varphi(z,e)\) admits limits \(\varphi(\pm \infty, e)\) as \(z\to\pm\infty\), and we have
\begin{equation*}
|\varphi( \pm z,e)-\varphi(\pm \infty, e)| + |\varphi_{zz}(\pm z,e)| \lesssim e^{-\lambda z}\;, \quad z\ge 0\;.
\end{equation*}
\end{lem}
\begin{proof}
The Fredholm alternative follows from \cite[Lemma~2.1]{FP24} and scaling in $|e|$. So we focus on asymptotic properties of the solution. As before, we treat the case $+z$ in detail; the situation for $-z$ is essentially identical. Similar as before, we omit the notation $e$ in various quantities.

By \cite[Lemma~2.1]{FP24}, under \(\bracket{G,Y}=0\), the solution $\varphi$ has the expression
\begin{equation} \label{eq:expression_varphi}
\varphi(z) = \overline{\uU}'(z) \int_0^z \frac{1}{Y^2(\eta)} \left(\int_{-\infty}^{\eta} G(\zeta)Y(\zeta)\,\md \zeta\right)\,\md \eta\;.
\end{equation}
We first note that by Lemma~\ref{lem:profile_refined_asymptotics} and the definition of $Y$ in \eqref{eq:defn_Y}, for $\eta > 0$, we have
\begin{equation*}
    Y(\eta) = a_e \big( \overline{\uU}(\eta) \big)\, \overline{\uU}'(\eta) =a_e(1) \gamma_+ \, \kappa_+ \, e^{- \gamma_+ \eta} + \oO \big( e^{-2 \gamma_+ \eta} \big)\;.
\end{equation*}
Using orthogonality between $G$ and $Y$ and the decay of $G$ in \eqref{e:decay_G}, we have
\begin{align*}
\int_{-\infty}^{\eta} G(\zeta)Y(\zeta)\,\md \zeta = -\int_{\eta}^{+\infty} G(\zeta) Y(\zeta)\,\md \zeta = - G(+\infty) \int_{\eta}^{+\infty} Y(\zeta)\,\md \zeta +\oO \big( e^{-(\lambda + \gamma_+)\eta} \big)\;.
\end{align*}
By definition of $Y$ in \eqref{eq:defn_Y} and a change of variable together with Taylor expansion, we have
\begin{equation*}
\int_{\eta}^{+\infty} Y(\zeta) \,\md \zeta = \int_{\overline{\uU}(\eta)}^{1} a_e(s) \, \md s = a_e(1) \, \big( 1 - \overline{\uU}(\eta) \big) + \oO \big( (1 - \overline{\uU}(\eta) )^2 \big)\;.
\end{equation*}
Applying again Lemma~\ref{lem:profile_refined_asymptotics}, we get
\begin{equation*}
    \frac{1}{Y^2(\eta)} \cdot \int_{\eta}^{+\infty} Y(\zeta) \,\md \zeta = \frac{1}{a_e(1) \, \gamma_+^2 \, \kappa_+} \, e^{\gamma_+\eta} + \oO(1)\;.
\end{equation*}
Plugging all the above into the expression of $\varphi$, we obtain
\begin{align*}
\varphi(z) = -\frac{G(+\infty)}{a_e(1) \, \gamma_+^2 \, \kappa_+} \cdot \overline{\uU}'(z) \int_{0}^{z} e^{\gamma_+ \eta} \,\md \eta \, + \, \overline{\uU}'(z) \int_{0}^{z} \oO \Big( e^{(\gamma_+ - \lambda) \eta} \Big)\, \md \eta\;.
\end{align*} 
Using again Lemma~\ref{lem:profile_refined_asymptotics} for $\overline{\uU}'(z)$, we get
\begin{equation*}
    \big| \varphi(z) - \varphi(+\infty) \big| \lesssim e^{- \lambda z} 
\end{equation*}
with $\varphi(+\infty) = - \frac{G(+\infty)}{a_e(1) \, \gamma_+^2} = \frac{G(+\infty)}{f'(1)}$. 

To get the decay for derivatives of $\varphi$, we first note that by the definition of $\overline{\uU}$ in \eqref{e:def_U0} and Lemma~\ref{lem:profile_refined_asymptotics}, we have
\begin{equation}\label{e:expansion_U0''}
    \overline{\uU}''(z) = -\gamma_+^2 \, \kappa_+ \, e^{-\gamma_+z} +\oO \big( e^{-2\gamma_+ \,z} \big)\;.
\end{equation}
Then differentiating the expression \eqref{eq:expression_varphi} once and arguing as above, we obtain
\begin{equation} \label{e:decay_varphi'}
    |\varphi'(z)| \lesssim e^{-\lambda z}\;.
\end{equation}
Finally, the desired bound for $\varphi''$ can be obtained by expanding the equation $-\aA_e \varphi = G$ and employing the asymptotics for $G$, $\overline\uU$, $\varphi$ and $\varphi'$ above. 

This proves the desired bounds for $+z$. The corresponding claims for $-z$ can be dealt with in exactly the same way; hence we omit the arguments. This completes the proof of the lemma. 
\end{proof}

Since \(f(\pm1)=0\) and \(f'(\pm1)<0\), by the implicit function theorem, there exists \(\delta_*>0\) and unique smooth functions $U_\pm:(-\delta_*,\delta_*)  \rightarrow \RR$ such that
\begin{equation}\label{eq:root_shift_f}
    f\bigl(U_\pm(\delta)\bigr)+\delta=0\;, \quad U_\pm(0)=\pm1\;, 
\end{equation}
and \(f'(U_\pm(\delta))<0\) for every \(|\delta|<\delta_*\). 

\begin{defn} \label{defn:function_space_psi}
    Let $\mathbf{\Psi}$ be the set of functions $\psi: \RR \times (\RR^d \setminus \{0\})$ such that for every $R>1$ and every multi-index $\alpha$, there exist $\gamma_R > 0$ and $C(R, |\alpha|) > 0$ such that
    \begin{equation*}
        |(\d_e^\alpha \psi)(z,e)| \leq C(R,|\alpha|) \cdot e^{-\gamma_R |z|}\;,
    \end{equation*}
    uniformly over all $z \in \RR$ and all $e \in \RR^d$ with $R^{-1} < |e| < R$. 
\end{defn}

For $\lambda>0$, we further define the function space $\xX_\lambda$ by
\begin{equation*}
    \mathcal X_\lambda := \Bigl\{u\in\cC^2(\RR): \lim_{z\to\pm\infty} u(z)=:u(\pm \infty) \; \text{exist, and } \sup_{z\in\RR} \big( e^{\lambda|z|}|u''(z)| \big) < +\infty \Bigr\}
\end{equation*}
with norm
\begin{equation*}
    \|u\|_{\mathcal X_\lambda} := |u(+\infty)| + |u(-\infty)| + \sup_{z\in\RR} \big( e^{\lambda|z|}|u''(z)| \big)\;.
\end{equation*}
We have the following important lemma. 

\begin{lem}\label{lem:CHL}
Let $(\psi_{ij})_{i,j=1}^{d}$ be a collection of functions in $\mathbf{\Psi}$ given in Definition~\ref{defn:function_space_psi}. For every \(R>1\), there exists \(\delta_0>0\) and $\lambda>0$ such that for every $e \in \RR^d$ with $R^{-1} < |e| < R$, every $(q_{ij})_{i,j=1}^d$ and $\delta \in \RR$ with $|q| + |\delta| < \delta_0$, there is a unique pair
\begin{equation*}
    \big( \uU(\cdot, e, q, \delta), \, \mathfrak{c}(e, q, \delta) \big) \in \xX_\lambda \times \RR
\end{equation*}
satisfying
\begin{equation}\label{TW}
    \begin{cases}
        \big(a_e(\uU)\,\uU_z\big)_z + \mathfrak{c}(e,q,\delta)\,\uU_z + f(\uU) + \displaystyle\sum\limits_{i,j} q_{ij}\,\psi_{ij} + \delta = 0\;,\quad z\in\RR\;,\\[2pt]
        \uU(\pm\infty)=U_{\pm}(\delta)\;,\quad \uU(0)=0\;.
    \end{cases}
\end{equation}
Furthermore, the pair $(\uU, \mathfrak{c})$ is infinitely differentiable in $(e,q,\delta)$ in the above region, and for all $k \leq 2$, $k' \in \NN$ and all multi-indices $\alpha, \beta$, one has
\begin{equation}\label{e:CHL_decay}
\bigl|\partial_z^k\partial_e^{\alpha}\partial_{q}^{\beta}\partial_{\delta}^{k'}\uU(\pm z,e,q,\delta) - \partial_z^k\partial_e^{\alpha}\partial_{q}^{\beta}\partial_{\delta}^{k'} U_{\pm}(\delta)\bigr|
    \lesssim e^{-\lambda z}
\end{equation}
for all \(z\in\RR_+\), $e \in \RR^d$ with $R^{-1} < |e| < R$ and $|q| + |\delta| < \delta_0$. 
\end{lem}
\begin{proof}
Fix \(R>1\). Since \(\gamma_\pm(e)\) are positive and continuous on the compact annulus \(\{e:R^{-1}\leq |e|\leq R\}\), we can take
\begin{equation*}
    \lambda \in \Big(0,\gamma_R\wedge \inf_{R^{-1} < |e| < R} \gamma_\pm(e) \Big)\;.
\end{equation*}
In what follows, we fix the $R>1$ and the choice of $\lambda$ above. Write $\xX = \xX_\lambda$ for simplicity. Let $\yY$ be the Banach space
\begin{equation*}
    \mathcal Y:= \Bigl\{ g\in\cC(\RR): \lim_{z\to\pm\infty} g(z) =: g(\pm \infty) \; \text{exist, and} \; \big| g\big(\pm |z| \big) - g(\pm \infty) \big| \lesssim e^{-\lambda |z|} \Bigr\}
\end{equation*}
with norm
\begin{equation*}
    \|g\|_{\mathcal Y}:=|g(+\infty)| + |g(-\infty)| + \sup_{z\ge 0} e^{\lambda z} \bigl(|g(z)-g(+\infty)| + |g(-z)-g(-\infty)|\bigr)\;.
\end{equation*}
For $(u,\mathfrak{c}) \in \xX \times \RR$, define the maps $\fF: (\RR^d \times \RR^{d \times d} \times \RR) \times (\xX \times \RR) \rightarrow \yY$ and $\widetilde{\fF}: (\RR^d \times \RR^{d \times d} \times \RR) \times (\xX \times \RR) \rightarrow \yY \times \RR$ by
\begin{equation*}
        \fF(e,q,\delta, u, \mathfrak{c}) := \bigl(a_e(u)u'\bigr)' + \mathfrak{c}u' + f(u) +\sum_{i,j}q_{ij}\psi_{ij} + \delta\;,
\end{equation*}
and
\begin{equation*}
    \widetilde{\mathcal F}(e,q,\delta, u, \mathfrak{c}) := \bigl(\fF(e,q,\delta,u,\mathfrak{c}),u(0) \bigr)\;.
\end{equation*}
For every $e \in \RR^{d} \setminus \{0\}$, both $\fF$ and $\widetilde{\fF}$ are infinitely Fr\'echet differentiable in all variables in a neighbourhood of $(e, 0, 0, \overline{\uU}(\cdot,e),0)$. Fix $e_*$ with $R^{-1} \leq |e_*| \leq  R$. By definition, we have
\begin{equation*}
    \widetilde{\fF} \big(e_*,0,0, \overline{\uU}(\cdot,e_*),0 \big)= (0,0)\;.
\end{equation*}
The Fr\'echet derivative of $\widetilde{\fF}$ with respect to the variable \((u,\mathfrak{c})\) at the point $\big( e_*,0,0, \overline{\uU}(\cdot,e_*),0 \big)$, denoted by $\lL_{e_*} \in \emL(\mathcal X \times\RR,\mathcal Y\times\RR)$, is given by
\begin{equation*}
    \mathcal L_{e_*}(\varphi,h) =
    \bigl(-\aA_{e_*}\varphi + h \, \overline{\uU}_z(\cdot,e_*), \, \varphi(0)\bigr)\;.
\end{equation*}
We next verify that \(\mathcal L_{e_*}\) is an isomorphism. Suppose \((\varphi,h)\in\ker\mathcal L_{e_*}\). Then
\[
    -\aA_{e_*}\varphi + h \, \overline{\uU}_z(\cdot,e_*)=0\;,\qquad \varphi(0)=0\;.
\]
Testing against $Y(\cdot, e_*)$ on both sides above, using that $Y(\cdot, e_*) \in \ker(\aA_{e_*}^*)$ and the positivity of $a_e$, we must have $h=0$. Hence, by Lemma~\ref{lem:linearized}, we also have $\varphi = 0$. 

Conversely, fix \((g,\ell)\in\mathcal Y\times\RR\). Let
\begin{equation*}
    h:= \frac{\bracket{g, \,Y(\cdot,e_*)}}
    {\bracket{\overline{\uU}_z(\cdot,e_*), \, Y(\cdot,e_*)}}\;,\qquad
    \widetilde g:=g- h \, \overline{\uU}_z(\cdot,e_*)\;.
\end{equation*}
For the above $h$, the function $\widetilde{g}$ satisfies the assumption of Lemma~\ref{lem:linearized}. Then, let $\widetilde{\varphi}$ be the unique $\lL^\infty (\RR)$ solution to
\begin{equation*}
    -\aA_{e_*} \widetilde{\varphi} = \widetilde g\;,\qquad \widetilde{\varphi}(0)=0\;,
\end{equation*}
and one has $\widetilde{\varphi} \in \xX$. Define
\begin{equation*}
    \varphi:=\widetilde\varphi+
    \frac{\ell}{\overline{\uU}_z(0,e_*)}\,\overline{\uU}_z(\cdot,e_*)\;.
\end{equation*}
One can check directly that \(\mathcal L_{e_*}(\varphi,h)=(g,\ell)\). Hence \(\mathcal L_{e_*}\) is surjective, and therefore a topological isomorphism.

Hence, by the implicit function theorem, there exists a neighbourhood of $(e_*, 0, 0)$ such that for every $(e,q,\delta)$ in that neighbourhood, there exists a unique pair
\begin{equation} \label{eq:pair_uc_implicit}
    \big( \uU^{e_*} (\cdot, e,q,\delta), \; \mathfrak{c}^{e_*} (e,q,\delta) \big) \in \xX \times \RR
\end{equation}
such that $\widetilde{\fF}(e,q,\delta,\uU^{e_*}, \mathfrak{c}^{e_*}) = (0,0)$. Furthermore, by infinite differentiability of the map $\widetilde{\fF}$ in all its variables, the pair $(\uU^{e_*}, \mathfrak{c}^{e_*})$ is also infinitely differentiable in $(e,q,\delta)$ in that neighbourhood. 

We next explain the uniformity in \(e\). The above construction gives a local branch near each base point \(e_*\). Since the closed annulus \(\{e:R^{-1}\leq |e|\leq R\}\) is compact, there exist finitely many $e_1, \dots, e_n$ in the annulus, neighbourhoods $\bB_{\delta_i}(e_i)$ of $e_i$ that covers the annulus, and $\delta_0>0$ sufficiently small such that the pair $(\uU^i, \mathfrak{c}^i)$ is the above implicit function from $(\delta_i, \delta_0)$ neighbourhood of $(e_i, 0, 0)$ to $\xX \times \RR$. 

It remains to patch the local branches. It suffices to show that the implicit functions for different $i$ and $j$ agree on overlaps of the above neighbourhoods of $(e_i, 0, 0)$ and $(e_j, 0, 0)$. By further decreasing $\delta_0$, this follows immediately from local uniqueness of the implicit function. Hence, we conclude there is a $\delta_0 > 0$ such that there is a unique smooth mapping $(\uU, \mathfrak{c})$ from the open set
\begin{equation*}
    \left\{ (e,q,\delta): \; R^{-1} < |e| < R\;, |q| + |\delta| < \delta_0  \right\}
\end{equation*}
to $\xX \times \RR$. 

Finally, the bound \eqref{e:CHL_decay} follows from the equation \eqref{TW} and that $(e,q,\delta) \mapsto \uU(\cdot,e,q,\delta)$ is a \(\cC^\infty\) map with values in \(\mathcal X\). This completes the proof of the lemma. 
\end{proof}

\subsection{Derivation of the equation for $\rho_\eps^{\pm}$ from a rough ansatz}
\label{sec:formal_derivation}

In this subsection, we give a formal derivation of the equations for $\rho_\eps^{\pm}$, which are building blocks for the super- and sub-solutions to \eqref{SAC}. By Lemma~\ref{lem:CHL} (applied in the situation with $\psi_{ij}=0$) and with an abuse of notation, we deduce the functions
\begin{equation} \label{eq:choice_psi}
    \psi_{ij} := D_{ij}(\overline{\uU})\,\overline{\uU}_z + \bigl(\d_{e_i} a_e(\overline{\uU}) \, \overline{\uU}_{e_j} \bigr)_z
\end{equation}
belong to the space $\bm\Psi$ specified in Definition~\ref{defn:function_space_psi}, and hence satisfy the assumption of Lemma~\ref{lem:CHL}. In what follows, we let $(\uU, \mathfrak{c})$ be the pair specified in Lemma~\ref{lem:CHL} built from the $\psi_{ij}$ in \eqref{eq:choice_psi}. The reason for this choice will be clear later in \eqref{e:Nv_expression}. 

Recall $u_\eps$ is the solution to the equation \eqref{SAC}, and that $\uU = \uU(z,e,q,\delta)$ is as in Lemma~\ref{lem:CHL} with $\psi_{ij}$ as in \eqref{eq:choice_psi}. The starting point is that for proper quantity $``\rho_\eps"$, the solution $u_\eps$ to \eqref{SAC} is approximately
\begin{equation*}
    u_\eps \approx \uU \Big( \frac{\rho_\eps}{\eps}, \, \nabla \rho_\eps, \, \eps \nabla^2 \rho_\eps, \, \eps \xi_\eps \Big)\;.
\end{equation*}
To rigorously justify such an ansatz, we fix $\beta \in (1,2)$, and let 
\begin{equation} \label{eq:ansatz_u_pm_prelim}
    \widetilde{u}_\eps^{\pm} := \uU \Big( \frac{\rho_\eps^\pm}{\eps}, \, \nabla \rho_\eps^\pm, \, \eps \nabla^2 \rho_\eps^\pm, \, \eps \xi_\eps \pm \eps^\beta \Big)\;.
\end{equation}
In derivation of the equation of $\rho_\eps^{\pm}$, we first postulate the behaviour of certain quantities related to derivatives of $\rho_\eps^\pm$. These properties and bounds will be proven later in Theorem~\ref{thm:well_posedness_smcf}.

\begin{postulation} \label{pos:postulation}
We postulate that there exists an open neighbourhood of $\d \vV_0$ such that in that neighbourhood, we have 
\begin{equation} \label{eq:postulation}
    |\nabla \rho_\eps^\pm| \equiv 1\;, \quad \sup_{0\leq k\leq 4} \|\nabla^k \rho_\eps^\pm\|_{\lL_{t,x}^\infty} \lesssim 1\;, \quad \sup_{0\leq k\leq 2} \|\d_t \nabla^k \rho_\eps^\pm\|_{\lL_{t,x}^\infty} \lesssim \eps^{-\kappa}
\end{equation}
with all proportionality constants uniform in $\eps$. 
\end{postulation}

In accordance with the equation \eqref{SAC} for $u_\eps$, we define the operator $\nN_\eps$ by 
\begin{equation} \label{eq:operator_nonlinear}
    \nN_\eps v:= \d_t v - \div \big( D(v) \nabla v \big) - \eps^{-2} f(v) - \eps^{-1} \xi_\eps\;.
\end{equation}
The solution $u_\eps$ to \eqref{SAC} satisfies $\nN_\eps u_\eps = 0$. We now start to derive the equations for $\rho_\eps^\pm$ so that \eqref{eq:ansatz_u_pm_prelim} are approximate solutions with $\nN_\eps \widetilde{u}_\eps^\pm \approx 0$. 

Differentiating in $t$ on both sides of \eqref{eq:ansatz_u_pm_prelim} and using \eqref{eq:postulation} and that $\eps |\d_t \xi_\eps| \lesssim \eps^{1-2\kappa}$ (implied by Part 4 in Assumption~\ref{ass:assumption_overall}), one gets
\begin{equation*}
    \d_t \widetilde{u}_\eps^\pm = \frac{1}{\eps} \cdot \uU_z \, \d_t\rho_\eps^\pm + \uU_e \cdot \d_t (\nabla \rho_\eps^{\pm}) + \eps \Big( \uU_q \cdot \d_t \nabla^2 \rho_\eps^{\pm} + \uU_\delta\,\d_t \xi_\eps \Big) = \frac{1}{\eps} \cdot \uU_z \, \d_t\rho_\eps^\pm + \oO(\eps^{-\kappa})\;.
\end{equation*}
Differentiating in the $x_j$ variable, one gets
\begin{equation*}
    \d_{x_j} \widetilde{u}_\eps^\pm = \frac{1}{\eps} \cdot \uU_z\, \d_{x_j} \rho_\eps^\pm + \sum_{k = 1}^{d} \uU_{e_k} \,\d_{x_j, x_k}^2 \rho_\eps^\pm + \eps \Big( \uU_q \cdot \nabla^2 (\d_{x_j} \rho_\eps^{\pm}) + \uU_\delta \, \d_{x_j} \xi_\eps \Big)\;.
\end{equation*}
Multiplying $D_{ij}(\widetilde{u}_\eps^{\pm})$ on both sides and further differentiating in $x_i$, and noting that $\big( D_{ij}(\uU) \, \uU_z \big)_{e_k} = \big( D_{ij}(\uU) \, \uU_{e_k} \big)_{z}$, as well as using that $|\nabla^k \rho_\eps^\pm| \lesssim 1$ and $|\nabla^k \xi_\eps| \lesssim \eps^{-\kappa}$, one gets
\begin{equation*}
    \begin{split}
    \d_{x_i} \Bigl( D_{ij} (\widetilde{u}_\eps^\pm)\, &\d_{x_j} \widetilde{u}_\eps^\pm\Bigr) =
    \frac1{\eps^2} \Bigl(D_{ij}(\uU)\, \uU_z \Bigr)_z\, \d_{x_i} \rho_\eps^\pm \,\d_{x_j} \rho_\eps^\pm + \frac{1}{\eps} \biggl[ D_{ij}(\uU) \, \uU_z \,\d_{x_i, x_j}^2 \rho_\eps^\pm\\
    &+\sum_{k = 1}^d \Bigl(D_{ij}(\uU)\,\uU_{e_k}\Bigr)_z \Bigl(\d_{x_i} \rho_\eps^\pm \,\d_{x_j, x_k}^2 \rho_\eps^\pm +\d_{x_j} \rho_\eps^\pm \,\d_{x_i, x_k}^2 \rho_\eps^\pm \Bigr) \biggr] + \oO(\eps^{-\kappa})\;.
    \end{split}
\end{equation*}
Write $e:=\nabla\rho_\eps^\pm$. Then we have
\begin{equation*}
    \sum_{i,j}\Bigl(D_{ij}(\uU)\,\uU_z\Bigr)_z\, \d_{x_i} \rho_\eps^\pm \,\d_{x_j} \rho_\eps^\pm
    =
    \bigl(a_e(\uU)\,\uU_z\bigr)_z\;,
\end{equation*}
and
\begin{equation*}
    \sum_{i,j}\sum_{k = 1}^d
    \Bigl(D_{ij}(\uU)\,\uU_{e_k}\Bigr)_z
    \Bigl(\d_{x_i} \rho_\eps^\pm \,\d_{x_j, x_k}^2 \rho_\eps^\pm
          +\d_{x_j} \rho_\eps^\pm \,\d_{x_i, x_k}^2 \rho_\eps^\pm\Bigr) =
    \sum_{i,j} \Bigl(\d_{e_i} a_e(\uU)\,\uU_{e_j}\Bigr)_z\, \d_{x_i, x_j}^2 \rho_\eps^\pm\;.
\end{equation*}
Putting together the above, we obtain 
\begin{equation*}
    \begin{split}
    \nN_\eps (\widetilde{u}_\eps^\pm) = - &\frac{1}{\eps^2} \Bigl (\bigl(a_e(\uU)\,\uU_z \bigr)_z + f(\uU) \Bigr) +\frac{1}{\eps} \Bigl[ \uU_z\,\d_t\rho_\eps^\pm\\
    &- \sum_{i,j}\Bigl(D_{ij}(\uU) \,\uU_z + \bigl( (\d_{e_i} a_e)(\uU)\,\uU_{e_j} \bigr)_z \Bigr) \d_{x_i, x_j}^2 \rho_\eps^\pm - \xi_\eps \Bigr] + \oO(\eps^{-\kappa})\;.
    \end{split}
\end{equation*}
Now, using the equation for $\uU$ in \eqref{TW} with $e = \nabla \rho_\eps^\pm$, $q_{ij}:=\eps\,\d^2_{x_i, x_j} \rho_\eps^\pm$ and $\delta = \eps \xi_\eps \pm \eps^\beta$, we get
\begin{equation}\label{e:Nv_expression}
    \begin{split}
    \nN_\eps (\widetilde{u}_\eps^\pm &) = \frac{1}{\eps} \, \uU_z \Big( \d_t \rho_\eps^{\pm} + \frac{1}{\eps} \, \mathfrak{c} \big( \nabla \rho_\eps^{\pm}, \, \eps \nabla^2 \rho_\eps^\pm, \, \eps \xi_\eps \pm \eps^\beta \big) \Big) \pm \eps^{\beta-2}\\
    &+\frac{1}{\eps} \sum_{i,j} \Big[ \psi_{ij} - D_{ij}(\uU) \,\uU_z - \big( (\d_{e_i} a_e)(\uU) \,\uU_{e_j} \big)_z \Big] \d_{x_i,x_j}^2\rho_\eps^\pm + \oO(\eps^{-\kappa})\;.
    \end{split}
\end{equation}
By the choice of $\psi_{ij}$ in \eqref{eq:choice_psi} and noting that $\overline{\uU}(\cdot, e) = \uU(\cdot, e, 0, 0)$, together with that $\uU$ is twice continuously differentiable in $z$ and smooth in other variables, 
we see the term in the bracket on the last line above (without the $\frac{1}{\eps}$ factor) is of order $\oO(\eps^{1-\kappa})$. Hence, we obtain
\begin{equation*}
    \nN_\eps (\widetilde{u}_\eps^\pm) = \frac{1}{\eps} \, \uU_z \Big( \d_t \rho_\eps^{\pm} + \frac{1}{\eps} \, \mathfrak{c} \big( \nabla \rho_\eps^{\pm}, \, \eps \nabla^2 \rho_\eps^\pm, \, \eps \xi_\eps \pm \eps^\beta \big) \Big) \pm \eps^{\beta-2} + \oO(\eps^{-\kappa})\;.
\end{equation*}
At this stage, it looks natural to define $\rho_\eps^\pm$ such that
\begin{equation} \label{eq:rho_ideal}
    \d_t \rho_\eps^{\pm} + \frac{1}{\eps} \, \mathfrak{c} \big( \nabla \rho_\eps^{\pm}, \, \eps \nabla^2 \rho_\eps^\pm, \, \eps \xi_\eps \pm \eps^\beta \big) = 0\;.
\end{equation}
However, one problem with this definition is that the equation is not parabolic. Instead, in accordance with \eqref{eq:mcf_deterministic}, we define $\rho_\eps^\pm$ such that
\begin{equation} \label{eq:rho_real}
    \d_t \rho_\eps^\pm = -\frac{1}{\eps}\mathfrak{c}\Bigl(\nabla\rho_\eps^\pm\;, \eps \nabla^2 \rho_\eps^\pm (\id-\rho_\eps^\pm\nabla^2\rho_\eps^\pm)^{-1}\;,\eps\xi_\eps(t,x-\rho_\eps^\pm \nabla \rho_\eps^\pm) \pm \eps^\beta \Bigr)\;.
\end{equation}
Although \eqref{eq:rho_real} is still not parabolic, it can be made uniformly parabolic by adding an auxiliary term (see Theorem~\ref{thm:well_posedness_smcf} below). It is crucial here that the spatial dependence of the noise is now $x - \rho_\eps^\pm \nabla \rho_\eps^\pm$ for adding such an effective auxiliary term.

We will prove in Theorem~\ref{thm:well_posedness_smcf} below that the solution $\rho_\eps^\pm$ to \eqref{eq:rho_real} does satisfy the bounds \eqref{eq:postulation} in Postulation~\ref{pos:postulation}, and hence all the above derivations (with specified error terms) are valid. In particular, these bounds ensure that \eqref{eq:rho_ideal} is \textit{almost true} (see Proposition~\ref{prop:rho_approximate_effect} below). This in turn will enable us to establish an essential quantitative bound for $\nN_\eps (\widetilde{u}_\eps^\pm)$ to apply the comparison principle (see Lemma~\ref{lem:sub_sup}). 

Finally, since $\mathfrak{c}(e,0,0)=0$, it is natural to expect $\rho_\eps^\pm\to\rho$ as $\eps\to0$, where $\rho$ solves the limiting equation obtained by linearising $\mathfrak{c}$ in $q$ and $\delta$:
\begin{equation*}
    \d_t \rho
    =
    -\partial_q \mathfrak{c}(\nabla\rho,0,0)
    :
    \bigl(\nabla^2\rho\,(\id-\rho\nabla^2\rho)^{-1}\bigr)-\partial_\delta \mathfrak{c} (\nabla \rho,0,0) \xi(t,x-\rho \nabla\rho)\;.
\end{equation*}
The following lemma states that the coefficients $\mu_{ij}$ and $b$ in the stochastic mean curvature flow in \eqref{eq:smcf_surface} can be expressed in terms of derivatives of the function $\mathfrak{c}$.

\begin{lem} \label{lem:expressions_coefficients}
We have
    \begin{equation*}
        \partial_{q_{ij}}\mathfrak{c}(e,0,0) = -\mu_{ij}(e)\;, \quad \partial_\delta \mathfrak{c}(e,0,0) = -b(e)\;,
    \end{equation*}
    where $\mu_{ij}$ and $b$ are defined in \eqref{e:mu_c_lambda}.
\end{lem}
\begin{proof}
Differentiating \eqref{TW} in $q$ and $\delta$ respectively and taking $q=0$, $\delta =0 $ gives (noting that $\mathfrak{c}(e,0,0) = 0$)
    \begin{equation*}
        \big(a_e(\uU)\,\uU_z\big)_{z, q_{ij}} + (\d_{q_{ij}}\mathfrak{c})(e,0,0)\,\overline{\uU}_z + f'(\overline{\uU}) \,\uU_{q_{ij}} + \psi_{ij} = 0
    \end{equation*}
    for all $1\leq i,j\leq d$ and 
    \begin{equation*}
        \big(a_e(\uU)\,\uU_z\big)_{z, \delta} + (\d_\delta \mathfrak{c})(e,0,0)\,\overline{\uU}_z + f'(\overline{\uU})\,\uU_\delta +  1 = 0\;.
    \end{equation*}
    Noting that $\big(a_e(\uU)\,\uU_z\big)_{z, q_{ij}} = \d_z^2 \big(a_e(\uU)\,\uU_{q_{ij}}\big)$ (and the same with $q_{ij}$ replaced by $\delta$), we see the above two relations are equivalent to
    \begin{equation*}
        \aA_e \,\uU_{q_{ij}} = (\d_{q_{ij}}\mathfrak{c})(e,0,0) \, \overline{\uU}_z +  \psi_{ij}\;, \qquad \aA_e \,\uU_{\delta} = (\d_{\delta}\mathfrak{c})(e,0,0)\,\overline{\uU}_z +  1\;.
    \end{equation*}
    Pairing with $Y$ and using $\aA_e^* Y = 0$, we get
    \begin{equation*}
        (\d_{q_{ij}}\mathfrak{c})(e,0,0) = -\frac{\bracket{\psi_{ij}, Y}}{\bracket{\overline{\uU}_z, \, Y}}\;, \qquad (\d_\delta \mathfrak{c})(e,0,0) = -\frac{\int_{-1}^{1} a_e(s)\, \md s}{\bracket{\overline{\uU}_z, \, Y}}\;.
    \end{equation*}
    It remains to show the identities
    \begin{equation*}
        \begin{split}
        &\bracket{\overline{\uU}_z,Y} = \int_{-1}^1 \sqrt{W_e(s)}\,\md s\;, \quad \bracket{D_{ij}(\overline{\uU}) \, \overline{\uU}_z, \, Y} = \int_{-1}^{1} D_{ij}(s) \sqrt{W_e(s)} \, \md s\;,\\
    &\left\langle \Bigl((\d_{e_i} a_e)(\overline{\uU}) \, \overline{\uU}_{e_j} \Bigr)_z, \; Y \right\rangle = -\frac{1}{2} \int_{-1}^{1} 
    \left[ \d_{e_i} (W_e(s)) \cdot \d_{e_j}
    \left(
        \dfrac{ a_e(s)}{\sqrt{W_e(s)}}
    \right)
    \right]\,\md s\;.
    \end{split}
    \end{equation*}
    This is the content of \cite[equations~(2.10)-(2.13)]{FP24}. We have thus completed the proof. 
\end{proof}

\begin{lem} \label{lem:ellipticity}
    There exists $C>0$ such that
    \begin{equation*}
        \sum_{i,j=1}^d \mu_{ij}(e) \, \eta_i \, \eta_j \geq C
    \end{equation*}
    for all unit vectors $e, \eta \in \SS^{d-1}$ satisfying $\bracket{e, \eta} = 0$. As a consequence, there exists $M >0$ such that 
    \begin{equation*}
        \widetilde{\mu}(e): = \mu(e) + M e \otimes e 
    \end{equation*}
    is positive definite (uniform in $e\in\SS^{d-1}$).
\end{lem}
\begin{proof}
    The first statement is \cite[equation~(2.14)]{FP24}. Since \(\mu\) is continuous on \(\SS^{d-1}\), it is uniformly bounded. For \(\zeta\in\RR^d\), write \(\zeta=\zeta_T+a e\) with \(\zeta_T\cdot e=0\). The first statement controls the \(\zeta_T\)-part uniformly, while the added term \(M e\otimes e\) controls the \(a e\)-part. Choosing \(M\) large enough and absorbing the mixed term by Cauchy inequality gives \(\zeta\cdot\widetilde{\mu}(e)\zeta\geq c|\zeta|^2\), uniformly in \(e\in\SS^{d-1}\).
\end{proof}

\subsection{Well-posedness of the stochastic mean curvature flow}

We consider the PDE
\begin{equation} \label{eq:PDE_approximate_rho}
    \begin{split}
    \d_t \rho_\eps^\pm =& -\frac{1}{\eps} \mathfrak{c} \Big( \nabla \rho_\eps^\pm, \; \eps \, \nabla^2 \rho_\eps^\pm \, \big( \id - \rho_\eps^\pm \nabla^2 \rho_\eps^\pm \big)^{-1}, \; \eps \xi_\eps (t, x - \rho_\eps^\pm \nabla \rho_\eps^\pm) \pm \eps^\beta \Big)\\
    &+ M \big( \nabla \rho_\eps^\pm \otimes \nabla \rho_\eps^\pm \big) : \big( \nabla^2 \rho_\eps^\pm \, ( \id - \rho_\eps^\pm \nabla^2 \rho_\eps^\pm )^{-1} \big)\;.
    \end{split}
\end{equation}
Taylor expanding $\mathfrak{c}$ around $(\nabla \rho, 0, 0)$, one sees the candidate limit $\rho$ should satisfy the equation
\begin{equation} \label{eq:PDE_limit_rho}
    \md  \rho = \widetilde{\mu}(\nabla\rho) : \bigl(\nabla^2\rho\,(\id-\rho\nabla^2\rho)^{-1}\bigr)\,\md t+b (\nabla \rho) \bm{\sigma}(x-\rho \nabla\rho)\,\md \W\;.
\end{equation}
We fix $\vV_0 \Subset \Omega$ with smooth boundary $\d \vV_0$. Let
\begin{equation*}
    g(x) := 
    \left\{
    \begin{aligned}
        &- \dist (x, \d \vV_0)\;, \qquad x \in \vV_0\\
        &\phantom{11}\dist (x, \d \vV_0)\;, \qquad x \in \vV_0^c\;,
    \end{aligned}\right.
\end{equation*}
be the signed distance function to the boundary $\d \vV_0$. The main theorem of this subsection is the following. 

\begin{thm}[Well-posedness] \label{thm:well_posedness_smcf}
    Fix $\alpha\in\left(\frac{1}{3},\frac{1}{2}\right)$. There exists an open set $\oO$ containing $\d \vV_0$ such that for almost every realisation of the noise $W^{j}$, there exists a random $\tau>0$ independent of $\eps$ and the $\pm$ sign such that the following assertions hold: 
    \begin{enumerate}
    \item The PDE \eqref{eq:PDE_approximate_rho} with initial and boundary conditions
    \begin{equation} \label{eq:rho_initial_bondary_cond}
        \rho_\eps^\pm (0,\cdot) = g\;, \qquad |\nabla \rho_\eps^\pm| = 1 \; \text{on} \; \d \oO
    \end{equation}
    has a unique classical solution in $[0,\tau] \times \oO$. Furthermore, the solution $\rho_\eps^\pm$ satisfies the bounds 
    \begin{equation} \label{eq:rho_gradient_bounds}
        \sum_{k=0}^4 \bigl\| \nabla^k \rho_\eps^\pm(t,\cdot) \bigr\|_{\lL_{t,x}^\infty([0,\tau] \times \overline{\oO})} \lesssim 1\;, \quad \sum_{\ell=0}^{2} \bigl\|\d_t \nabla^\ell \rho_\eps^\pm \bigr\|_{\lL_{t,x}^\infty([0,\tau] \times \overline{\oO})} \lesssim \eps^{-\kappa}\;.
    \end{equation}
    \item The rough PDE \eqref{eq:PDE_limit_rho} with the same initial and boundary condition as \eqref{eq:rho_initial_bondary_cond} has a unique solution $(\rho, \rho')$ in the sense of Definition~\ref{def:RPDE} with $\rho' = b(\nabla\rho)\bm{\sigma}(x-\rho \nabla\rho)$. Furthermore, they satisfy
    \begin{equation*}
        \big( \nabla^k \rho, \nabla^k \rho' \big) \in \lL_x^\infty \dD_\W^{2\alpha} \; \; \text{for} \; \; k \leq 2\;, \quad \nabla^3\rho, \nabla^2 \rho' \in \cC_t^\alpha \lL_x^\infty \cap \lL_t^\infty \cC_x^1([0,\tau] \times \overline{\oO})\;.
    \end{equation*}
    \item There exists $\theta>0$ such that
    \begin{equation} \label{eq:rho_convergence}
    \bigl\| \rho_{\eps}^\pm - \rho \bigr\|_{\lL_t^\infty \cC_x^2 ([0,\tau] \times \overline{\oO} )} \lesssim \eps^\theta\;.
    \end{equation}
    \item $|\nabla \rho_\eps^\pm| = |\nabla \rho| \equiv 1$ on $[0, \tau] \times \overline{\oO}$. 
    \end{enumerate}
All proportionality constants above depend on the realisation of $\W$ but are independent of $\eps$. 
\end{thm}

\begin{rmk}
    The term $M (\nabla \rho_\eps^\pm \otimes  \nabla \rho_\eps^\pm): \nabla^2\rho_\eps^\pm(\id-\rho_\eps^\pm\nabla^2\rho_\eps^\pm)^{-1}$ is auxiliary. It is added to make the PDE uniformly elliptic. Since $|\nabla\rho_\eps^\pm|\equiv1$ not only on the boundary but on all of $\oO$, this additional term is actually $0$.
\end{rmk}

\begin{prop} \label{prop:rho_approximate_effect}
    Let $\rho_\eps^\pm$ be the solution to \eqref{eq:PDE_approximate_rho} specified in Theorem~\ref{thm:well_posedness_smcf}. Then we have
    \begin{equation} \label{eq:rho_approximate_effect}
    \left| \d_t\rho_\eps^\pm + \frac1\eps \mathfrak{c}(\nabla\rho_\eps^\pm, \eps \nabla^2 \rho_\eps^\pm, \eps\xi_\eps\pm \eps^\beta) \right| \lesssim \eps^{-\kappa}|\rho_\eps^\pm|
    \end{equation}
    uniformly on $[0,\tau] \times \oO$. 
\end{prop}
\begin{proof}
By Theorem~\ref{thm:well_posedness_smcf}, we have $|\nabla \rho_\eps^\pm| = 1$. This implies
\begin{equation*}
    (\nabla \rho_\eps^\pm \otimes \nabla \rho_\eps^\pm) : \big(\nabla^2 \rho_\eps^\pm ( \id - \rho_\eps^\pm \nabla^2 \rho_\eps^\pm )^{-1}  \big) \equiv 0\;.
\end{equation*}
Hence, the left hand side of \eqref{eq:rho_approximate_effect} equals 
    \begin{equation*}
        \begin{split}
        \frac{1}{\eps} \Big| &\mathfrak{c}(\nabla\rho_\eps^\pm, \eps \nabla^2 \rho_\eps^\pm, \eps\xi_\eps\pm \eps^\beta)\\
        &-  \mathfrak{c} \big( \nabla \rho_\eps^\pm, \; \eps \, \nabla^2 \rho_\eps^\pm \, \big( \id - \rho_\eps^\pm \nabla^2 \rho_\eps^\pm \big)^{-1}, \; \eps \xi_\eps (t, x - \rho_\eps^\pm \nabla \rho_\eps^\pm) \pm \eps^\beta \big) \Big|\;.
        \end{split}
    \end{equation*}
    It can then be bounded by (uniformly over $(t,x) \in [0,\tau] \times \oO$)
    \begin{equation*}
        \|\d_q \mathfrak{c}\| \cdot \|\nabla^2 \rho_\eps^\pm\| \cdot \big|\id - (\id - \rho_\eps^\pm \nabla^2 \rho_\eps^\pm)^{-1} \big| + \|\d_\delta \mathfrak{c}\| \cdot \|\nabla \xi_\eps\| \cdot \big|\rho_\eps^\pm \nabla \rho_\eps^\pm \big| \lesssim \eps^{-\kappa} |\rho_\eps^\pm|\;,
    \end{equation*}
    where all the norms on the left hand side above are $\lL_{t,x}^\infty ([0,\tau] \times \oO)$, and the inequality comes from uniform boundedness of derivatives of $\rho_\eps^\pm$ in \eqref{eq:rho_gradient_bounds} and that $\|\nabla \xi_\eps\| \lesssim \eps^{-\kappa}$ (from the assumption on $\xi_\eps$ in Assumption~\ref{ass:assumption_overall}). This completes the proof. 
\end{proof}

The following proposition is in the same spirit of \cite[Theorem~2.2]{Web09}.

\begin{prop}[Interpretation of stochastic curvature flow]
\label{prop:interpretation}
Let $\Sigma_t:= \{x\in\oO: \rho(t,x) = 0\}$. Then we have the followings:
\begin{itemize}
\item[(i)] For every $t$ the function $x \mapsto \rho(t,x)$ is the signed distance function of $\Sigma_t$ on $\oO$.
\item[(ii)] If $X_0\in\Sigma_0$, then almost surely, there exists a stopping time $\tau>0$ such that the rough differential equation
\[
\md X_t= -\left(\left(\mu(\nabla\rho):\nabla^2\rho\right)\nabla\rho\right)(t,X_t) \,\md t - \left(b(\nabla\rho) \nabla\rho\right)(t,X_t) \bm{\sigma}(X_t)\, \md \W_t\;
\]
has a unique solution $(X_t)_{t \in [0,\tau]}$ with $X_t \in \Sigma_t$. 
\end{itemize}
\end{prop}
\begin{proof}
    The first assertion seems to be standard (see for example \cite[Lemma~5]{NT25}). We give an argument for completeness. We need to show that $\rho(t, \cdot)$ is the signed distance function for $\Sigma_t$ for every $t$. For simplicity, we omit the $t$ in notation, and only consider $x \in \oO$ such that $\rho(x)>0$. 
    
    Since $|\nabla \rho| = 1$, for every path $\gamma$ moving along the direction of $\nabla \rho$ (that is, $\dot{\gamma}_s = \nabla \rho (\gamma_s)$), we must have that $\nabla \rho (\gamma_s) = \nabla \rho (\gamma_0)$ is constant in time. One can then deduce that such a path has to be a straight line. 

    Fix $x \in \oO$ with $\rho(x)>0$ (and shrinking $\oO$ if necessary). Then there exists a unique $y \in \Sigma = \{\rho=0\}$ such that $|x-y| = \dist (x, \Sigma)$. This point $y \in \Sigma$ also satisfies that $x-y \perp \tT_y \Sigma$, where $\tT_y \Sigma$ denotes the tangent plane of $\Sigma$ at $y$. 

    On the other hand, $\rho \equiv 0$ on $\Sigma$ implies that we also have $\nabla \rho (y) \perp \tT_y \Sigma$. Since $\text{dim} \tT_y \Sigma = d-1$, we must have $\nabla \rho(y)$ parallel to $x-y$. In particular, we have
    \begin{equation*}
        \nabla \rho \big( y + s(x-y) \big) = \nabla \rho(y) = \frac{x-y}{|x-y|}
    \end{equation*}
    for all $s \in [0,1]$. Hence, we get
    \begin{equation*}
        \rho(x) = \rho(x) - \rho(y) = \int_{0}^{1} \bracket{ \nabla \rho \big(y +s(x-y) \big), x-y} \, \md s = |x-y|\;.
    \end{equation*}
    This completes the proof for the first assertion. 
    
    For the second assertion, we first consider the solution $\widetilde{X}_t$ to the RDE
    \begin{align*}
        \md \widetilde{X}_t=& -\left(\left(\widetilde{\mu}(\nabla\rho)
    :
    \bigl(\nabla^2\rho\,(\id-\rho\nabla^2\rho)^{-1}\bigr)\right)\nabla\rho\right)(t,\widetilde{X}_t) \,\md t \\
        &- \left(b(\nabla\rho) \nabla\rho\right)(t,\widetilde{X}_t) \bm{\sigma}(\widetilde{X}_t - \rho(t,\widetilde{X}_t)\nabla\rho(t,\widetilde{X}_t))\, \md \W_t
    \end{align*}
    with initial condition $\widetilde{X}_0 = X_0 \in \Sigma_0$. Then it has a unique solution with derivative process $\widetilde{X}'_t =- \left(b(\nabla\rho) \nabla\rho\right)(t,\widetilde{X}_t) \bm{\sigma}(\widetilde{X}_t - \rho(t,\widetilde{X}_t)\nabla\rho(t,\widetilde{X}_t))$. For the composite process $\rho(t, \widetilde{X}_t)$, it has the form
    \begin{equation*}
        \rho(t, \widetilde{X}_t) = \rho(0, X_0) + \int_0^t \big( E(s), E'(s) \big) \,\md \W + \int_0^t F(s) \,\md s\;.
    \end{equation*}
    By Lemma~\ref{lem:composition_2}, one can show that $E, E'$ and $F$ are all $0$. This shows that
    \begin{equation*}
        \rho (t, \widetilde{X}_t) = \rho(0, \widetilde{X}_0) \equiv 0\;.
    \end{equation*}
    As a result, $\widetilde{X}_t$ is also the solution to the RDE
    \begin{equation*}
    \md \widetilde{X}_t= -\left(\left(\mu(\nabla\rho):\nabla^2\rho\right)\nabla\rho\right)(t,\widetilde{X}_t) \,\md t - \left(b(\nabla\rho) \nabla\rho\right)(t,\widetilde{X}_t) \bm{\sigma}(\widetilde{X}_t)\, \md \W_t\;.
    \end{equation*}
    Then by the uniqueness of the RDE, we have $X_t = \widetilde{X}_t \in \Sigma_t$.
\end{proof}

\subsection{Proof of Theorem~\ref{thm:well_posedness_smcf}}

We first write the equation \eqref{eq:PDE_approximate_rho} for \(\rho_\eps^\pm\) in the form of Theorem~\ref{thm:stability}:
\begin{equation*}
    \partial_t \rho_\eps^\pm = \fF_\eps^\pm(\rho_\eps^\pm) - (\d_\delta \mathfrak{c})(\nabla\rho_\eps^\pm,0,0) \, \xi_\eps(t,x-\rho_\eps^\pm\nabla\rho_\eps^\pm)\;,
\end{equation*}
where
\begin{equation} \label{eq:approximate_RPDE_rho_F}
    \begin{split}
    F_\eps^\pm(t,x,u,p,q)
    :=& -\frac{1}{\eps} \mathfrak{c} \bigl(p,\eps q(\id-u q)^{-1},\eps\xi_\eps(t,x-u p)\pm\eps^\beta\bigr)\\
    & + M(p\otimes p):q(\id-u q)^{-1}+\partial_\delta \mathfrak{c}(p,0,0) \, \xi_\eps(t,x-u p)\;.
    \end{split}
\end{equation}
As for the equation \eqref{eq:PDE_limit_rho}, by Lemma~\ref{lem:expressions_coefficients}, it has the form of Definition~\ref{def:RPDE} with
\begin{equation} \label{eq:limit_RPDE_rho_F}
    F_0(t,x,u,p,q) : =\left(-\partial_q \mathfrak{c}(p,0,0)+M p\otimes p\right):q(\id-u q)^{-1}=\widetilde{\mu}(p):q(\id-u q)^{-1}\;.
\end{equation}
In both cases, by Lemma~\ref{lem:expressions_coefficients} again, we have
\begin{equation} \label{eq:RPDE_rho_H}
    H(x,u,p) = - (\d_\delta \mathfrak{c})(p,0,0) \, \bm{\sigma} \big(x - u p \big) = b(p) \, \bm{\sigma} \big(x - u p \big)\;.
\end{equation}
We now fix the domains on which we check the hypotheses of Theorem~\ref{thm:stability}. Let
\[
    R:=\sup_{x\in\partial\vV_0}|\nabla^2\rho^0(x)|\;.
\]
By Lemma~\ref{lem:ellipticity} and the continuity of \(\widetilde\mu\), we may choose \(r_0\in(0,1)\) sufficiently small such that
\begin{equation}\label{e:mu_ellipticity_bound}
    \lambda_\mu
    :=
    \inf_{\substack{1-r_0\leq |p|\leq1+r_0\\|\eta|=1}}
    \eta^\sft\widetilde\mu(p)\eta>0\;,
    \quad
    M_\mu
    :=
    \sup_{1-r_0\leq |p|\leq1+r_0}
    \|\widetilde\mu(p)\|<+\infty\;.
\end{equation}
Choose 
\[a_0 := \frac{\lambda_\mu}{24RM_\mu}\] 
and set
\begin{equation*}
    \begin{split}
    \dD_2:=\Big\{ (u,p) \in\RR\times\RR^d:|u|<a_0,\;1-r_0<|p|<1+r_0 \Big\}\;,\\
    \dD_3 := \Big\{ (u,p,q) \in \RR \times \RR^d \times \RR^{d \times d}: (u,p) \in \dD_2\;, \; |q| < 3R \Big\}\;.
    \end{split}
\end{equation*}
On \(\dD_3\), the matrix \(\id - u \, q\) is invertible and \(q \, (\id-uq)^{-1}\) is uniformly bounded. Take a smooth neighbourhood \(\widehat{\oO}\) of \(\partial\vV_0\) such that
\[
    |\rho^0(x)|<\frac{a_0}{2}\;,\qquad |\nabla^2\rho^0(x)|<2R
\]
for all \(x\in\overline{\widehat{\oO}}\). Then choose the tubular neighbourhood \(\widetilde\oO=\{x:|\rho^0(x)|<a_0/3\} \Subset \widehat{\oO}\) of \(\partial\vV_0\), and choose \(T>0\) such that \(X_t^\eps(x,u,p)\in\widehat{\oO}\) for all \(t\in[0,T]\), \(x\in\overline{\widetilde{\oO}}\), \((u,p)\in\dD_2\) and $\eps \in [0,1]$. Finally, fix an open set \(\oO\) with smooth boundary such that \(\partial\vV_0\subset\oO\Subset\widetilde\oO\).

\begin{lem}\label{lem:convergence_Feps}
There exists $\theta>0$ such that for every realisation of $\W$, one has
\[
    \sup_{\substack{x\in\RR^d,\;(u,p,q)\in\dD_3\\ |\gamma|\leq4}}
    \bigl\|\nabla_{(x,u,p,q)}^\gamma(F_\eps^\pm-F_0)(\cdot,x,u,p,q)\bigr\|_{\cC^\alpha([0,T])}
    \lesssim \eps^\theta
\]
for all sufficiently small $\eps$. The proportionality constant depends on the realisation of the noise but is independent of $\eps$. 
\end{lem}
\begin{proof}
Define $\widehat{F}_\eps^\pm$ and $\widehat{F}_0$ by
\begin{equation*}
    \begin{split}
    \widehat{F}_\eps^\pm(t,x,p,q) &:= -\frac{1}{\eps} \mathfrak{c} \big(p, \, \eps q, \,\eps \xi_\eps(t,x) \pm \eps^\beta \big) + M(p\otimes p): q + \d_\delta \mathfrak{c}(p,0,0) \, \xi_\eps(t,x)\;,\\
    \widehat{F}_0(t,x,p,q) &:= \left(-\partial_q \mathfrak{c}(p,0,0) + M (p \otimes p) \right): q\;.
    \end{split}
\end{equation*}
Then we have
\begin{equation*}
    F_\eps^\pm (t,x,u,p,q) = \widehat{F}_\eps^\pm \big(t, \, x-up, \, p, \, q (\id - uq)^{-1} \big)\;,
\end{equation*}
and the same relation for $\widehat{F}_0$ and $F_0$. The map
\begin{equation*}
    (x,u,p,q) \mapsto \bigl(x-up, \, p, \, q(\id-uq)^{-1}\bigr)
\end{equation*}
has uniformly bounded derivatives up to any finite order on \(\RR^d\times\dD_3\), and the range of its last two components is within the bounded set
\begin{equation*}
    1 - r_0 < |p| < 1 + r_0\;, \qquad \big| q \, (\id - uq)^{-1} \big| < 10 R\;.
\end{equation*}
Hence, it suffices to prove the corresponding convergence of derivatives of \(\widehat F_\eps^\pm-\widehat F_0\) uniformly for $x \in \RR^d$ and $(p,q)$ in compact sets of $\RR^d \times \RR^{d \times d}$. 

Taylor expanding $\mathfrak{c}$ at $(p,0,0)$ and using $\mathfrak{c}(p,0,0) = 0$, we get
\begin{equation*}
    \widehat{F}_\eps^\pm(t,x,p,q) - \widehat{F}_0(t,x,p,q) = -\frac{1}{\eps} \cdot \widehat{\rR}_\eps^{\pm}(t,x,p,q) \mp \eps^{\beta-1} \, (\d_\delta \mathfrak{c})(p,0,0)\;,
\end{equation*}
where $\widehat{\rR}_\eps^\pm$ is the second order Taylor remainder given by
\begin{equation*}
    \begin{split}
    \widehat{\rR}_\eps^\pm (t,x,p,q) &= \mathfrak{c} \big( p, \, \eps q, \, \eps \xi_\eps \pm \eps^\beta \big) - \mathfrak{c}(p,0,0)\\
    &\phantom{111}- (\d_\delta \mathfrak{c})(p,0,0) \cdot (\eps \xi_\eps \pm \eps^\beta) - (\d_q \mathfrak{c})(p,0,0) : \eps q\;.
    \end{split}
\end{equation*}
Since $\eps |\nabla^k \xi_\eps| \lesssim \eps^{1-\kappa}$ for every $k$, we have
\begin{equation*}
    |D^\gamma \widehat{\rR}_\eps^\pm| \lesssim \eps^{2(1-\kappa)}
\end{equation*}
for any fixed multi-index $\gamma$ with respect to $(x,p,q)$, and hence
\begin{equation*}
    \| D_{x,p,q}^\gamma (\widehat{F}_\eps^\pm - \widehat{F}_0) \|_{\lL_t^\infty \lL_{x,p,q}^\infty} \lesssim \eps^{1-2\kappa}
\end{equation*}
uniformly in $\eps$, $t \in [0,T]$ for fixed $T$, $x \in \RR^d$ and $(p,q)$ in compact sets of $\RR^d \times \RR^{d \times d}$. 

As for H\"older continuity in time, we note that
\begin{equation*}
    \eps \|\d_t \nabla^k \xi_\eps\|_{\lL_{t,x}^\infty} \lesssim \eps^{1-2\kappa}\;,
\end{equation*}
and hence 
\begin{equation*}
    \| D_{x,p,q}^\gamma (\widehat{F}_\eps^\pm - \widehat{F}_0) \|_{\cC^1_t \lL_{x,p,q}^\infty} \lesssim \eps^{1-3\kappa}\;.
\end{equation*}
This completes the proof. 
\end{proof}

\begin{proof}[Proof of Theorem~\ref{thm:well_posedness_smcf}]
    Recall the forms of $F_\eps^\pm$, $F_0$ and $H$ from \eqref{eq:approximate_RPDE_rho_F}, \eqref{eq:limit_RPDE_rho_F} and \eqref{eq:RPDE_rho_H}. The boundary condition corresponds to
    \begin{equation} \label{eq:RPDE_rho_G}
        G_\eps(x,u,p)=G(x,u,p):=|p|^2-1\;.
    \end{equation}
    We verify the assumptions of Theorem~\ref{thm:stability}. We start with assumptions related to the boundary condition $G$. Since \(\rho^0\) is the signed distance function to $\d \vV_0$, we have
    \begin{equation*}
        G(x,\rho^0(x),\nabla\rho^0(x))=|\nabla\rho^0(x)|^2-1=0
    \end{equation*}
    for all $x \in \overline{\widetilde{\oO}}$, and in particular for $x \in \d \widetilde\oO$. For non-tangentiality, we note that the outward normal $\nu$ on $\d \widetilde\oO$ satisfies $\nu = \pm \nabla \rho^0$. Hence, on $\d \widetilde\oO$, we have
    \begin{equation*}
        D_p G (x,\rho^0,\nabla\rho^0) \cdot \nu = 2 \, \nabla\rho^0 \cdot \nu = \pm 2\;.
    \end{equation*}
    This verifies the non-tangentiality condition \eqref{e:condition_G_non_tangent_RPDE} and hence Assumption 5. 

    For the compatibility condition \eqref{e:condition_GH_factorization}, a direct computation gives 
    \[
        -D_xG(D_pH)^\sft+D_uG(H-(D_pH) p)^\sft+D_pG((D_xH)^\sft+p(D_uH)^\sft)=2b(p)(1-|p|^2)p^\sft\nabla\bm{\sigma}^\sft(x-up)\;.
    \]
    Thus condition~\eqref{e:condition_GH_factorization} holds with
    \(L(x,u,p):=-2b(p)p^\sft\nabla\bm{\sigma}^\sft(x-up)\).
    The required regularity of \(H\) and \(L\) follows from the smoothness of \(b\) on the annulus and of \(\bm{\sigma}\). This verifies Assumption 4. 

    Assumption 3 is immediate from the definition \eqref{eq:RPDE_rho_G}, while Assumption 2 is implied by Lemma~\ref{lem:convergence_Feps}. 
    
    We now turn to Assumption 1 on ellipticity of $F_\eps^\pm$ and $F_0$. In view of Lemma~\ref{lem:convergence_Feps}, it suffices to check that of $F_0$. Note that
    \[
        D_q\bigl(q(\id-uq)^{-1}\bigr)[h] = (\id - u \, q)^{-1} \, h \, (\id - u \, q)^{-1}\;.
    \]
    By the choice of $a_0$ above and the fact that $\lambda_\mu\leq M_\mu$,
    \[
        \big\|(\id-uq)^{-1}-\id\big\|
        \leq \frac{|u|\,\|q\|}{1-|u|\,\|q\|}
        \leq \frac{\lambda_\mu}{7M_\mu}
    \]
    uniformly on $\dD_3$. Hence, by \eqref{e:mu_ellipticity_bound}, for every $\eta \in \RR^d$,
    \begin{align*}
        D_q F_0(t,x,u,p,q)[\eta \otimes\eta]
        &= \widetilde{\mu}(p):(\id-uq)^{-1}(\eta\otimes\eta)(\id-uq)^{-1}\\
        &\geq \left[\lambda_\mu
        -M_\mu\left(2\frac{\lambda_\mu}{7M_\mu}
        +\frac{\lambda_\mu^2}{49M_\mu^2}\right)\right]|\eta|^2
        \geq \frac{\lambda_\mu}{2}|\eta|^2\;.
    \end{align*}
    Thus $F_0$ is uniformly elliptic on $\dD_3$. Combining with Lemma~\ref{lem:convergence_Feps}, we obtain
    \[
        \triplenorm{F_\eps^\pm-F_0}_{\mathcal F,\mathrm{RPDE}} \lesssim \eps^\theta
    \]
    for a possibly different $\theta > 0$.
    In particular, \(F_\eps^\pm\) belongs to the fixed neighbourhood of \(F_0\) required in Theorem~\ref{thm:stability} for all sufficiently
    small \(\eps\).

    Hence, by Theorem~\ref{thm:stability}, Lemma~\ref{lem:PDE_further_regularity} with $k=2$, and Assertion~2 of Lemma~\ref{lem:transform}, there exists $\tau>0$ such that
    \begin{equation*}
        \|\rho_\eps^\pm - \rho\|_{\lL_t^\infty \cC_x^2 ([0,\tau] \times \overline{\oO})} \lesssim |\!|\!| F_\eps^\pm - F_0 |\!|\!| + d_\alpha (\W^\eps, \W) \lesssim \eps^\theta\;,
    \end{equation*}
    together with the desired spatial derivative bounds for $\rho_\eps^\pm$ in \eqref{eq:rho_gradient_bounds}. Moreover, $\rho_\eps^\pm$ satisfies the boundary condition
    \begin{equation} \label{eq:rho_boundary_cond_moving_domain}
        |\nabla \rho_\eps^\pm (t)| = 1 \quad \text{on} \; \; \d \oO^{\eps;\pm}_t\;, \quad \oO_t^{\eps;\pm} = \big\{ \widehat{X}_t^{\eps;\pm} (y): \; y \in \widetilde{\oO} \big\}\;.
    \end{equation}
    Moreover, we have $\oO\Subset\oO_t^{\eps;\pm}$ for every $t\in[0,\tau]$, all sufficiently small $\eps$, and both signs.
    Furthermore, differentiating the equation for \(\rho_\eps^\pm\) in space up to two derivatives and using \(\mathfrak{c}(e,0,0)=0\), the smoothness of \(\mathfrak{c}\), and the assumptions \eqref{eq:cond_noise_form} and \eqref{eq:cond_BM_approximation} on the noise $\xi_\eps$, one also deduces the bound for $\|\d_t \nabla^\ell \rho_\eps^\pm\|$ in \eqref{eq:rho_gradient_bounds}. 

    Now, it is almost there except that the boundaries for $\rho_\eps^\pm$ depend on time. Recall that
    \begin{equation*}
        \Phi_t^{\eps;\pm}(x) = \big( x,v_\eps^{\pm}(t,x), \nabla v_\eps^{\pm}(t,x) \big)\;, \quad \widehat{\Theta}_t^{\eps;\pm} = \Theta_t^\eps \circ \Phi_t^{\eps;\pm}\;.
    \end{equation*} 
    For every fixed $\eps >0$, since $W_t^{j,\eps}$ is smooth, we deduce that $\rho_\eps^{\pm} $ is a classical solution to \eqref{eq:PDE_approximate_rho} on the domain
    \begin{equation} \label{eq:moving_domain}
        \sS_\tau^{\eps;\pm} := \big\{ (t,x): t \in [0,\tau]\;, \; x \in \oO_t^{\eps;\pm} \big\}
    \end{equation}
    with initial data $\rho_\eps^{\pm}(0) = \rho^0$ defined on $\oO_0^{\eps;\pm} = \widetilde\oO$ and satisfies the boundary condition \eqref{eq:rho_boundary_cond_moving_domain}. We now need to justify that $|\nabla \rho_\eps^{\pm}| = 1$ on $\d \oO$. Indeed, we will show the identity on all of \eqref{eq:moving_domain}. Let
\begin{equation} \label{eq:w_boundary_defn}
    w_\eps^{\pm} := |\nabla \rho_\eps^{\pm}|^2 - 1\;.
\end{equation}
We shall derive the equation for $w_\eps^{\pm}$ and show it satisfies the regularity assumption needed to apply Lemma~\ref{lem:moving_domain_maximum_principle} from the fixed-domain representation. 

Since \(W^\eps\) is smooth and the coefficients and initial data are smooth, standard parabolic regularity gives, \(v_\eps^{\pm}\in\cC^{1,4}_{\mathrm{loc}}((0,T] \times \widetilde\oO)\), while \(v_\eps^{\pm}\), \(\nabla v_\eps^{\pm}\), and \(\nabla^2 v_\eps^{\pm}\) are continuous up to \([0,T] \times \overline{\widetilde{\oO}}\). Thus the quantities below are continuous up to the fixed boundary and have the required interior differentiability. 

Since \(\widehat X_t^{\eps;\pm}\) is a diffeomorphism from \(\overline{\widetilde{\oO}}\) onto \(\overline{\oO_t^{\eps;\pm}}\), with inverse continuous up to the boundary, \(\nabla\rho^\pm_\eps(t,\cdot)=\widehat P_t^{\eps;\pm}\circ(\widehat X_t^{\eps;\pm})^{-1}\) is continuous on the closure of the moving cylinder and has the required interior differentiability. 
One can then deduce that $w^\pm_\eps\in\cC(\overline{\sS_\tau^{\eps;\pm}})\cap\cC^{1,2}(\sS_\tau^{\eps;\pm})$. 

Introduce the shorthand notation
\[
    A:=\id -\rho_\eps^{\pm} \nabla^2 \rho_\eps^{\pm}\;,\quad
    y_\eps^\pm := x - \rho_\eps^{\pm} \nabla \rho_\eps^{\pm}\;,
\]
\[
    \delta_\eps^\pm:=\eps\xi_\eps(t,y_\eps^\pm) \pm \eps^\beta\;, \quad
    \mu_\eps^\pm := M \nabla \rho_\eps^{\pm} \otimes \nabla \rho_\eps^{\pm} -\d_q \mathfrak{c} \bigl(\nabla \rho_\eps^{\pm}, \, \eps \nabla^2 \rho_\eps^{\pm} A^{-1}, \, \delta_\eps^\pm \bigr)\;.
\]
Differentiating the relation \eqref{eq:w_boundary_defn} and using the equation \eqref{eq:PDE_approximate_rho} for $\rho_\eps^\pm$, we get
\[
 \partial_t w_\eps^{\pm}
    =\sum_{i,j=1}^d a_{ij}^{\eps;\pm} \d_{ij}^2 w_\eps^{\pm}
    +\sum_{i=1}^d b_i^{\eps;\pm} \d_i w_\eps^{\pm} + c^{\eps;\pm} w_\eps^{\pm}
\]
on $\sS_\tau^{\eps;\pm}$, where the coefficients are given by
\begin{equation*}
\begin{split}
    a^{\eps;\pm} =& A^{-1} \mu_\eps^\pm A^{-1}\;,\\
    b^{\eps;\pm} =& -\eps^{-1} \d_e \mathfrak{c} \bigl(\nabla \rho_\eps^{\pm}, \, \eps \nabla^2 \rho_\eps^{\pm} A^{-1}, \, \delta_\eps^\pm \bigr)\\
    &+ \rho_\eps^{\pm} \d_\delta \mathfrak{c} \bigl(\nabla \rho_\eps^{\pm}, \, \eps \nabla^2 \rho_\eps^{\pm} A^{-1}, \,\delta_\eps^\pm \bigr) \nabla\xi_\eps(t,y_\eps^\pm)
    + 2 M \nabla^2 \rho_\eps^{\pm}A^{-1} \nabla\rho_\eps^{\pm}\;,\\
    c^{\eps;\pm} =& 2 \mu_\eps^\pm: \bigl(\nabla^2 \rho_\eps^{\pm}A^{-1} \bigr)^2
    +2 \d_\delta \mathfrak{c} \bigl(\nabla \rho_\eps^{\pm}, \, \eps \nabla^2 \rho_\eps^{\pm}A^{-1}, \, \delta_\eps^\pm \bigr) \nabla \rho_\eps^{\pm} \cdot \nabla \xi_\eps(t,y_\eps^\pm)\;.
\end{split}
\end{equation*}
The matrix \(a^{\eps;\pm}\) is positive definite for all small \(\eps\): indeed, \(\mu_\eps^\pm\) converges uniformly to \(\widetilde\mu(p)\), which is uniformly positive definite by Lemma~\ref{lem:ellipticity}, and \(A\) is uniformly invertible on the time interval under consideration. For every fixed $\eps>0$, these coefficients are bounded on $\sS_\tau^{\eps;\pm}$. Moreover, $w_\eps^{\pm} = 0$ on $\{0\} \times \overline{\widetilde{\oO}}$ and on $\d \oO_t^{\eps;\pm}\times \{t\}$ for $t\in [0,\tau]$. Hence, by Lemma~\ref{lem:moving_domain_maximum_principle}, we have $w_\eps^{\pm} \equiv 0$ on all of $\sS_\tau^{\eps;\pm}$. Since $[0,\tau] \times \overline{\oO} \subset \sS_\tau^{\eps;\pm}$, this shows that $|\nabla \rho_\eps^\pm| \equiv 1$ on $[0,\tau] \times \overline{\oO}$ and in particular on $[0, \tau] \times \d \oO$. 

Hence, we obtain the solution $\rho_\eps^\pm$ and $\rho$ to \eqref{eq:PDE_approximate_rho} and \eqref{eq:PDE_limit_rho} with the required initial and boundary conditions on the domain $[0,\tau] \times \overline\oO$. Finally, the convergence \eqref{eq:rho_convergence} implies that we also have $|\nabla \rho| = 1$ on $\overline\oO$. This completes the proof. 
\end{proof}

\section{Convergence to the stochastic mean curvature flow}
\label{sec:sub_sup_convergence}

In this section, we rigorously construct sub- and super-solutions to \eqref{SAC}, and proves the convergence of the interface of $u_\eps$ to the direction-dependent stochastic curvature flow. The arguments follow the lines of \cite[Section~3]{Web09}, but with a more sophisticated ansatz incorporating the nonlinear diffusion. 

\subsection{Construction of sub- and super-solutions}
\label{sec:sub_sup}

Fix an initial surface $\Sigma_0$ as in Theorem \ref{thm:main}. Let $\rho_\eps^\pm$ be the solutions to \eqref{eq:PDE_approximate_rho} in Theorem~\ref{thm:well_posedness_smcf} on $[0,\tau] \times \oO$. Let
\begin{equation*}
    \Sigma_t^{\eps;\pm} := \big\{ x \in \oO: \rho_\eps^\pm (t,x) = 0 \big\}\;.
\end{equation*}
Let $\vV_t^{\eps;\pm}$ be the bounded region enclosed by $\Sigma_t^{\eps;\pm}$. Choose $d_*>0$ (by shortening the time interval if necessary) such that for all $ t \in [0, \tau]$, the $(3d_*)$-neighborhood of $\Sigma^{\eps;\pm}_t$ is contained in $\oO$. For each $t \in [0,\tau]$, we split $\Omega$ into disjoint sets $\Omega = \dD_1^{\eps;\pm} \cup \dD_2^{\eps;\pm} \cup \dD_3^{\eps;\pm}$ (omitting $t$ in the notation) with
\begin{equation*}
    \dD_1^{\eps;\pm} = \big\{ x: \dist (x, \Sigma_t^{\eps;\pm}) < d_*  \big\}\;, \quad \dD_2^{\eps;\pm} = \big\{ x: \dist (x, \Sigma_t^{\eps;\pm}) > 2 d_*  \big\}
\end{equation*}
and $\dD_3^{\eps;\pm} = \Omega \setminus \big( \dD_1^{\eps;\pm} \cup \dD_2^{\eps;\pm} \big)$. Choose a smooth cut-off function \(\vartheta\in\cC^\infty(\RR;[0,1])\) such that
\[
    \vartheta(s)=0\quad\text{for } |s|\leq d_*\;,\quad
    \vartheta(s)=1\quad\text{for } |s|\geq 2d_*\;,
\]
and smooth in between. Fix $a,\beta$ with $\beta \in (1, 2-\kappa)$ and $a \in (\beta-1,\beta)$. Define 
\[
    \delta_\eps^\pm(t,x):=\eps\xi_\eps(t,x)\pm\eps^\beta\;,
\]
and $\widehat{u}_\eps^\pm$ and $\overline{u}_\eps^\pm$ by
\[
 \widehat u_\eps^\pm(t,x)=\uU\left(\frac{\rho_{\eps}^{\pm}(t,x)}{\eps} \pm\eps^a,
 \nabla\rho_\eps^\pm(t,x),
 \eps \nabla^2\rho_\eps^\pm(t,x),
 \delta_\eps^\pm(t,x) \right)\;,
\]
and
\begin{equation} \label{eq:u_bar_def}
 \overline{u}_\eps^\pm(t,x):=
 \begin{cases}
 U_+(\delta_\eps^\pm(t,x))\;,& x\notin \vV_t^{\eps;\pm}\;,\\
 U_-(\delta_\eps^\pm(t,x))\;,& x\in \vV_t^{\eps;\pm}\;.
 \end{cases}
\end{equation}
We then define the sub- and super-solutions $u_\eps^\pm$ by
\begin{equation} \label{eq:sub_sup_def}
 u_\eps^\pm(t,x):=
 \begin{cases}
     \bigl(1-\vartheta(\rho_\eps^\pm(t,x))\bigr)\widehat u_\eps^\pm(t,x)
 +
 \vartheta(\rho_\eps^\pm(t,x))\overline{u}_\eps^\pm(t,x)\;, & x\in\oO\;,\\
 \overline{u}_\eps^\pm(t,x)\;,& x\notin \oO\;.
 \end{cases}
\end{equation}
We note that the right hand side above is well defined. Indeed, $\vartheta(\rho_\eps^\pm)=0$ in $\dD_1^{\eps;\pm}$ and $\vartheta(\rho_\eps^\pm)=1$ in $\dD_2^{\eps;\pm}$, while
\[
    \operatorname{supp}\bigl(1-\vartheta(\rho_\eps^\pm)\bigr)
    \subset \dD_1^{\eps;\pm}\cup\dD_3^{\eps;\pm}\subset\oO\;.
\]
Thus $\bigl(1-\vartheta(\rho_\eps^\pm)\bigr)\widehat u_\eps^\pm$ is always well-defined, and the jump of $\overline u_\eps^\pm$ across $\Sigma_t^{\eps;\pm}\subset\dD_1^{\eps;\pm}$ does not contribute to $u_\eps^\pm$. Moreover, since a neighbourhood of $\partial\oO$ is contained in $\dD_2^{\eps;\pm}$, the two cases in \eqref{eq:sub_sup_def} agree near $\partial\oO$.
In particular, $u_\eps^\pm=\widehat u_\eps^\pm$ in $\dD_1^{\eps;\pm}$, while $u_\eps^\pm=\overline{u}_\eps^\pm$ in $\dD_2^{\eps;\pm}$.  Also recall the choice $\psi_{ij}$ from \eqref{eq:choice_psi}. We have the following comparison statement.

\begin{lem}\label{lem:sub_sup}
There exists a (random) $\eps_0>0$ such that for all $\eps\leq \eps_0$ and for $0 \leq t \leq \tau$, we have 
\[
 u_\eps^-(t,x) \leq u_{\eps} (t,x) \leq u_\eps^+(t,x)\;,
\]
for every solution $u_{\eps}(t,x)$ of (\ref{SAC}) with initial data satisfying  $u_\eps^-(0,x) \leq u_{\eps} (0,x) \leq u_\eps^+(0,x)$.
\end{lem}
\begin{proof}

Define the operator $\nN_\eps$ by
\begin{equation*}
    \nN_\eps v := \d_t v - \div \big( D(v) \nabla v \big) -\eps^{-2}f(v)-\eps^{-1}\xi_\eps\;.
\end{equation*}
It suffices to show that
\begin{equation} \label{eq:sub_sup}
    \nN_\eps u_\eps^+\geq 0\;, \qquad \nN_\eps u_\eps^-\leq 0
\end{equation}
on $[0,\tau] \times \Omega$ for all sufficiently small $\eps$, together with the Neumann boundary condition. The conclusion of the lemma then follows from the comparison principle and the ordering of the initial data. In fact, we will show the more quantitative bound
\begin{equation} \label{eq:sub_sup_quant}
    \nN_\eps u_\eps^\pm = \pm \eps^{\beta-2} + \oO(\eps^{-\kappa})\;, \qquad \text{on} \; [0,\tau] \times \Omega\;.
\end{equation}
Note that by Proposition~\ref{prop:interpretation}, we have $|\rho_\eps^\pm(t,x)| = \text{dist}(x, \Sigma_t^{\eps;\pm})$. 

We start with $\dD_1^{\eps;\pm}$. By definition, we have $u_\eps^\pm = \widehat{u}_\eps^\pm$ on $\dD_1^{\eps;\pm}$. By the bounds \eqref{eq:rho_gradient_bounds} from Theorem~\ref{thm:well_posedness_smcf}, we see the bounds \eqref{eq:postulation} in Postulation~\ref{pos:postulation} are satisfied. Hence the computations in Section~\ref{sec:formal_derivation} are all valid, and we have
\begin{align}\label{e:inner_residual_new}
    \nN_\eps \widehat u_\eps^\pm = \frac1\eps \, \uU_z \cdot \left( \d_t \rho_\eps^\pm +\frac1\eps \mathfrak{c}(\nabla\rho_\eps^\pm,\eps \nabla^2 \rho_\eps^\pm, \eps \xi_\eps \pm \eps^\beta) \right) \pm \eps^{\beta-2}
    +\oO(\eps^{-\kappa})\;,
\end{align}
where the $z$-variable in $\uU_z$ is $z_\eps^\pm = \frac{\rho_\eps^\pm}{\eps} \pm \eps^a$. By Lemma~\ref{lem:CHL}, we have
\begin{equation*}
    \frac{|\rho_\eps^\pm|}{\eps}\,
    \big|\uU_z\big(z_\eps^\pm,e,q,\delta\big)\big|
    =\big|z_\eps^\pm\mp\eps^a\big|\,
    \big|\uU_z\big(z_\eps^\pm,e,q,\delta\big)\big|
    \lesssim 1\;.
\end{equation*}
Also, since $\dD_1^{\eps;\pm}\subset\oO$, Proposition~\ref{prop:rho_approximate_effect} applies to the term in parentheses on the right hand side of \eqref{e:inner_residual_new}. Combining the proposition with the preceding estimate shows that the first term on the right hand side of \eqref{e:inner_residual_new} is $\oO(\eps^{-\kappa})$. Hence, \eqref{eq:sub_sup_quant} holds on $\dD_1^{\eps;\pm}$.

We next consider $\dD_2^{\eps;\pm}$, where $u_\eps^\pm=\overline{u}_\eps^\pm$. By the definition of $u_\eps^\pm$ in \eqref{eq:u_bar_def} and the identity \eqref{eq:root_shift_f} (with $\delta = \delta_\eps^\pm$), we have
\begin{equation*}
    -\eps^{-2}f(\overline{u}_\eps^\pm)-\eps^{-1}\xi_\eps = \eps^{-2} \delta_\eps^\pm - \eps^{-1} \xi_\eps = \pm \eps^{\beta-2}\;.
\end{equation*}
The remaining time and diffusion terms contain only derivatives of $\delta_\eps^\pm$. By assumptions on the mollified noise, one time derivative and up to two spatial derivatives of $\delta_\eps^\pm$ are all of order $\oO(1)$. Also, since $\overline{u}_\eps^\pm$ are smooth functions of $\delta_\eps^\pm$ only, we have
\begin{equation*}
    \big| \d_t \overline{u}_\eps^\pm - \div \big( D(\overline{u}_\eps^\pm) \nabla \overline{u}_\eps^\pm \big) \big| \lesssim 1 + |\d_t \delta_\eps^\pm| + |\nabla \delta_\eps^\pm|^2 + |\nabla^2 \delta_\eps^\pm| = \oO(1)\;.
\end{equation*}
Hence, we conclude \eqref{eq:sub_sup_quant} also holds on $\dD_2^{\eps;\pm}$. 

It remains to deal with $\dD_3^{\eps;\pm}$. Set
\begin{equation*}
    r_\eps^\pm:=\widehat u_\eps^\pm-\overline{u}_\eps^\pm\;.
\end{equation*}
By Lemma~\ref{lem:CHL} (with $\uU = \widehat{u}_\eps^\pm$, $U_{\pm} = \overline{u}_\eps^\pm$ and $\delta = \delta_\eps^\pm$) and the derivative bounds above, we have
\begin{equation} \label{e:r_exp_decay_new}
    |r_\eps^\pm| +\eps^{1+\kappa} |\d_t r_\eps^\pm| + \eps|\nabla r_\eps^\pm| + \eps^2 |\nabla^2 r_\eps^\pm| \lesssim e^{-\lambda d_*/(2\eps)}
\end{equation}
uniformly in $\dD_3^{\eps;\pm}$. The exponential decay follows from the fact that $\big| \frac{\rho_\eps^\pm}{\eps} \pm \eps^a \big| \geq \frac{d_*}{2\eps}$ in $\dD_3^{\eps;\pm}$ for all sufficiently small $\eps$. Since
\begin{equation*}
    u_\eps^\pm = \overline{u}_\eps^\pm + \bigl( 1 -\vartheta(\rho_\eps^\pm) \bigr) r_\eps^\pm\;,
\end{equation*}
using that the spatial derivatives of $\rho_\eps^\pm$ are uniformly bounded and that $\|\d_t \rho_\eps^\pm\| \lesssim \eps^{-\kappa}$, and combining with \eqref{e:r_exp_decay_new}, we get
\begin{equation*}
    |u_\eps^\pm - \overline{u}_\eps^\pm| + \eps^{1+\kappa} |\d_t (u_\eps^\pm - \overline{u}_\eps^\pm)| + \eps |\nabla (u_\eps^\pm - \overline{u}_\eps^\pm)| + \eps^2 |\nabla^2 (u_\eps^\pm -\overline{u}_\eps^\pm)| \lesssim e^{-\lambda d_*/(2\eps)}\;.
\end{equation*}
Using the smoothness and boundedness of $D$ and $f$, we may compare $\nN_\eps u_\eps^\pm$ and $\nN_\eps\overline{u}_\eps^\pm$ termwise to get
\begin{equation*}
    |\nN_\eps u_\eps^\pm - \nN_\eps\overline{u}_\eps^\pm|
    \lesssim \eps^{-2} e^{-\lambda d_*/(2\eps)}\;.
\end{equation*}
Combining this with the fact that $\nN_\eps \overline{u}_\eps^\pm = \pm \eps^{\beta-2} + \oO(\eps^{-\kappa})$ also holds on $\dD_3^{\eps;\pm}$, we obtain that \eqref{eq:sub_sup_quant} also holds on $\dD_3^{\eps;\pm}$, thus concluding \eqref{eq:sub_sup} on all of $[0,\tau] \times \Omega$. 

We finally turn to the boundary condition. By definition, we have
\begin{equation*}
    u_\eps^\pm(t,x) = \overline{u}_\eps^\pm (t,x) = U_+ \big( \eps \xi_\eps (t,x) \pm \eps^\beta \big)
\end{equation*}
near $\d \Omega$. Since $\xi_\eps$ has compact support in $\Omega$, it then follows directly that $u_\eps^\pm$ is constant near $\d \Omega$, and hence satisfies the Neumann boundary condition. The comparison principle now gives the desired ordering.
\end{proof}

\begin{lem} \label{lem:initial_ordering}
    For all sufficiently small $\eps$, we have
    \begin{equation*}
        u_\eps^-(0,x) \leq u_\eps^+(0,x)\;.
    \end{equation*}
\end{lem}
\begin{proof}
We will actually show that
\begin{equation} \label{eq:initial_ordering}
    \widehat{u}_\eps^{-} (0,x) \leq \widehat{u}_\eps^{+}(0,x) \text{ for all } x\in \oO\;, \qquad \overline{u}_\eps^{-}(0,x) \leq \overline{u}_\eps^{+}(0,x) \text{ for all } x \in \Omega\;.
\end{equation}
The conclusion of the lemma then follows immediately from the definition of $u_\eps^\pm$ in \eqref{eq:sub_sup_def} and the fact that $\rho_\eps^+(0,\cdot) = \rho_\eps^-(0,\cdot) = \rho^0$ in $\oO$. 

We start with the simpler case $\overline{u}_\eps^\pm$. Note that by \eqref{eq:root_shift_f}, there exists $c_1>0$ such that
\begin{equation} \label{eq:U_delta_increasing}
    U_\pm'(\delta) = -\frac{1}{f'(U_\pm(\delta))} > 2 c_1
\end{equation}
for all sufficiently small $\delta$. With the definition of $\overline{u}_\eps^\pm$ in \eqref{eq:u_bar_def}, this immediately implies the assertion for $\overline{u}_\eps^\pm$ in \eqref{eq:initial_ordering}. 

As for the assertion for $\widehat{u}_\eps^\pm$, it suffices to show that for every fixed $R>0$ and $\delta_0\ll 1$, if $\eps$ is sufficiently small, then one has
\begin{equation} \label{e:initial_profile_ordering}
    \uU(z-\eps^a, e, \eps q, \delta - \eps^\beta) \leq \uU(z+\eps^a, e, \eps q, \delta + \eps^\beta)
\end{equation}
for all $z \in \RR$, $R^{-1} < |e| < R$, $|q|<R$ and $|\delta|<\delta_0$. Take $Z^*>0$ such that $e^{-\lambda Z^*} \ll c_1$. We distinguish the cases whether $|z| \geq Z^*$ or $|z| < Z^*$. 

By \eqref{eq:U_delta_increasing} and Lemma~\ref{lem:CHL}, we have
\begin{equation*}
    \d_\delta \uU(z, e, \widetilde{q}, \delta) \geq c_1
\end{equation*}
whenever $|z| \geq Z^*$, $R^{-1}<|e| < R$, and $\widetilde{q}$ and $\delta$ sufficiently small. Thus, if $|z|\geq Z^* + 1$ and $\eps$ is sufficiently small, we have
\begin{equation} \label{eq:comparison_delta}
    \begin{split}
    \uU(z+\eps^a, e, \eps q, \delta+\eps^\beta) - \uU(z+\eps^a, e, \eps q,\delta) &\geq c_1 \eps^\beta\;,\\
    \uU(z-\eps^a, e, \eps q, \delta) - \uU(z-\eps^a, e, \eps q, \delta-\eps^\beta) &\geq c_1 \eps^\beta\;.
    \end{split}
\end{equation}
It remains to control the displacement in $z$. We first note that regularity in the $z$- and $q$-variables gives
\begin{equation*}
    \left| \uU(z+\eps^a,e,\eps q,\delta) - \uU(z-\eps^a,e,\eps q,\delta) - \bigl(\uU(z+\eps^a,e,0,\delta) -\uU(z-\eps^a,e,0,\delta) \bigr) \right|
    \lesssim \eps^{a+1}\;.
\end{equation*}
Moreover, by \cite[Theorem~2.39]{GK04}, $\uU$ is strictly increasing in $z$ when $q=0$. Hence, we have
\begin{equation} \label{eq:displace_z}
    \uU(z+\eps^a,e,\eps q,\delta) - \uU(z-\eps^a,e,\eps q,\delta)
    \geq - C \eps^{a+1}
\end{equation}
for some $C>0$. Hence, combining \eqref{eq:comparison_delta} and \eqref{eq:displace_z}, we see that \eqref{e:initial_profile_ordering} holds for $|z| \geq Z^*+1$ and all sufficiently small $\eps$ since $a+1>\beta$. 

On the bounded region $\{|z| < Z^*+2\}$, by the strict positivity of $\partial_z\uU$ for $q=0$, the compactness of the parameter set, and the continuity of $\d_z\uU$ in $q$, there exists $c_2>0$ such that
\begin{equation*}
    \d_z \uU(z,e, \widetilde{q},\delta) \geq c_2
\end{equation*}
for all $|z|\leq Z^*+2$ and $\widetilde{q}$ and $\delta$ sufficiently small. This immediately implies
\begin{equation*}
    \uU(z+\eps^a,e,\eps q,\delta+\eps^\beta) - \uU(z-\eps^a,e,\eps q,\delta-\eps^\beta) \geq 2 c_2\eps^a - C\eps^\beta
\end{equation*}
for all $|z| \leq Z^*+1$. Since $a<\beta$, it is positive, which proves \eqref{e:initial_profile_ordering} for $|z| \leq Z^*+1$. This proves the assertion for $\widehat u_\eps^\pm$ in \eqref{eq:initial_ordering}, and hence completes the proof of the lemma.
\end{proof}

\subsection{Convergence to the stochastic mean curvature flow -- proof of Theorem~\ref{thm:main}}
\label{sec:convergence_smcf}

We are now ready to give the proof of the main theorem. 

\begin{proof} [Proof of Theorem~\ref{thm:main}]
Choose the initial configurations $u_\eps^0$ as follows. Set
\begin{equation*}
  \widehat u_\eps^0(x) := \uU \Big( \frac{\rho^0(x)}{\eps}, \,\nabla\rho^0(x), \, \eps \nabla^2\rho^0(x), \, \eps\xi_\eps(0,x) \Big)\;.
\end{equation*}
Define \(\overline u_\eps^0\) as in \eqref{eq:u_bar_def}, with \(\delta_\eps^\pm(t,x)\) and \(\mathcal V_t^{\eps;\pm}\) replaced by \(\eps\xi_\eps(0,x)\) and \(\mathcal V_0\). Then define \(u_\eps^0\) as in \eqref{eq:sub_sup_def}, with \(\rho_\eps^\pm\), \(\widehat u_\eps^\pm\), and \(\overline u_\eps^\pm\) replaced by \(\rho^0\), \(\widehat u_\eps^0\), and \(\overline u_\eps^0\).

By the same argument as in the proof of the previous lemma, we have
\begin{equation*}
    u_\eps^{-}(0,x) \leq u_\eps(0,x) \leq u_\eps^{+}(0,x)
\end{equation*}
for all $x \in \Omega$ and all sufficiently small $\eps$. Hence, by Lemma~\ref{lem:sub_sup} (and noting that $\nN_\eps u_\eps = 0$), we have
\begin{equation*}
\begin{split}
\| u_{\eps}(t) - \chi_{\Sigma_t} \| &\leq \| u_{\eps}(t) - u_\eps^+(t) \| + \| u_\eps^+(t) - \chi_{\Sigma^{\eps,+}_t} \| +\| \chi_{\Sigma^{\eps,+}_t} - \chi_{\Sigma_t} \|\\
&\leq \| u_\eps^-(t) - u_\eps^+(t) \| + \| u_\eps^+(t) - \chi_{\Sigma^{\eps,+}_t}\| + \| \chi_{\Sigma^{\eps,+}_t} - \chi_{\Sigma_t}\|\\
&\leq \| u_\eps^-(t) - \chi_{\Sigma^{\eps,-}_t}\| + 2\| u_\eps^+(t) - \chi_{\Sigma^{\eps,+}_t}\| + 2\| \chi_{\Sigma^{\eps,+}_t} - \chi_{\Sigma_t}\| +\| \chi_{\Sigma^{\eps,-}_t} - \chi_{\Sigma_t} \|\;,
\end{split}
\end{equation*}
where all the norms appearing above are $\lL^2(\Omega)$. By Lemma~\ref{lem:CHL}, we have
\begin{align*}
    |u_\eps^\pm - 1|\leq |u_\eps^\pm - U_+(\eps\xi_\eps \pm\eps^\beta)| + |U_+(\eps\xi_\eps \pm \eps^\beta) -1| \lesssim e^{-\frac{\lambda \rho_\eps^\pm}{\eps}} + |\eps\xi_\eps \pm \eps^\beta|
\end{align*}
uniformly in $(t,x) \in [0,\tau] \times \oO$ where $\rho_\eps^\pm (t,x) > 0$. A similar bound also holds for points at which $\rho_\eps^\pm<0$. 
Since $|\eps\xi_\eps\pm\eps^\beta|\lesssim
\eps^{1-\kappa}+\eps^\beta$, integration in tubular coordinates gives,
uniformly in $t\in[0,\tau]$,
\begin{equation*}
    \bigl\|u_\eps^\pm(t)-\chi_{\Sigma_t^{\eps;\pm}}\bigr\|_{\lL^2(\Omega)}
    \lesssim \eps^{1/2}+\eps^{1-\kappa}\;.
\end{equation*}

Also, if the phase functions differ at $x\in\oO$, then $\rho_\eps^\pm(t,x)\rho(t,x)\leq0$, and hence
\begin{equation*}
    |\rho(t,x)|\leq |\rho_\eps^\pm(t,x)-\rho(t,x)|\leq\|\rho_\eps^\pm(t)-\rho(t)\|_{\lL^\infty(\oO)}\;.
\end{equation*}
Thus, the set on which they differ is contained in a tubular neighbourhood of $\Sigma_t$ of this width. Since $|\nabla\rho|=1$, its volume is bounded uniformly in $t$ by a constant times the width. By Theorem~\ref{thm:well_posedness_smcf}, there exists $\theta_\rho>0$ such that \eqref{eq:rho_convergence} holds. Then we have
\begin{equation*}
    \| \chi_{\Sigma^{\eps;\pm}_t} - \chi_{\Sigma_t} \|_{\lL^2(\Omega)}^2
    = 2\int_{\oO} \big| \chi_{\Sigma^{\eps;\pm}_t}(x) - \chi_{\Sigma_t}(x) \big| \,\md x
    \lesssim \|\rho_\eps^\pm(t)-\rho(t)\|_{\lL^\infty(\oO)}
    \lesssim \eps^{\theta_\rho}\;.
\end{equation*}
Combining these estimates with the comparison bound yields
\begin{equation*}
    \sup_{t\in[0,\tau]}
    \|u_\eps(t)-\chi_{\Sigma_t}\|_{\lL^2(\Omega)}
    \lesssim
    \eps^{1/2}+\eps^{1-\kappa}+\eps^{\theta_\rho/2}
    \lesssim \eps^\theta
\end{equation*}
for some
$\theta>0$.
This completes the proof of the theorem. 
\end{proof}

\section{Stability of the RPDE -- proof of Theorem~\ref{thm:stability}}
\label{sec:RPDE_stability_proof}

\subsection{Some preliminary lemmas}

Before we start the proof, we need the following lemmas. We first introduce some notations, which will be used throughout this section. Recall the solution $v$ to the classical PDE \eqref{e:RPDE_transformed}, and that
\begin{equation*}
    \Phi_t(x) := \big( x, v(t,x), \nabla v(t,x) \big) \in \RR^d \times \RR \times \RR^d\;.
\end{equation*}
Write
\begin{equation*}
    \widehat{\Theta}_t (x) := \Theta_t \big( \Phi_t(x) \big)\;, 
\end{equation*}
and $\widehat{X}_t = X_t \circ \Phi_t$, $\widehat{U}_t = U_t \circ \Phi_t$ and $\widehat{P}_t = P_t \circ \Phi_t$ for the corresponding components. For function
\begin{equation}
    F: \RR^d \times \RR \times \RR^d \rightarrow V\;, \qquad (x,u,p) \mapsto F(x,u,p)\;,
\end{equation}
we write $D_x F$, $D_u F$ and $D_p F$ for the Jacobian matrices for corresponding derivatives, and $D F = (D_x F, D_u F, D_p F)$. Also, we use the abbreviations
\[
    H_x:=D_xH\;, \qquad H_u:=D_uH\;, \qquad H_p:=D_pH\;.
\]

\begin{lem} \label{lem:RDE_stability}
Suppose $H \in \cC_b^7$. There exists $\theta>0$ such that for every realisation of the noise $\W$, we have
\begin{equation*}
    |X_t^\eps (x,u,p) - x| + |D_x X_t^\eps - \id| + |D_u X_t^\eps| + |D_p X_t^\eps| \lesssim t^\theta
\end{equation*}
uniformly over all $\eps$ and all $(x,u,p)$. Furthermore, for every multi-index $\gamma$, we have
\begin{equation*}
    \sup_{t \in [0,T]} \big| D_{\bm\theta}^\gamma \Theta_t^\eps - D_{\bm\theta}^\gamma \Theta_t \big| \lesssim d_\alpha (\W^\eps, \W)\;.
\end{equation*}
The proportionality constant depends on $T$ and $\gamma$ but is uniform over all $(x,u,p)$ and $\eps$. 
\end{lem}
\begin{proof}
    This is standard; see for example \cite[Theorem~11.6]{FV10}. 
\end{proof}

\begin{lem} (\cite[Lemma~4.2]{BKMZ20})
\label{lem:jacobian_identity}
For the characteristic flow $\Theta=(X,U,P)$ defined in \eqref{e:characteristic} with initial data $(X_0,U_0,P_0)$ and every $\bm\theta = (x,u,p) \in \RR^d \times \RR \times \RR^d$, we have

\begin{equation*}
   \mathsf D_{\bm\theta} U_t - P_t^{\sft} \, \mathsf D_{\bm\theta} X_t  = \big( u - x^{\sft} P_0 \big) \, \exp\!\left(\int_0^t H_u^\sft (\Theta_s )\,\md \W_s\right)\;.
\end{equation*}
Here $\mathsf D_{\bm\theta}$ denotes the directional derivative in the direction $\bm\theta$.
In particular, one has
\begin{equation} \label{eq:jacobian_identity_special}
    \nabla^\sft \widehat{U}_t = \widehat{P}_t^\sft \; \nabla^{\sft} \widehat{X}_t\;.
\end{equation}
\end{lem}
\begin{proof}
Differentiating the characteristic system \eqref{e:characteristic} with respect to the initial condition in the direction $\bm\theta$, with all coefficients evaluated at $\Theta_t$, we obtain
\begin{equation} \label{eq:system_DTheta}
    \md \mathsf D_{\bm\theta} X_t = - \langle D H_p^\sft, \, \mathsf D_{\bm\theta} \Theta_t \rangle \,\md \W\;, \quad
    \md \mathsf D_{\bm\theta} U_t =  \big\langle  D (H^\sft - p^\sft H_p^\sft), \, \mathsf D_{\bm\theta} \Theta_t  \big\rangle \,\md\W\;.
\end{equation}
For the terms on the right hand side of the expression for $\md \mathsf D_{\bm\theta} U_t$, we have
\begin{equation} \label{eq:expansion_DP}
    \begin{split}
    \langle  D H^\sft, \mathsf D_{\bm\theta} \Theta_t \rangle &= \mathsf D_{\bm\theta} X_t^\sft \cdot H_x^\sft + \mathsf D_{\bm\theta} U_t \cdot H_u^\sft + \mathsf D_{\bm\theta} P_t^\sft \cdot H_p^\sft\;,\\
    \langle  D (p^\sft H_p^\sft), \mathsf D_{\bm\theta} \Theta_t \rangle &=  \mathsf D_{\bm\theta} P_t^\sft \cdot H_p^\sft + P^\sft_t \langle  D H_p^\sft, \mathsf D_{\bm\theta} \Theta_t \rangle\;,
    \end{split}
\end{equation}
where both expressions above take values as row vectors of length $m$. Set $Z_t:=\mathsf D_{\bm\theta} U_t - P^\sft_t \mathsf D_{\bm\theta} X_t$. By the rough product rule, we have
\begin{equation*}
    \md Z_t = \md \mathsf D_{\bm\theta} U_t - (\md P_t)^\sft\mathsf D_{\bm\theta} X_t - P_t^\sft\,\md\mathsf D_{\bm\theta} X_t\;.
\end{equation*}
Combining \eqref{eq:system_DTheta}, the expressions from \eqref{eq:expansion_DP} as well as the equation for $P$ from \eqref{e:characteristic}, we get
\begin{equation*}
    \md Z_t = Z_t \, H_u^\sft (\Theta_t) \,\md \W\;,
\end{equation*}
which yields the unique solution
\begin{equation*}
    Z_t = Z_0 \; \exp\!\left(\int_0^t H_u^\sft(\Theta_s)\,\md \W_s\right)\;.
\end{equation*}
The first claim then follows since $Z_0 = \mathsf D_{\bm\theta} U_0 - P_0^\sft \mathsf D_{\bm\theta} X_0 = u-x^\sft P_0$. For the second claim, it suffices to show
\begin{equation} \label{eq:cancellation_identity_component}
    \d_k \widehat{U}_t = \widehat{P}_t^\sft \cdot \d_k \widehat{X}_t
\end{equation}
for every $k=1, \dots, d$. By chain rule, we have
\begin{equation*}
    \d_k \widehat{U}_t - \widehat{P}_t^\sft \cdot \d_k \widehat{X}_t = \big\langle ( D U_t)(\Phi_t), \, \d_k \Phi_t \big\rangle - P_t^\sft (\Phi_t) \, \big\langle ( D X_t)(\Phi_t), \, \d_k \Phi_t \big\rangle\;.
\end{equation*}
Note that for $\Phi_t (x) = \big(x, v(t,x), (\nabla v)(t,x) \big)$, the first two components of the quantity $\bm\theta=(x,u,p):=(\d_k \Phi_t)(x)$ are given by
\begin{equation*}
    x=e_k \qquad \text{and} \qquad u=(\d_k v)(t,x)\;. 
\end{equation*}
Also by definition, we have $P_0 \big( \Phi_t(x) \big) = (\nabla v)(t,x)$. Hence, we have the identity
\begin{equation*}
    u-x^\sft P_0 = (\d_k v)(t,x) - \langle (\nabla v)(t,x), \, e_k \rangle = 0\;. 
\end{equation*}
This proves the identity \eqref{eq:cancellation_identity_component} and hence completes the proof of the lemma. 
\end{proof}

\begin{lem} \label{lem:nabla_Theta_hat_stability}
    Suppose $\|v_\eps\|_{\lL_t^\infty \cC_x^2 ([0,\tau] \times \overline{\widetilde{\oO}})}$ is uniformly bounded. Then for $k=0,1$, we have
    \begin{equation*}
        \big\| \nabla^k \big(\widehat{\Theta}_t^\eps - \widehat{\Theta}_t^0 \big) \|_{\lL_{t,x}^\infty([0,\tau] \times \overline{\widetilde{\oO}})} \lesssim d_\alpha (\W^\eps, \W) + \|v_\eps - v\|_{\lL_t^\infty \cC_x^2 ([0,\tau] \times \overline{\widetilde{\oO}})}\;.
    \end{equation*}
\end{lem}
\begin{proof}
    We prove the case $k=1$ only. Note that
    \begin{equation*}
        \nabla \widehat{\Theta}_t^\eps = \nabla \big( \Theta_t^\eps \circ \Phi_t^\eps \big) = (D_x \Theta_t^\eps)(\Phi_t^\eps) + (D_u \Theta_t^\eps)(\Phi_t^\eps) \cdot \nabla v_\eps + (D_p \Theta_t^\eps)(\Phi_t^\eps) \cdot \nabla^2 v_\eps\;.
    \end{equation*}
    The desired bound then follows from Lemma~\ref{lem:RDE_stability} and the assumption that $\|v_\eps\|_{\lL_t^\infty \cC_x^2([0,\tau] \times \overline{\widetilde{\oO}})}$ is uniformly bounded. 
\end{proof}

\begin{lem}\label{lem:nablaX_id}
For every $r>0$, there exists $\tau_\eps > 0$ such that for every $t\in[0,\tau_\eps]$ and $x \in \overline{\widetilde{\oO}}$, we have
\begin{equation} \label{eq:X_hat_approximate_identity}
    |\widehat{X}^\eps_t(x)-x|\leq r\;, \qquad
        \|\nabla^\sft\widehat{X}^\eps_t(x) - \id \| \leq r\;.
\end{equation}
As a result, for sufficiently small $r$, $\widehat{X}^\eps_t$ is injective for $t \in [0, \tau_\eps]$, and $\widehat{X}^\eps_t (\d \widetilde{\oO}) = \d \oO_t^\eps$.

If in addition, $\|v_\eps\|_{\lL_t^\infty \cC_x^2([0,T]\times 
\overline{\widetilde{\oO}})}$ is uniformly bounded in $\eps$, then for every $\oO \Subset \widetilde{\oO}$, there exists $\tau>0$ independent of $\eps$ such that
\begin{equation*}
    \oO \Subset \big\{ \widehat{X}_t^\eps(y): \; y \in \widetilde{\oO} \big\}
\end{equation*}
for all $t \in [0,\tau]$ and all sufficiently small $\eps$. Finally, if in addition,  $\|\nabla^3 v_0\|_{\lL_{t,x}^\infty([0,\tau] \times \overline{\oO_1})}$ is finite for some $\oO\Subset\oO_1\Subset\widetilde\oO$, then the inverse $\widehat{S}_t^\eps: \overline{\oO} \rightarrow \RR^d$ satisfies
\begin{equation*}
    \begin{split}
    \big\|\widehat{S}_t^\eps - \widehat{S}_t^0 \big\|_{\lL^\infty(\overline{\oO})} &\lesssim \big\| \widehat{X}_t^\eps - \widehat{X}_t^0 \big\|_{\lL^\infty(\widetilde{\oO})}\;,\\
    \big\|\nabla^\sft \big( \widehat{S}_t^\eps - \widehat{S}_t^0 \big) \big\|_{\lL^\infty(\overline{\oO})} &\lesssim \big\| \widehat{X}_t^\eps - \widehat{X}_t^0 \big\|_{\lL^\infty(\widetilde{\oO})} + \big\|\nabla^\sft \big( \widehat{X}_t^\eps - \widehat{X}_t^0 \big) \big\|_{\lL^\infty(\widetilde{\oO})}\;.
    \end{split}
\end{equation*}
The proportionality constants are independent of $t \in [0,\tau]$. 
\end{lem}
\begin{proof}
Recall that $\widehat X^\eps_t(x)=X_t^\eps(\Phi^\eps_t(x))$. By the convention above and the chain rule,
\begin{equation*}
    \nabla^\sft\widehat X^\eps_t(x)
    =D_x X^\eps_t(\Phi^\eps_t(x))
    +D_u X^\eps_t(\Phi^\eps_t(x))(\nabla v_\eps(t,x))^\sft
    +D_p X^\eps_t(\Phi^\eps_t(x))\nabla^2 v_\eps(t,x)\;.
\end{equation*}
Combining the boundedness of $\nabla v_\eps$ and $\nabla^2 v_\eps$ with Lemma~\ref{lem:RDE_stability}, we may shrink the time interval to $[0, \tau_\eps]$ to get \eqref{eq:X_hat_approximate_identity}.
Since $\widetilde\oO$ is a bounded, connected Lipschitz domain, \eqref{eq:X_hat_approximate_identity} implies that $\widehat X_t^\eps$ is injective when $r$ is sufficiently small. Moreover, the continuous injection $\widehat X_t^\eps:\overline{\widetilde\oO}\to\RR^d$ is a homeomorphism onto its image, and invariance of domain therefore gives $\widehat X_t^\eps(\partial\widetilde\oO)=\partial\oO_t^\eps$.

As for the second claim, since $\oO \Subset \widetilde\oO$, there exists $r_{\oO}>0$ such that for every $x \in\overline{\oO}$, $\overline{B(x,r_{\oO})} \subset \widetilde\oO$. Consider the map $T_x:\overline{B(x,r_{\oO})}\to\RR^d$ defined by
\begin{equation*}
    T_x(y):= x - \bigl( \widehat X^\eps_t(y) - y \bigr)
\end{equation*}
Then
\begin{equation*}
    |T_x(x)-x| = \bigl| \widehat X^\eps_t(x) - x\bigr| \leq \frac{r_{\oO}}{2}\;, \quad \|\nabla^\sft T_x (y)\| = \bigl\|\nabla^\sft \widehat X^\eps_t(y) - \id \bigr\| \leq \frac{1}{2}\;,
\end{equation*}
where we have assumed that $t$ is small enough. As a result, $T_x:\overline{B(x,r_{\oO})}\to\overline{B(x,r_{\oO})}$ is a $\frac{1}{2}$-contraction. By the Banach fixed point theorem, there exists $y\in \overline{B(x,r_{\oO})}$ such that $T_x(y) =y$, which means $x = \widehat X^\eps_t(y)$ is in the image of $\widehat X^\eps_t(\cdot)$.
Here the time interval can be chosen independently of $\eps$ because the $\cC_x^2$ norms of $v_\eps$ and the rough-path norms are uniformly bounded.

To control the difference $\widehat{S}_t^\eps - \widehat{S}_t^0$, we first note that since $|(\nabla^\sft \widehat{X}_t^0)(y) - \id| < r$ for a sufficiently small $r$ for all $t \in [0,\tau]$ and $y \in \widetilde{\oO}$ and $\widetilde\oO$ is a bounded, connected Lipschitz domain, we have
\begin{equation*}
    \big| \widehat{S}_t^\eps(x) - \widehat{S}_t^0(x) \big| \lesssim \big| \widehat{X}_t^0 \big( \widehat{S}_t^\eps(x) \big) - \widehat{X}_t^0 \big( \widehat{S}_t^0(x) \big) \big| = \big| \widehat{X}_t^0 \big( \widehat{S}_t^\eps(x) \big) - x \big|\;,
\end{equation*}
where the proportionality constant is uniform in $x \in \oO$. On the other hand, we also have
\begin{equation*}
    \big| \widehat{X}_t^0 \big( \widehat{S}_t^\eps(x) \big) - x \big| = \big| \widehat{X}_t^0 \big( \widehat{S}_t^\eps(x) \big) - \widehat{X}_t^\eps \big( \widehat{S}_t^\eps(x) \big) \big| \lesssim \| \widehat{X}_t^\eps - \widehat{X}_t^0 \|_{\lL^\infty(\widetilde{\oO})}
\end{equation*}
uniformly in $t \in [0,\tau]$. Combining the above proves the bound for $\widehat{S}_t^\eps - \widehat{S}_t^0$. As for the difference of the gradients, we have
\begin{equation*}
    \nabla^\sft \widehat{S}_t^\eps = \big( \nabla^\sft \widehat{X}_t^\eps (\widehat{S}_t^\eps) \big)^{-1}
\end{equation*}
as matrix inverse. Now, using the identity
\begin{equation*}
    A^{-1} - B^{-1} = A^{-1} (B-A) B^{-1}
\end{equation*}
as well as $\nabla^\sft \widehat{X}_t^\eps$ is close to identity (and hence has a uniformly bounded inverse), we get 
\begin{equation*}
    \begin{split}
     \big| \nabla^\sft \widehat{S}_t^\eps - \nabla^\sft \widehat{S}_t^0 \big| &\lesssim \big| \nabla^\sft \widehat{X}_t^\eps (\widehat{S}_t^\eps) - \nabla^\sft \widehat{X}_t^0 (\widehat{S}_t^0) \big|\\
     &\leq \big| \nabla^\sft \widehat{X}_t^\eps (\widehat{S}_t^\eps) - \nabla^\sft \widehat{X}_t^0 (\widehat{S}_t^\eps) \big| + \big| \nabla^\sft \widehat{X}_t^0 (\widehat{S}_t^\eps) - \nabla^\sft \widehat{X}_t^0 (\widehat{S}_t^0) \big|\;.
     \end{split}
\end{equation*}
Note that we have
\begin{equation*}
    \|\nabla^2 \widehat{X}_t^0\|_{\lL^\infty(\overline{\oO_1})} \lesssim 1+\|v_0\|_{\lL_t^\infty\cC_x^3([0,\tau]\times\overline{\oO_1})}
\end{equation*}
on $[0,\tau]$ and after decreasing $\tau$ if necessary, we also have $\widehat S_t^\eps(\overline{\oO})\subset\oO_1$ for all sufficiently small $\eps\geq 0$ and $t\in[0,\tau]$. Combining these  gives the desired bound for $\|\nabla^\sft(\widehat S_t^\eps-\widehat S_t^0)\|_{\lL^\infty(\overline{\oO})}$. This completes the proof of the lemma.
\end{proof}

\subsection{Proof of Lemma~\ref{lem:transform}}

\begin{proof} [Proof of Lemma~\ref{lem:transform}]
By Lemma~\ref{lem:nablaX_id}, after decreasing $\tau$ if necessary, $\widehat{X}_t$ is injective on $\overline{\widetilde{\oO}}$ for $t\in[0,\tau]$. We set $\widehat{S}_t := \widehat{X}_t^{-1}$ be the inverse, and define
\[
    \rho(t,x) := \widehat{U}_t \big( \widehat{S}_t (x) \big)\;,\quad x\in\overline{\oO}_t:= \widehat{X}_t (\overline{\widetilde{\oO}})\;.
\]
Since $\widehat{S}_t$ is the inverse of $\widehat{X}_t$, by Lemmas~\ref{lem:composition_2} and~\ref{lem:inverse_rough_path}, one has
\begin{equation*}
    \begin{split}
    \widehat{U}_t \big( \widehat{S}_t (x) \big) - \widehat{U}_s \big( \widehat{S}_s (x) \big) = &\phantom{11} \eE_s \big( \widehat{S}_s(x) \big) \, (W_t - W_s) + \langle \eE_s' \big( \widehat{S}_s(x) \big), \, \WW_{s,t} \rangle\\
    &+ \fF_s \big( \widehat{S}_s(x) \big) \cdot (t-s) + \oO \big(|t-s|^{1+\gamma} \big)
    \end{split}
\end{equation*}
for some $\gamma>0$. In other words, both the rough path integrand and drift are functions of $\widehat{S}$. To compute $(\eE, \eE')$ and $\fF$, we need information on the controlled integrand $(A,A')$ and drift $B$ for $\widehat{U}$ as well as the pair $(C, C')$ and $D$ for $\widehat{S}$. For $\widehat{U}$, we have
\begin{equation*}
    \md \widehat{U}_t = \big(A_t, A_t' \big) \, \md \W_t + B_t \,\md t\;,
\end{equation*}
where
\begin{equation} \label{eq:rp_integrand_Phat}
    \begin{split}
    &A_t = H^\sft \big( \widehat{\Theta}_t \big) - P^\sft \big( \widehat{\Theta}_t \big) H_p^\sft \big( \widehat{\Theta}_t \big)\;, \qquad B_t = (D U_t)(\Phi_t) \cdot \dot{\Phi}_t\;,\\
    &A_t' = D\big((H-H_p p)^\sft\big)(\widehat{\Theta}_t) \cdot \vV\big(\widehat{\Theta}_t\big)\;.
    \end{split}
\end{equation}
Here
\[
    \vV(x,u,p):=
    \begin{pmatrix}
        -H_p^\sft\\
        (H-H_p p)^\sft\\
        H_x^\sft+pH_u^\sft
    \end{pmatrix}(x,u,p)\in\RR^{(2d+1)\times m}\;.
\]
All the above expressions make sense as matrix multiplication. As for $\widehat{X}$, we have
\begin{equation*}
    \md \widehat{X}_t = \big(\widetilde{A}_t, \widetilde{A}_t' \big) \, \md \W_t + \widetilde{B}_t \, \md t\;,
\end{equation*}
where
\begin{equation} \label{eq:rp_integrand_Xhat}
    \widetilde{A}_t = - H_p^\sft \big( \widehat{\Theta}_t \big)\;, \quad \widetilde{A}_t' = - D(H_p^\sft)\big( \widehat{\Theta}_t \big) \cdot \vV \big( \widehat{\Theta}_t \big)\;, \quad \widetilde{B}_t = (D X_t)(\Phi_t) \cdot \dot{\Phi}_t\;.
\end{equation}
The expression for $\widetilde{A}_t'$ is in the sense that its $k$-th component is given by
\begin{equation} \label{eq:rp_integrand_Xhat_derivative}
    \big( \widetilde{A}_t' \big)^{(k)} = - D\big(H_{p_k}^\sft\big)\big( \widehat{\Theta}_t \big) \cdot \vV \big( \widehat{\Theta}_t \big)
\end{equation}
as an $m \times m$ matrix. Again, all the above operations make sense as matrix multiplications. Hence, by Lemma~\ref{lem:inverse_rough_path}, we have
\begin{equation*}
    \md \widehat{S}_t = \big( C_t (\widehat{S}_t), \, C_t'(\widehat{S}_t) \big) \, \md \W_t + D_t(\widehat{S}_t) \, \md t\;,
\end{equation*}
where
\begin{equation} \label{eq:rp_integrand_Shat}
    \begin{split}
    &C_t = - \big( \nabla^\sft \widehat{X}_t \big)^{-1} \widetilde{A}_t = \big( \nabla^\sft \widehat{X}_t \big)^{-1} H_p^\sft \big( \widehat{\Theta}_t \big)\;,\\
    &D_t = - \big( \nabla^\sft \widehat{X}_t \big)^{-1} \widetilde{B}_t = - \big( \nabla^\sft \widehat{X}_t \big)^{-1} (D X_t)(\Phi_t) \cdot \dot{\Phi}_t\;,
    \end{split}
\end{equation}
and $C_t'$ is defined via the relation
\begin{equation*} 
    \langle C_t', \WW_{s,t} \rangle = - \big( \nabla^\sft \widehat{X}_t \big)^{-1} \Big\langle \widetilde{A}_t'  + 2 \, \Sym \big( \nabla^\sft \widetilde{A}_t^\sft \cdot C_t \big) + C_t^\sft \, \nabla^2 \widehat{X}_t \, C_t, \; \WW_{s,t} \Big\rangle\;.
\end{equation*}
Again, by the identity \eqref{eq:jacobian_identity_special}, we have
\begin{equation} \label{eq:rp_integrand_Shat_derivative}
    \begin{split}
    &\phantom{111}\nabla^\sft \widehat{U}_t \langle C_t', \WW_{s,t} \rangle\\
    &= - \sum_{k=1}^{d} \widehat P_t^{(k)} \Big\langle \big(\widetilde{A}_t'\big)^{(k)}  + 2 \, \Sym \big( \nabla^\sft \widetilde{A}_t^{(k)} \cdot C_t \big) + C_t^\sft \, \nabla^2 \widehat{X}^{(k)}_t \, C_t, \; \WW_{s,t} \Big\rangle\;.
    \end{split}
\end{equation}
We are now ready to compute the controlled integrand $(\eE_t, \eE_t')$. For $\eE_t$, it follows directly from Lemma~\ref{lem:composition_2}, the form of $A$ and $C$ in \eqref{eq:rp_integrand_Phat} and \eqref{eq:rp_integrand_Shat}, and the identity $\nabla^\sft \widehat{U}_t \big( \nabla^\sft \widehat{X}_t \big)^{-1} = \widehat{P}_t^\sft$ from Lemma~\ref{lem:jacobian_identity} that
\begin{equation*}
    \eE_t = A_t + \nabla^\sft \widehat{U}_t \cdot C_t = H^\sft \big( \widehat{\Theta}_t \big)\;.
\end{equation*}
For $\eE'$, by Lemma~\ref{lem:composition_2} and \eqref{eq:rp_integrand_Shat_derivative}, we have
\begin{equation*}
    \langle \eE_t', \, \WW_{s,t} \rangle = \Big\langle A_t' + 2 \, \Sym \big( \nabla^\sft A_t^\sft \cdot C_t \big) + C_t^\sft \, \nabla^2 \widehat{U}_t \, C_t, \; \WW_{s,t} \Big\rangle + \nabla^\sft \widehat{U}_t \langle C_t', \, \WW_{s,t} \rangle\;.
\end{equation*}
Combining with the expression \eqref{eq:rp_integrand_Shat_derivative}, we can write
\begin{equation*}
    \eE_t' = \eE_{t,1}' + \eE_{t,2}' + \eE_{t,3}'
\end{equation*}
with
\begin{equation*}
    \begin{split}
    &\eE_{t,1}' = A_t' + \sum_k \widehat P_t^{(k)} \big(\widetilde{A}_t'\big)^{(k)}\;, \quad \eE_{t,2}' = 2 \, \Sym \Big( \big( \nabla^\sft A_t^\sft - \sum_k \widehat P_t^{(k)} \nabla^\sft \widetilde{A}_t^{(k)} \big) \; C_t \Big)\;,\\
    &\eE_{t,3}' = C_t^\sft \Big( \nabla^2 \widehat{U}_t - \sum_k \widehat P_t^{(k)} \nabla^2 \widehat{X}^{(k)}_t  \Big) C_t\;,
    \end{split}
\end{equation*}
where the sums above are all over $k$ from $1$ to $d$. For $\eE_{t,1}$, according to the expressions of $A'$ and $\widetilde{A}'$ in \eqref{eq:rp_integrand_Phat} and \eqref{eq:rp_integrand_Xhat_derivative}, we have
\begin{equation} \label{eq:rp_integrand_control_1}
    \eE_{t,1}' = - H_x \big( \widehat{\Theta}_t \big) H_p^\sft \big( \widehat{\Theta}_t \big) + H_u \big( \widehat{\Theta}_t \big) \Big( H^\sft \big( \widehat{\Theta}_t \big) - P^\sft \big( \widehat{\Theta}_t \big) H_p^\sft \big( \widehat{\Theta}_t \big) \Big)\;.
\end{equation}
For $\eE_{t,2}'$, by the expressions for $A_t$ and $\widetilde{A}_t$ in \eqref{eq:rp_integrand_Phat} and \eqref{eq:rp_integrand_Xhat}, we have
\begin{equation*}
    \begin{split}
    &\phantom{111}\nabla^\sft A_t^\sft - \sum_k \widehat P_t^{(k)} \nabla^\sft (\widetilde{A}_t)^{(k)}\\
    &= \nabla^{\sft} \Big( H ( \widehat{\Theta}_t ) - H_p ( \widehat{\Theta}_t ) \widehat P_t \Big) + \sum_{k} \widehat P_t^{(k)} \nabla^\sft \Big( H_{p_k} \big( \widehat{\Theta}_t \big) \Big)\\
    &= \nabla^\sft \Big( H(\widehat{\Theta}_t) \Big) - H_p \big(\widehat{\Theta}_t \big) \nabla^\sft \Big( P (\widehat{\Theta}_t) \Big) = H_x \big( \widehat{\Theta}_t \big) \nabla^\sft \widehat{X}_t + H_u \big( \widehat{\Theta}_t \big) \nabla^\sft \widehat{U}_t\;.
    \end{split}
\end{equation*}
Together with the form for $C_t$ from \eqref{eq:rp_integrand_Shat}, we have
\begin{equation} \label{eq:rp_integrand_control_2}
    \eE_{t,2}' = \big( H_x + H_u P^\sft \big) H_p^\sft + H_p \big( H_x^\sft + P H_u^\sft \big)\;,
\end{equation}
where all quantities are evaluated at $\widehat{\Theta}_t$. Finally for $\eE_{t,3}'$, we first note that by differentiating \eqref{eq:jacobian_identity_special}, we have
\begin{equation*}
    \nabla^2 \widehat{U}_t - \sum_k \widehat P_t^{(k)} \nabla^2 \widehat{X}_t^{(k)} = \nabla \widehat{X}_t^\sft \cdot \nabla^\sft \widehat{P}_t\;.
\end{equation*}
Again, combining it with the expression for $C_t$ in \eqref{eq:rp_integrand_Shat}, we have
\begin{equation} \label{eq:rp_integrand_control_3}
    \eE_{t,3}' = H_p \big( \widehat{\Theta}_t \big) \cdot \nabla^\sft \widehat{P}_t \cdot \big( \nabla^\sft \widehat{X}_t \big)^{-1} \cdot H_p^\sft \big( \widehat{\Theta}_t \big)\;.
\end{equation}
Combining \eqref{eq:rp_integrand_control_1}, \eqref{eq:rp_integrand_control_2} and \eqref{eq:rp_integrand_control_3}, we get
\begin{equation*}
    \eE_t' = H_u H^\sft + H_p \Big( H_x^\sft + P H_u^\sft + \nabla^\sft \widehat{P}_t \cdot \big( \nabla^\sft \widehat{X}_t \big)^{-1} \cdot H_p^\sft \Big)\;,
\end{equation*}
where all the quantities involving $H$ and the component function $P$ are evaluated at $\widehat{\Theta}_t$. 

Next we show that $\eE_t'$ given above is indeed the ``derivative" process for $\hH(\rho)$. By the form of $\hH(\rho)$, we need also to control the fluctuation of $\nabla \rho$. Taking $\nabla^\sft$ on $\rho(t,\cdot) = \widehat{U}_t \circ \widehat{S}_t$ and using that $\nabla^\sft \widehat{S}_t = (\nabla^\sft \widehat{X}_t)^{-1}(\widehat{S}_t)$ as matrix inverse, we have
\begin{equation*}
    \nabla^\sft \rho (t, \cdot) = (\nabla^\sft \widehat{U}_t)(\widehat{S}_t) \cdot (\nabla^{\sft} \widehat{X}_t)^{-1}(\widehat{S}_t) = \widehat{P}_t^{\sft} (\widehat{S}_t)\;,
\end{equation*}
where we also used \eqref{eq:jacobian_identity_special} in the last equality. Taking another gradient on both sides above, we get
\begin{equation*}
    \nabla^2 \rho (t, \cdot) = (\nabla^\sft \widehat{P}_t)(\widehat{S}_t) \cdot (\nabla^\sft \widehat{X}_t)^{-1}(\widehat{S}_t)\;.
\end{equation*}
Now, applying Lemma~\ref{lem:composition_1} to $\nabla \rho (t, \cdot) = \widehat{P}_t \circ \widehat{S}_t$, we see $\nabla \rho$ is controlled by $W$ with derivative process
\begin{equation*}
    (\nabla \rho)'_t = H_x^\sft \big( \widehat{\Theta}_t \big) + P \big( \widehat{\Theta}_t \big) H_u^\sft \big( \widehat{\Theta}_t \big) + (\nabla^\sft \widehat{P}_t) \cdot (\nabla^{\sft} \widehat{X}_t)^{-1} \cdot H_p^\sft\;,
\end{equation*}
where all quantities on the right hand side above are evaluated at (composed with) $\widehat{S}_t$. Hence, combining all the above, we see the process $H \big(x, \rho(t,x), \nabla \rho(t,x) \big)$ (as a process in $t$) is controlled by $W$ with the Gubinelli derivative $(\hH\rho)'_t$ precisely given by $\eE_t'$. 

We now compute the drift term $\fF_t$. Again by Lemma~\ref{lem:composition_2}, we have
\begin{equation*}
    \fF_t = B_t + \nabla^\sft \widehat{U}_t \cdot D_t = \Big( (D U_t)(\Phi_t) - P_t^\sft(\Phi_t) \cdot (D X_t)(\Phi_t) \Big) \cdot \dot{\Phi}_t\;.
\end{equation*}
By Lemma~\ref{lem:jacobian_identity} with $\bm\theta = \dot{\Phi}_t = (0, \d_t v, \d_t \nabla v)$, we get
\begin{equation*}
    \begin{split}
    \fF_t &= \d_t v \cdot \exp \Big( \int_0^t H_u^\sft \big( \Theta_s(\Phi_t(x)) \big) \,\md \W_s \Big)\\
    &= \widetilde{F}\big(t, \Phi_t(x), \nabla^2 v \big) \cdot \exp \Big( \int_0^t H_u^\sft \big( \Theta_s(\Phi_t(x)) \big) \,\md \W_s \Big)\;.
    \end{split}
\end{equation*}
By definition of $\widetilde{F}$ in Lemma~\ref{lem:transform}, we get
\begin{equation*}
    \fF_t \big( \widehat{S}_t(x) \big) = F \Big( t, \; \widehat{\Theta}_t\big(\widehat{S}_t(x)\big), \; \big( \nabla^\sft \widehat{P}_t \big)\big( \widehat{S}_t(x) \big) \cdot \big( \nabla^\sft \widehat{X}_t \big)^{-1} \big( \widehat{S}_t(x) \big) \Big)\;.
\end{equation*}
Note that
\begin{equation*}
    \begin{split}
    \nabla^\sft \widehat{P}_t &= D_x P_t + D_u P_t \cdot \nabla^\sft v + D_p P_t \cdot \nabla^2 v\;,\\
    \nabla^\sft \widehat{X}_t &= D_x X_t + D_u X_t \cdot \nabla^\sft v + D_p X_t \cdot \nabla^2 v\;,
    \end{split}
\end{equation*}
we have thus verified the pointwise relation for $\rho$ in Definition~\ref{def:RPDE}. 

It remains to check the initial and boundary conditions. At $t=0$, we have $\widehat{S}_0 = \widehat{X}_0 = \id$, and hence
\begin{equation*}
    \widehat{U}_0 \big( \widehat{S}_0 (x) \big) = U_0 \big( \Phi_0 (x) \big) = \rho^0(x)\;,
\end{equation*}
    which agrees with the initial condition. As for the boundary condition, by Lemma~\ref{lem:nablaX_id}, for every $x \in \d \oO_t$, there exists $y \in \d \widetilde{\oO}$ such that $\widehat{X}_t(y) = x$. We then have
\begin{equation*}
    \rho(t,x) = \widehat{U}_t \big( \widehat{S}_t (x) \big) = \widehat{U}_t(y)\;, \qquad \nabla \rho (t,x) = \widehat{P}_t \big( \widehat{S}_t (x) \big) = \widehat{P}_t (y)\;.
\end{equation*}
This gives
\begin{equation*}
    G \big(x, \rho(t,x), \nabla \rho (t,x) \big) = G \big( \widehat{\Theta}_t(y) \big) = \widetilde{G}\big(t, \Phi_t(y) \big) = 0
\end{equation*}
for $x \in \d \oO_t$. Finally, the remaining regularity assertions follow from Lemmas~\ref{lem:composition_2} and \ref{lem:inverse_rough_path}, applied to the compositions above and to their spatial derivatives. Moreover,
\[
    \nabla^3\rho=\left[
    \nabla^\sft\left(
    \nabla^\sft\widehat P_t
    \bigl(\nabla^\sft\widehat X_t\bigr)^{-1}
    \right)
    \bigl(\nabla^\sft\widehat X_t\bigr)^{-1}
    \right]\circ\widehat S_t\;,
\]
and the same composition estimates give $\left(\nabla^2 \rho,\nabla^2(\hH\rho)\right)\in\lL^\infty_{x,loc}\dD^{2\alpha}_{W}$ and $\nabla^3\rho\in\lL_{x,loc}^\infty\cC_t^{\alpha}$. Differentiating the reconstruction formulas once more and applying the same composition and inverse-map estimates gives the stated bound for $\|\nabla^4\rho\|$.

We now turn to stability of the solution, assuming $\|v_\eps\|_{\lL_t^\infty \cC_x^2}$ being uniformly bounded. Fix $\oO \Subset \widetilde{\oO}$. By Lemma~\ref{lem:nablaX_id}, there exists $\tau>0$ independent of $\eps$ such that
\begin{equation*}
    \begin{split}
    \rho_\eps (t,x) &= \widehat{U}_t^\eps \big( \widehat{S}_t^\eps (x) \big)\;, \quad \nabla \rho_\eps (t,x) = \widehat{P}_t^\eps \big( \widehat{S}_t^\eps (x) \big)\;,\\
    \nabla^2 \rho_\eps (t,x) &= \nabla^\sft \widehat{P}_t^\eps \big( \widehat{S}_t^\eps (x) \big) \cdot  \Big( \nabla^\sft \widehat{X}_t^\eps \big( \widehat{S}_t^\eps (x) \big) \Big)^{-1}\;.
    \end{split}
\end{equation*}
are all defined on $[0,\tau] \times \oO$ for all sufficiently small $\eps$. The desired bound \eqref{eq:PDE_stability} then follows from Lemmas~\ref{lem:RDE_stability}, \ref{lem:nabla_Theta_hat_stability} and~\ref{lem:nablaX_id}. This completes the proof of Lemma~\ref{lem:transform}. 
\end{proof}

\subsection{Stability of the deterministic PDE}

For $R, T>0$ and $g \in \cC^{2+2\alpha}(\overline{\widetilde{\oO}})$, define the set $\yY_T^{g;R}$ of spacetime functions by
\begin{equation*}
    \yY_T^{g;R} := \Big\{ v \in \cC_{\pa}^{2+2\alpha}\big([0,T] \times \overline{\widetilde{\oO}} \big): v(0,\cdot)=g,\ \|v-g\|_{\cC_{\pa}^{2+2\alpha}([0,T] \times \overline{\widetilde{\oO}})} \leq R \Big\}\;.
\end{equation*}
Here, $g$ is regarded as a spacetime function which is constant in time.

\begin{lem}\label{lem:nonlinear_PDE}
Fix $\alpha \in (0,\frac{1}{2})$. Let $\widetilde\oO \subset \RR^d$ be an open bounded domain with smooth boundary. Fix $R_0>0$ and $(\bar{u}, \bar{p}, \bar{q}) \in \RR \times \RR^d \times \RR^{d \times d}$. Consider the domains
\begin{equation*}
    Q_T = [0,T] \times \overline{\widetilde{\oO}}\times B\bigl((\bar u,\bar p,\bar q),R_0\bigr)\;, \qquad S=\overline{\widetilde{\oO}}\times B\bigl((\bar u,\bar p),R_0\bigr)\;. 
\end{equation*}
Suppose $\{\widetilde{F}_\eps\}$, $\{\widetilde{G}_\eps\}$ and $\{g_\eps\}$ are families of real-valued functions on $Q_T, S$ and $\overline{\widetilde{\oO}}$ and satisfy the followings:
\begin{enumerate}
\item $\widetilde{F}_\eps$ is differentiable in $q$ and satisfies the uniform ellipticity condition: there exists $\lambda>0$ such that
\begin{equation} \label{e:condition_F_elliptic}
\sum_{i,j=1}^d \d_{q_{ij}} \widetilde{F}_\eps(t,x,u,p,q) \, \eta_i \, \eta_j \geq \lambda |\eta|^2
\end{equation}
uniformly over $(t,x,u,p,q) \in Q_T$, $\eta \in \RR^d$ and $\eps \in [0,1]$. 

\item Each $\widetilde{F}_\eps$ is twice differentiable in $(u,p,q)$, and that the derivatives up to order two belongs to $\cC_\pa^{2\alpha}$ in $(t,x)$. More precisely, we have
\begin{equation} \label{e:condition_F}
|\!|\!| \widetilde{F}_\eps |\!|\!|_{\mathcal F,\mathrm{PDE}}:=
\sup_{\substack{(u,p,q) \in B((\bar u,\bar p,\bar q),R_0)\\ |\beta|\leq 2}}
    \bigl\| D_{(u,p,q)}^\beta \widetilde F_\eps(\cdot\,,\cdot\,,u,p,q) \bigr\|_{\cC_{\pa}^{2\alpha}([0,T]\times\overline{\widetilde{\oO}})}
<+\infty\;.
\end{equation}

\item $\widetilde G_\eps$ is three times differentiable in $(u,p)$ and each derivative $D_{(u,p)}^\beta \widetilde G_\eps$ with $|\beta|\leq3$ belongs to $\cC^{1+2\alpha}(\overline{\widetilde{\oO}})$ in $x$, uniformly with respect to $(u,p)$. More precisely, we have
\begin{equation}\label{e:condition_G}
|\!|\!| \widetilde G_\eps |\!|\!|_{\mathcal G}:=
\sup_{\substack{(u,p) \in B((\bar u,\bar p),R_0) \\ |\beta|\leq 3}}
    \bigl\| D_{(u,p)}^\beta \widetilde G_\eps(\cdot\,,u,p) \bigr\|_{\cC^{1+2\alpha}(\overline{\widetilde{\oO}})} <+\infty\;.
\end{equation}

\item $g_\eps \in \cC^{2+2\alpha} (\overline{\widetilde{\oO}})$ satisfies the compatibility condition
\begin{equation*}
\widetilde G_\eps\bigl(x,g_\eps(x),\nabla g_\eps(x)\bigr)=0\;,\quad x\in\partial\widetilde\oO\;,
\end{equation*}
and the uniform non–tangentiality condition: there exists $c>0$ such that
\begin{equation} \label{e:condition_G_non_tangent}
\inf_{x \in \d \widetilde{\oO},\eps\in[0,1]} \Big| \sum_{i=1}^d \d_{p_i} \widetilde{G}_\eps \big( x,g_\eps(x),\nabla g_\eps(x) \big) \nu_i(x) \Big| \geq c\;,
\end{equation}
where $\nu(x)$ denotes the unit outward normal at $x \in \partial \widetilde\oO$.
Moreover, the range of $(g_\eps,\nabla g_\eps,\nabla^2 g_\eps)$ is contained in the ball of radius $R_0/2$ centered at $(\bar u,\bar p,\bar q)$ . 
\end{enumerate}
Consider the second order fully nonlinear parabolic initial-boundary value problem
\begin{equation}\label{e:nonlinear_PDE}
\left\{
\begin{aligned}
&\d_t u_\eps(t,x) = \widetilde F_\eps \bigl(t,x,u_\eps(t,x),\nabla u_\eps(t,x),\nabla^2 u_\eps(t,x) \bigr)\;, \quad &&t\in [0,T]\;, x\in\overline{\widetilde{\oO}}\;,\\
&\widetilde G_\eps \bigl(x,u_\eps(t,x),\nabla u_\eps(t,x)\bigr) = 0\;,&& t\in[0,T]\;, x\in\partial\widetilde\oO\;,\\
&u_\eps(0) = g_\eps\;, && x\in\overline{\widetilde{\oO}}\;.
\end{aligned}\right.
\end{equation}
Then for every $\eps$, there exist $\tau_\eps > 0$ and $R_\eps>0$ such that the equation \eqref{e:nonlinear_PDE} has a unique solution $u_\eps \in \yY_{\tau_\eps}^{g_\eps; R_\eps}$. Furthermore, the local existence time $\tau_\eps$ and the radius $R_\eps$ depend on $\widetilde{F}_\eps$, $\widetilde{G}_\eps$ and $g_\eps$ via the following quantity $\mM_\eps$ only:
\begin{equation} \label{e:norms_F_G_f}
    \mM_\eps:= |\!|\!|\widetilde F_\eps|\!|\!|_{\mathcal F,\mathrm{PDE}} +|\!|\!|\widetilde G_\eps|\!|\!|_{\mathcal G} + \|g_\eps\|_{\cC^{2+2\alpha}(\overline{\widetilde{\oO}})}\;.
\end{equation}
and one has the bound
\begin{equation} \label{eq:bound_nonlinear_PDE}
    \|u_\eps\|_{\cC_{\pa}^{2+2\alpha}( [0,\tau_\eps] \times \overline{\widetilde{\oO}})} \lesssim_{\mM_\eps} 1\;.
\end{equation}
If furthermore, $\mM_\eps$ is uniformly bounded in $\eps$, then there exists $\tau>0$ independent of $\eps \in [0,1]$ such that
\begin{equation*}
    \|u_\eps-u_0\|_{\cC_{\pa}^{2+2\alpha}([0,\tau]\times\overline{\widetilde{\oO}})} \lesssim |\!|\!|\widetilde F_\eps - \widetilde F_0|\!|\!|_{\mathcal F,\mathrm{PDE}} +|\!|\!|\widetilde G_\eps-\widetilde G_0|\!|\!|_{\mathcal G} + \|g_\eps-g_0\|_{\cC^{2+2\alpha}(\overline{\widetilde{\oO}})}\;.
\end{equation*}
\end{lem}
\begin{proof}
The existence and uniqueness of the solution for fixed $\eps$ is essentially the same as \cite[Theorem~8.5.4]{Lun95}. The stability is also implicit in its proof. On the other hand, our setup is slightly different from that of \cite[Theorem~8.5.4]{Lun95} in terms of the boundary condition. So for the sake of completeness, we give a sketched argument here. 

Since \(\widetilde\oO\) is fixed, in what follows we use the abbreviation
\[
    \cC_{\pa,T}^{2+2\alpha}:=
    \cC_{\pa}^{2+2\alpha}([0,T]\times\overline{\widetilde{\oO}})\;,
\]
and use the analogous notation with \(T\) replaced by \(\tau\) or
\(\tau_\eps\). Define the operators $\widetilde\fF_\eps^t$ for $t \geq 0$ and $\widetilde\gG_\eps$ by
\begin{equation*}
    (\widetilde\fF_\eps^t h)(x) := \widetilde F_\eps \big( t, x, h(x), \nabla h(x), \nabla^2 h(x) \big)\;, \quad (\widetilde\gG_\eps h)(x) := \widetilde G_\eps \big( x, h(x), \nabla h(x) \big)\;.
\end{equation*}
Define the linearised operators $\aA_\eps$ and $\bB_\eps$ by
\begin{equation*}
    \begin{split}
    \aA_\eps h &:= \sum_{i,j=1}^{d} \d_{q_{ij}} \widetilde F_\eps \cdot \nabla_{ij}^2 h + \sum_{j=1}^{d} \d_{p_j} \widetilde F_\eps \cdot \nabla_j h + \d_u \widetilde F_\eps \cdot h\;,\\
    \bB_\eps h &:= \sum_{j=1}^{d} \d_{p_j} \widetilde G_\eps \cdot \nabla_j h + \d_u \widetilde G_\eps \cdot h\;,
    \end{split}
\end{equation*}
where the above derivatives of $\widetilde F_\eps$ and $\widetilde G_\eps$ are evaluated at $\big(0, x, g_\eps(x), \nabla g_\eps(x), \nabla^2 g_\eps (x) \big)$ and $\big( x, g_\eps(x), \nabla g_\eps (x) \big)$ respectively. 

For $v \in \yY_{T}^{g_\eps; R_\eps}$, let $\Gamma_\eps(v)$ denote the solution $w_\eps$ to the linear problem
\begin{equation}
\left\{
\begin{aligned}
\d_t w_\eps &= \aA_\eps w_\eps + \widetilde\fF_\eps^t(v) - \aA_\eps v\;,
&& t\in [0,T]\;, x \in \overline{\widetilde{\oO}}\;,\\
\bB_\eps w_\eps &= - \widetilde\gG_\eps (v) + \bB_\eps v\;,
&& t \in [0,T]\;, x\in\partial\widetilde\oO\;,\\
w_\eps(0)& = g_\eps,\;&& x\in\overline{\widetilde{\oO}}\;.
\end{aligned}\right.
\end{equation}
Note that
\begin{equation*}
    \|u-v\|_{\lL_t^\infty\cC_x^2([0,T]\times\overline{\widetilde{\oO}})} \lesssim T^\alpha \|u-v\|_{\cC_{\pa,T}^{2+2\alpha}}
\end{equation*}
for $u,v \in \yY_{T}^{g_\eps; R_\eps}$. In particular, this implies that after decreasing $T$ if necessary, $(v,\nabla v,\nabla^2v)$ remains in $B((\bar u,\bar p,\bar q),R_0)$ for every $v\in\yY_T^{g_\eps;R_\eps}$. Also, the map $\Gamma_\eps$ satisfies the bounds
\begin{equation*}
    \| \Gamma_\eps (v) \|_{\cC_{\pa,T}^{2+2\alpha}} \leq C(\lambda, \mM_\eps, c) \cdot \big( T^\theta ( 1 + R_\eps )^2 + 1 \big)\;,
\end{equation*}
and
\begin{equation*}
    \|\Gamma_\eps (u) - \Gamma_\eps (v)\|_{\cC_{\pa,T}^{2+2\alpha}} \leq C'(\lambda, \mM_\eps, c, R_\eps) \, T^\theta \|u-v\|_{\cC_{\pa,T}^{2+2\alpha}}
\end{equation*}
for some $\theta > 0$ and $C$ and $C'$ depending on relevant quantities. Thus, one can choose $R_\eps$ sufficiently large and $T = \tau_\eps$ sufficiently small such that
\begin{equation*}
    C(\lambda, \mM_\eps, c) \cdot \big( T^\theta ( 1 + R_\eps )^2 + 1 \big) < R_\eps\;, \quad C'(\lambda, \mM_\eps, c, R_\eps) \cdot T^\theta < \frac{1}{2}\;,
\end{equation*}
and $\Gamma_\eps$ is thus a $(1/2)$-contraction from $\yY_{T}^{g_\eps; R_\eps}$ into itself. Hence, there is a unique fixed point $u_\eps \in \yY_{T}^{g_\eps; R_\eps}$ of $\Gamma_\eps$, which is the solution to \eqref{e:nonlinear_PDE} by definition. The choices above also show that both $R_\eps$ and $\tau_\eps$ depend on $\eps$ only through $\mM_\eps$. It also follows directly from the above argument that the bound \eqref{eq:bound_nonlinear_PDE} holds. If $\mM_\eps$ is uniformly bounded, then we can take $\tau>0$ and $R>0$ independent of $\eps$ such that all the above hold.

As for stability, it suffices to note further that
\begin{equation*}
    \begin{split}
    \big\| \Gamma_\eps (u_\eps) &- \Gamma_0 (u_0) \big\|_{\cC_{\pa,\tau}^{2+2\alpha}} \leq \big\| \Gamma_\eps (u_\eps) - \Gamma_0 (u_\eps) \big\|_{\cC_{\pa,\tau}^{2+2\alpha}} + \big\| \Gamma_0 (u_\eps) - \Gamma_0 (u_0) \big\|_{\cC_{\pa,\tau}^{2+2\alpha}}\\
    &\lesssim_R |\!|\!| \widetilde F_\eps - \widetilde F_0 |\!|\!|_{\fF,\mathrm{PDE}} + |\!|\!| \widetilde G_\eps - \widetilde G_0 |\!|\!|_{\gG} + \|g_\eps - g_0\|_{\cC^{2+2\alpha}} + \tau^\theta \|u_\eps - u_0\|_{\cC_{\pa,\tau}^{2+2\alpha}}\;.
    \end{split}
\end{equation*}
Using that $u_\eps$ is the fixed point of $\Gamma_\eps$ and that $\tau$ can be chosen sufficiently small, the claim then follows. 
\end{proof}

\begin{lem} \label{lem:PDE_further_regularity}
Fix $k\in\NN$. If $\widetilde{F}_\eps$ further satisfies that
\begin{equation*}
\sup_{\eps \in [0,1]} \sup_{\substack{(u,p,q)\in B((\bar u,\bar p,\bar q),R_0)\\|\beta|\leq k+1}} \bigl\| D_{(u,p,q)}^{\beta} \widetilde{F}(\cdot\,,\cdot\,,u,p,q)\bigr\|_{\cC^{\alpha,k+2\alpha}([0,T]\times \overline{\widetilde{\oO}})} <+\infty\;,
\end{equation*}
and $g_\eps \in \cC^{k+2+2\alpha}$ uniformly in $\eps$, then for every open $\oO \Subset \widetilde\oO$, we have
\begin{equation*}
    \|\nabla^k u_\eps\|_{\cC_\pa^{2+2\alpha}([0,\tau] \times \overline{\oO})}\lesssim 1\;,
\end{equation*}
where the proportionality constant depends on the above norms of $\widetilde{F}_\eps$, $\widetilde{G}_\eps$ and $g_\eps$ as well as the set $\oO$. 
\end{lem}
\begin{proof}
    The case $k=1$ is shown in \cite[Proposition~8.5.6]{Lun95}. The proof for arbitrary $k$ is similar. So we omit the details. 
\end{proof}

\subsection{Proof of Theorem~\ref{thm:stability}}

We are now ready to give the proof of the main well-posedness theorem for the general RPDE. 

\begin{proof}[Proof of Theorem~\ref{thm:stability}]
We only give details for the stability assuming the uniform boundedness of $F_\eps$ and $G_\eps$. The bounds for fixed $\eps$ are simpler so we omit it. 

Let $v_\eps$ be the solution to \eqref{e:RPDE_transformed} with $\widetilde{F}_\eps$ and $\widetilde{G}_\eps$ transformed from $F_\eps$ and $G_\eps$. We first assume that the transformed solutions $v_\eps$ satisfy the additional regularity hypotheses in Lemma~\ref{lem:transform}; this assumption will be verified at the end of the proof. Then by Lemma~\ref{lem:transform}, we have
\begin{equation*}
    \|\rho_\eps - \rho_0\|_{\lL_t^\infty \cC_x^2([0,\tau]\times\overline{\oO})} \lesssim_{\|\W^\eps\|, \, \|v_\eps\|_{\cC_{\pa,T}^{2+2\alpha}}} d_{\alpha}(\W^\eps, \W) + \|v_\eps - v_0\|_{\lL_t^\infty \cC_x^2([0,\tau]\times\overline{\widetilde{\oO}})} \;.
\end{equation*}

We hope to apply Lemma~\ref{lem:nonlinear_PDE} to say $\|v_\eps\|$ is uniformly bounded and to control $\|v_\eps - v_0\|$. But the boundary term in Lemma~\ref{lem:nonlinear_PDE} does not depend on $t$ while the boundary term for the solution to \eqref{e:RPDE_transformed} has the form
\begin{equation*}
    \widetilde{G}_\eps(t; x,u,p) = G_\eps \big( \Theta_t^\eps (\bm\theta) \big) \;,
\end{equation*}
which in general depends on $t$. To this end, we use the factorization condition \eqref{e:condition_GH_factorization} to say that these two boundary conditions actually coincide. To see this, let $\widehat{v}_\eps$ be the solution to \eqref{e:nonlinear_PDE} in Lemma~\ref{lem:nonlinear_PDE} with $\widetilde{F}_\eps$ transformed from $F_\eps$, but the boundary term being $G_\eps$ itself. This implies the solution $\widehat{v}_\eps$ satisfies the boundary condition
\begin{equation*}
    G_\eps \big( x, \widehat{v}_\eps(t,x), \nabla \widehat{v}_\eps(t,x) \big) = 0 \quad \text{on} \; \d \widetilde\oO\;.
\end{equation*}
Now, let $Z_\eps(t) := G_\eps(\Theta_t^\eps(\bm\theta))$. The rough chain rule and the characteristic system give 
\begin{equation*}
    \md Z_\eps(t)=\left[-D_xG_\eps H_p^\sft + D_u G_\eps (H-H_p \, P)^\sft +D_pG_\eps (H_x^\sft+PH_u^\sft )\right] (\Theta_t^\eps(\bm\theta)) \,\md\W_t\;.
\end{equation*}
By the factorization assumption \eqref{e:condition_GH_factorization}, we see $Z_\eps$ satisfies the linear equation
\begin{equation*}
    \md Z_\eps(t)=Z_\eps(t)L_\eps(\Theta_t^\eps(\bm\theta))\,\md\W_t\;, \quad Z_\eps(0) = G_\eps(\bm\theta)\;.
\end{equation*}
If $G_\eps(\bm\theta)=0$, then $Z_\eps(t)\equiv0$. Taking $\bm\theta = (x, \widehat{v}_\eps(t,x), \nabla \widehat{v}_\eps(t,x))$ gives
\begin{equation*}
    \widetilde{G}_\eps \big(t, \bm\theta \big) = G_\eps \big( \Theta_t^\eps (\bm\theta) \big) = 0\;.
\end{equation*}
This shows that $\widehat{v}_\eps$ also satisfies the boundary condition in \eqref{e:RPDE_transformed} given by $\widetilde{G}_\eps$. Hence by uniqueness, we have $v_\eps = \widehat{v}_\eps$. 

We next check ellipticity. We have the expression
\begin{equation}\label{e:expression_tildeF}
    \widetilde{F}_\eps(t, \bm\theta, q) = F_\eps \big( t, \Theta_{t}^\eps (\bm\theta), \; Q_t^\eps(\bm\theta; q) \big) \cdot \exp \Big( -\int_0^t (D_u H)^\sft\big(\Theta_s^\eps(\bm\theta)\big)\,\md\W_s^\eps \Big)\;.
\end{equation} 
Write $M_t^\eps := D_p P_t^\eps - Q_t^\eps D_p X_t^\eps$. Differentiating $\widetilde{F}_\eps$ in $q \in \RR^{d \times d}$, we get
\begin{equation*}
    \d_q \widetilde{F}_\eps =\exp \Big( -\int_0^t (D_u H)^\sft\big(\Theta_s^\eps(\bm\theta)\big)\,\md\W_s^\eps \Big) (M_t^\eps)^\sft \, \d_q F_\eps\big(t,\Theta_t^\eps(\bm\theta), Q_t^\eps\big) \, \big( D_x X_t^\eps + D_u X_t^\eps \, p^\sft + D_p X_t^\eps \, q \big)^{-\sft}\;.
\end{equation*}
Here \(\partial_qF_\eps\) and \(\partial_q\widetilde F_\eps\) are understood as matrices whose $(i,j)$-th entries are partial derivatives with respect to $q_{ij}$. 

Choose a bounded open set $\dD_3'\Subset\dD_3$ containing the range of $(\rho^0,\nabla\rho^0,\nabla^2\rho^0)$ on $\overline{\widetilde\oO}$. Since $\widetilde\oO\Subset\widehat\oO$ and $\overline{\dD_3'}\subset\dD_3$, the rough flow estimate Lemma~\ref{lem:RDE_stability} allows us, after decreasing $\tau$ if necessary, to ensure uniformly in $\eps$ that
\[
    X_t^\eps(\bm\theta)\in\widehat\oO\;,\qquad
    \big(U_t^\eps(\bm\theta),P_t^\eps(\bm\theta),Q_t^\eps(\bm\theta,q)\big)\in\dD_3
\]
for $t\in[0,\tau]$, $x\in\overline{\widetilde\oO}$, and $(u,p,q)\in\dD_3'$. Thus the ellipticity assumption \eqref{e:condition_F_elliptic_RPDE} applies to \(  \d_qF_\eps\big(t,\Theta_t^\eps(\bm\theta),Q_t^\eps(\bm\theta,q)\big)\). Moreover, by the rough flow estimate Lemma~\ref{lem:RDE_stability}, we also have
\begin{align*}
&\left|\exp\!\Big(-\int_0^t(D_uH)^\sft(\Theta_s^\eps(\bm\theta))\,\md\W_s^\eps\Big)-1\right|
+\|M_t^\eps-\id\|\\
&\qquad
+\left\|\big(D_xX_t^\eps+D_uX_t^\eps p^\sft+D_pX_t^\eps q\big)^{-\sft}-\id\right\|
\lesssim t^\theta
\end{align*}
uniformly on the above region. Consequently, after decreasing $\tau$ once more if necessary, we obtain
\[
    \eta^\sft \partial_q\widetilde F_\eps(t,\bm\theta,q)\eta
    \geq \frac{\lambda}{2}|\eta|^2
\]
uniformly for $t\in[0,\tau]$, $x\in\overline{\widetilde\oO}$, $(u,p,q)\in\dD_3'$, $\eta\in\RR^d$, and $\eps\in[0,1]$.

We can then apply Lemma~\ref{lem:nonlinear_PDE} to obtain $\|v_\eps\|_{\cC_{\pa, \tau}^{2+2\alpha}}$ is uniformly bounded and
\begin{equation*}
    \|v_\eps - v_0\|_{\cC_{\pa, \tau}^{2+2\alpha}} \lesssim |\!|\!| \widetilde{F}_\eps - \widetilde{F}_0 |\!|\!| + |\!|\!| G_\eps - G_0 |\!|\!|\;.
\end{equation*}
The proportionality constant depends on the uniform bounds of $|\!|\!| \widetilde{F}_\eps |\!|\!|$ as well as ellipticity. Moreover, by Lemma~\ref{lem:PDE_further_regularity} with $k=2$, for every open $\oO\Subset\widetilde\oO$, we have $\nabla v_\eps,\nabla^2 v_\eps\in\cC^{2+2\alpha}_\pa\bigl([0,\tau]\times \overline{\oO}\bigr)$, with these norms also uniformly bounded in $\eps$. These estimates give the regularity required in Lemma~\ref{lem:transform}, and hence justify the assumption made at the beginning of the proof.

It then remains to control $|\!|\!| \widetilde{F}_\eps - \widetilde{F}_0 |\!|\!|$ in terms of relevant norms of $F_\eps - F_0$. By the expression \eqref{e:expression_tildeF}, the derivatives of $\widetilde{F}_\eps$ with respect to $(\bm\theta, q) = (x,u,p,q)$ up to order three are linear combinations of products of derivatives of $F_\eps$ with respect to $(x,u,p,q)$ up to order three and derivatives of $\Theta_{t}^\eps$ with respect to $\bm\theta = (x,u,p)$ up to order four. Also, for $\alpha < \frac{1}{2}$, the $\alpha$-H\"older norm of $\widetilde{F}_\eps$ in $t$ is controlled by the $\alpha$-H\"older norm of $F_\eps$ together with the first order derivatives of $F_\eps$ in terms of other variables. This gives
\begin{equation*}
    |\!|\!| \widetilde{F}_\eps - \widetilde{F}_0 |\!|\!|_{\text{PDE}} \lesssim |\!|\!| F_\eps - F_0 |\!|\!|_{\text{RPDE}} + \sup_{|\beta| \leq 4} \|D^\beta (\Theta_{t}^\eps - \Theta_{t}^0)\|  \lesssim |\!|\!| F_\eps - F_0 |\!|\!|_{\text{RPDE}} + d_\alpha (\W^\eps, \W)\;,
\end{equation*}
where the second bound follows from Lemma~\ref{lem:RDE_stability}. This completes the proof. 
\end{proof}

\appendix

\section{Operations with rough paths}

We give some standard lemmas on operations with controlled rough paths. These are well known and we give statements without proofs. 

\begin{lem}\label{lem:composition_1}
    Suppose $Y \in \cC_{x,loc}^2$ and $Y' \in \cC_{x,loc}^1$, and
    \begin{equation*}
        (Y(t;x),Y'(t;x))\in\lL_{x,loc}^\infty\dD_W^{2\alpha}\;, \qquad (Z(t),Z'(t))\in\dD_W^{2\alpha}\;.
    \end{equation*}
    Then we have
    \begin{equation*}
        \left(Y(t;Z(t)),Y'(t;Z(t)) + \nabla^\sft Y(t;Z(t))Z'(t)\right)\in\dD_W^{2\alpha}\;.
    \end{equation*}
\end{lem}

\begin{lem}\label{lem:composition_2}
    Suppose $Y \in \cC_{x,loc}^3$, and there exist $A \in \cC_{x,loc}^2$, $A' \in \cC_{x,loc}^1$, $B \in \cC_t^0 \cC_{x,loc}^0$ such that
    \begin{equation*}
    \begin{aligned}
        Y(t;x) &= y_0(x)+\int_0^t (A(s;x),A'(s;x))\,\md \W_s +\int_0^t B(s;x)\,\md s\;,\\ Z(t) &= z_0+\int_0^t (C(s),C'(s))\,\md \W_s+\int_0^t D(s)\,\md s\;.
    \end{aligned}
    \end{equation*}
    Suppose further $\nabla^\sft Y$ is also controlled by $W$ with $(\nabla^\sft Y)' = \nabla^\sft A$. Then we have
    \begin{equation*}
        Y(t;Z(t)) = y_0(z_0)+ \int_0^t (E(s),E'(s))\,\md \W_s +\int_0^t F(s)\,\md s\;,
    \end{equation*}
    where
    \begin{equation*}
        \begin{split}
        E(s) &= A \big(s; Z(s)\big) + (\nabla^\sft Y)\big(s;Z(s)\big)C(s)\;, \qquad F(s) = B\big( s; Z(s) \big) + (\nabla^\sft Y)\big(s; Z(s)\big) D(s)\;,\\
        E'(s) &= A' +  2 \, \Sym \big( (\nabla^\sft A) \, C \big) + (\nabla^\sft Y) \, C' + \langle \nabla^2 Y, \; C, \, C \rangle\;.
        \end{split}
    \end{equation*}
    Here, the spacetime dependent quantities on the right hand side of the expression for $E'(s)$ are evaluated at $(s, Z(s))$. 
\end{lem}

\begin{lem}\label{lem:inverse_rough_path}
    Suppose $Y \in \cC_{x,loc}^3$, and there exists $(A, A') \in \dD_W^{2\alpha}$ with $A \in \cC_{x,loc}^2$, $A' \in \cC_{x,loc}^1$ and $B \in \cC_{x,loc}^1 \cap \cC_t^\gamma$ such that
    \begin{equation*}
        Y(t;x) - Y(s;x) = A(s;x) (W_t - W_s) + A'(s;x) \WW_{s,t} + B(s;x)(t-s) +\oO((t-s)^{1+\gamma})\;,
    \end{equation*}
    and that $Y(t;\cdot)$ is injective with $\sigma_{\min}(\nabla^
    \sft Y(t;x)) \geq c > 0$. Then $Z(t;\cdot):= Y(t;\cdot)^{-1}$ satisfies
    \begin{equation*}
        Z(t;x) -Z(s;x) = C(s;x) (W_t - W_s) + C'(s;x) \WW_{s,t} + D(s;x)(t-s) +\oO((t-s)^{1+\gamma})\;,
    \end{equation*}
    where 
    \begin{equation*}
        \begin{split}
        C(s) &= - (\nabla^\sft Y)^{-1} A\;, \qquad D(s) = -(\nabla^\sft Y)^{-1} B\;,\\
        C'(s) &= - (\nabla^\sft Y)^{-1} \bigg[ A' - 2 \, \Sym \Big( \big(\nabla^\sft A\big) \big( (\nabla^\sft Y)^{-1} A \big) \Big)\\
        &\phantom{11111}+ \langle \nabla^2 Y, \; (\nabla^\sft Y)^{-1} A, \, (\nabla^\sft Y)^{-1} A  \rangle \bigg]\;.
        \end{split}
    \end{equation*}
    The right-hand sides above are all evaluated at $s$ (and $(s; Z(s))$ for spacetime dependent quantities). 
\end{lem}

\section{Maximum principle with time-dependent domains}
\begin{lem}\label{lem:moving_domain_maximum_principle}
    Let $\widetilde\oO\subset\RR^d$ be a bounded domain and suppose $F:[0,T]\times\overline{\widetilde{\oO}}\to\RR^d$ is continuous and satisfies that $F(t,\cdot)$ is injective for all $t\in[0,T]$, as well as 
    \begin{equation*}
        \det(\nabla_x^\sft F(t,x)) \neq  0
    \end{equation*}
    for all $t\in(0,T)$ and $x\in\widetilde\oO$. Let $\oO_t:= F(t,\widetilde{\oO})$ and
    \begin{equation*}
        \sS:=\bigcup_{t\in(0,T)}\{t\}\times\oO_t\;.
    \end{equation*}
    Let $w\in\cC(\overline{\sS})\cap \cC^{1,2}(\sS)$ satisfy
    \begin{equation*}
        \d_t w(t,x) = \sum_{i,j=1}^d a_{ij}(t,x) \d_{ij} w(t,x) + \sum_{i=1}^d b_i(t,x)\d_i w(t,x) + c(t,x)w(t,x)\;,\quad t\in(0,T)\;, x\in\oO_t\;,
    \end{equation*}
    where \(a=(a_{ij})\) is positive definite in \(\sS\), and the coefficients \(a_{ij}\), \(b_i\), and \(c\) are bounded. Assume moreover that $w(t,x) = 0$ for $t=0, x\in \overline{\oO_0}$ and $t\in(0,T], x\in \partial \oO_t$. Then $w(t,x)=0$ for all $(t,x)\in\overline{\sS}$.
\end{lem}
\begin{proof}
    It suffices to show that $w\leq 0$. Without loss of generality, we assume that $c<0$ (or we can consider $e^{-Ct} w$ for a large $C$). Considering the maximum point of $w$ gives the desired result.
\end{proof}

\endappendix

\bibliographystyle{Martin}
\bibliography{Refs}
\end{document}